\documentclass[12pt]{amsart}

\usepackage{amsmath}
\usepackage[alphabetic]{amsrefs}
\usepackage{mathrsfs}
\usepackage{amssymb}
\usepackage{amsthm}
\usepackage{bbm}
\usepackage{bm}
\usepackage{dsfont}
\usepackage{geometry}
\usepackage{paralist}
\usepackage{multirow}
\usepackage{array}   % for \newcolumntype macro
\usepackage{subfiles}
\usepackage{mathtools}
\usepackage[normalem]{ulem}
\usepackage[cal=boondox]{mathalfa}
\usepackage{stackengine}
\usepackage{scalerel}
\usepackage[hidelinks]{hyperref}
\usepackage{booktabs}
\usepackage{multirow, tabularx, caption, cellspace}
\usepackage{setspace}
\usepackage{fancyhdr}
\usepackage{epigraph}
\usepackage{microtype}

\usepackage{graphicx,wrapfig}
\usepackage{xcolor}

\addparagraphcolumntypes{X}

\newcolumntype{L}{>{$}l<{$}}

\newtheorem{theorem}{Theorem}[section]
\newtheorem{lemma}{Lemma}[section]

\newtheorem{proposition}{Proposition}[section]
\newtheorem{remark}{Remark}[section]

\theoremstyle{definition}

\newtheorem{assumption}{Assumption}[section]

\newtheorem{question}{Question}[section]

\def\aa{\mathsf{a}}
\def\bb{\mathsf{b}}
\def\ggg{\mathsf{g}}
\def\hh{\mathsf{h}}
\def\JJ{\mathsf{J}}
\def\kk{\mathsf{k}}
\def\mm{\mathsf{m}}
\def\nn{\mathsf{n}}
\def\pp{\mathsf{p}}

\def\tt{\mathsf{t}}
\def\uu{\mathsf{u}}

\def\xx{\mathsf{x}}
\def\yy{\mathsf{y}}

\DeclareMathOperator{\disc}{disc}
\DeclareMathOperator{\fin}{fin}

\DeclareMathOperator{\Ind}{Ind}

\DeclareMathOperator{\Mat}{Mat}
\DeclareMathOperator{\Proj}{Proj}

\DeclareMathOperator{\Res}{Res}

\DeclareMathOperator{\St}{St}
\DeclareMathOperator{\Stab}{Stab}
\DeclareMathOperator{\Supp}{Supp}
\DeclareMathOperator{\sym}{sym}
\DeclareMathOperator{\Vol}{Vol}

\DeclareMathOperator{\GL}{GL}
\DeclareMathOperator{\GSp}{GSp}

\DeclareMathOperator{\PGL}{PGL}

\DeclareMathOperator{\SL}{SL}

\DeclareMathOperator{\Nk}{N_K}
\DeclareMathOperator{\Ns}{N_S}
\DeclareMathOperator{\Pk}{P_K}
\DeclareMathOperator{\Ps}{P_S}

\def\*{\times}
\def\1{\mathbbm{1}}

\def\A{\mathbb{A}}

\def\B{\mathcal{B}}

\def\back{\backslash}
\def\C{\mathbb{C}}

\def\fin{\text{fin}}
\def\F{\mathbb{F}}

\def\H{\mathscr H}

\def\I{\mathcal I}

\def\K{\mathcal K}

\def\mz{\mathfrak a}
\def\mf{\mathfrak m}
\def\ma{{\mathfrak a}^\perp}
\def\mq{\mathfrak q}

\def\nuk{\mu_2}
\def\nus{\mu_1}

\def\o{\mathfrak{o}}

\def\p{\mathfrak{p}}

\def\Q{\mathbb{Q}}
\def\R{\mathbb{R}}

\def\Ss{\mathbb S}

\def\W{\mathcal{W}}

\def\Y{\mathcal Y}
\def\Z{\mathbb{Z}}

\numberwithin{equation}{section}

\let\old\pagestyle
\renewcommand{\pagestyle}[1]{
\ifx #1 empty
\old{plain}
\fi
}

\newcommand{\trans}[1]{{}^\top#1}
\newcommand{\Lie}[1]{\mathfrak{#1}}
\newcommand{\Ad}{\text{Ad}}
\newcommand{\mat}[4]{{\setlength{\arraycolsep}{0.5mm}\left[
		\begin{smallmatrix}#1&#2\\#3&#4\end{smallmatrix}\right]}}

\title[Equidistribution of Satake parameters]{A weighted vertical Sato-Tate law for Maa{\ss} forms on $\GSp_4$}
\author{F\'{e}licien Comtat}


\usepackage{xpatch}
\makeatletter   
\xpatchcmd{\@tocline}
{\hfil\hbox to\@pnumwidth{\@tocpagenum{#7}}\par}
{\ifnum#1<0\hfill\else\dotfill\fi\hbox to\@pnumwidth{\@tocpagenum{#7}}\par}
{}{}
\makeatother    

 \makeatletter
  \def\l@section{\@tocline{1}{-7pt}{0pc}{5pc}{}}
 \def\l@subsection{\@tocline{2}{-7pt}{2pc}{5pc}{}}
\begin{document}
\begin{abstract} We prove a weighted Sato-Tate law for the Satake parameters of automorphic forms on $\GSp_4$ with respect to a fairly general congruence subgroup~$H$ whose level tends to infinity. When the level is squarefree we refine our result to the cuspidal spectrum. The ingredients are the $\GSp_4$ Kuznetsov formula and the explicit calculation of local integrals involved in the Whittaker coefficients of $\GSp_4$ Eisenstein series. We also discuss how the problem of bounding the continuous spectrum in the level aspect naturally leads to some combinatorial questions involving the double cosets in $P \back G / H$,
for each parabolic subgroup $P$.
   \end{abstract}
   \maketitle
\tableofcontents
\section{Introduction}
The local parameters of $L$-functions are rather mysterious numbers. Up to twist by $|\cdot|^{it}$, all $L$-functions of degree $1$ are coming from Dirichlet characters $\chi \mod N$, and in this case the local parameter 
at a prime $p \nmid N$ is given by $\alpha_p = \chi(p)$. In particular, $\alpha_p$ is a root of unity, and thus
\begin{enumerate}
    \item[(A)] $\alpha_p$ is an algebraic number,
    \item[(B)] $\alpha_p$ has modulus $1$.
\end{enumerate}
Besides questions about an individual $\alpha_p$, one can also ask questions about their statistical distribution. It is easy to see that
\begin{enumerate}
    \item[(C)] as $N$ tends to infinity the local parameters at $p$ of all degree $1$ $L$-functions of conductor less than $N$ equidistribute on the unit circle.
\end{enumerate} 
A deeper result is Dirichlet's theorem on primes in arithmetic progressions, that can be stated as
\begin{enumerate}
    \item[(D)] for fixed $\chi$, the numbers $\alpha_p=\chi(p)$ equidistribute among the possible values of $\chi$ as~$p$ varies. 
\end{enumerate}
On the other hand the local parameters of higher degree $L$-functions are more elusive. The requirement of being a $L$-function is quite rigid, so one should still expect them to be very special numbers. 
In particular it seems reasonable to ask 
whether analogous properties to (A)-(D) hold for higher degree $L$-functions.
This gives rise to questions that are widely open and in some cases very deep.
It is expected that the local parameters are only algebraic in very special situations ``coming from arithmetic". The analogue of property~(B) is the so-called Ramanujan-Petersson conjecture, and it is known for $L$-functions of
holomorphic cusp forms of degree $2$ by the groundbreaking work of Deligne and Serre. Aside for a handful of additional cases~\cite{BCGNT,Car,W}, 
the Ramanujan-Petersson conjecture is not known, but expected to hold in the generic case. In contrast to the degree $1$ case,
the local parameters of a fixed primitive $L$-function of degree larger than $2$ do not take value in a discrete set.
However, an analogue of property~(D) is still expected to hold in ``typical" cases,
 with the counting measure replaced by a suitable Sato-Tate measure. This is referred to as the ``horizontal" Sato-Tate conjecture, and is known in the case of $L$-functions attached to non-CM holomorphic newforms of degree $2$ and weight $k \ge 2$ by the deep work~\cite{BLGHT}.
Finally, question~(C), where one fixes a prime and varies the $L$-function can be generalised in more than one aspect. Namely, it is expected  that, just as in the degree $1$ case, every $L$-function is attached to some automorphic representation $\pi$.
While this is known in some cases using converse theorems, the general case is far from known. Thus for the analogue of property~(C), it is natural to switch gear and to focus on automorphic representations $\pi$ varying in some suitable family
with one (or several) parameters going to infinity. In this case the local parameters at unramified primes $p$ of the $L$-functions attached to $\pi$ are completely determined by the Satake parameters
of the local representation $\pi_p$. 
These are $p$-adic analogues of the eigenvalues of the Laplacian, and hence they can be studied with the help of methods stemming from harmonic analysis, and in particular (relative) trace formulas. This makes the corresponding ``vertical" distribution problem easier than the horizontal distribution problem.
Whether one uses the Arthur-Selberg trace formula or a relative 
trace formula results in variants of the vertical distribution problem, where each representation $\pi$ is counted respectively with weight $1$ or with a spectral weight arising from the specific shape of the relative trace formula being used.

The unweighted vertical equidistribution problem for $\GL_2$ has been addressed by Sarnak~\cite{Sar} and Bruggeman~\cite{Bru} for Maa{\ss} forms whose eigenvalue is allowed to tend to infinity, and by 
Serre~\cite{Ser} and Duke-Conrey-Farmer~\cite{CDF} for holomorphic forms whose weight or level (or both) are allowed to tend to infinity. In those cases the limiting distribution is the $p$-adic Plancherel measure 
$$d\mu_p(x)=\frac{p+1}{(p^{\frac12}-p^{-\frac12})^2-x^2} d\mu_{\text{ST}}(x),$$
where $d\mu_{\text{ST}}$ is the Sato-Tate measure. Observe that the $p$-adic Plancherel measure converges to the Sato-Tate measure as $p$ tends to infinity. On the other hand the weighted vertical equidistribution problem for $\GL_2$ in the level aspect
has been treated by Knightly and Li~\cites{L,KL1,KL} using the Petersson and Kuznetsov formulae, in this case the limiting distribution is the Sato-Tate measure, independently of the prime~$p$. The picture is similar in higher rank: in every known
case, the limiting measure for the unweighted vertical distribution problem is the $p$-adic Plancherel measure, that converges to the Sato-Tate measure when $p$ tends to infinity (see~\cites{Shin,ST,MT,KWY} and references therein); and 
the limiting measure for the weighted distribution problem  
is the Sato-Tate measure~\cites{KST,D,KL3,BBR,Z}.
Vertical distributions questions for~$\GSp_4$ was pioneered by~\cite{KST}, who treated the case of Siegel modular forms in the weight aspect. Their work was extended in the level aspect in~\cite{D}, and in the level aspect for Siegel modular forms of higher degree by~\cite{KL3}.
In this work, we consider the vertical weighted distribution problem for automorphic representations $\pi$ on $\GSp_4(\A_\Q)$ which are spherical at infinity, in the level aspect.
To the best of our knowledge this is not covered by the existing literature. 
For instance~\cite{ST} treats families of representations whose Archimedean component is a discrete series (and the equidistribution problems considered there are somewhat different). 
On the other hand the results of~\cite{FM} can be used to prove a Sato-Tate law for Maaß forms on fairly general groups \textit{in the eigenvalue aspect}.
If $\pi$ is an automorphic representation as above that is unramified at 
a prime $p$ we denote by $(\alpha(\pi_p),\beta(\pi_p))$ its
Satake parameters at~$p$.

\begin{theorem}[Theorem~\ref{Thm1}]
Let $K(N) \subset \GSp_4(\A_\fin)$ be a compact open subgroup
satisfying Assumption~\ref{GeneralAssumption}.
For instance $K(N)$ can be the Borel, Siegel, Klingen or
paramodular subgroup of level $N$.
Consider the weight $w_{K(N)}$ defined in~(\ref{weigth}).
Then for any prime~$p$ and for any continuous function $g$ on 
$\C^\times \times \C^\times$ that is invariant by the action of the Weyl group we have 
 $$\lim_{\substack{N \to \infty \\ p \nmid N}}
\frac{\int w_{K(N)}(\pi) g(\alpha(\pi_p),\beta(\pi_p))\,  d\pi}{\int w_N(\pi) \, d\pi} = \int g \, d\mu_{\text{ST}},$$
where $d\mu_{\text{ST}}$ is the Sato-Tate measure for $\GSp_4$
and $\int d\pi$ denotes the integral over the complete spectrum  of $L^2(\R_{>0}\GSp_4(\Q) \back\GSp_4(\A))$.
\end{theorem}
Moreover when the level is squarefree we are able to refine 
our result to the smaller set of Maa{\ss} cusp forms on $\GSp_4$.
Fix a diagonal matrix $\tt \in \GSp_4(\R)$ with positive diagonal entries and $h$ a ``suitable" Weyl group-invariant Paley-Wiener function on $\C \times \C$ (see~\S~\ref{sec3}).

\begin{theorem}[Theorem~\ref{Thm2}]
 Let $K(N) \subset \GSp_4(\A_\fin)$ be a compact open subgroup
satisfying Assumption~\ref{NewAssumption}.
For instance $K(N)$ can be the Borel, Siegel, Klingen or
paramodular subgroup of level $N$, with $N$ square-free.
Let $\B(N)$ be an orthonormal basis of Hecke-Maa{\ss} forms
on $\GSp_4(\A_\Q)$ for $K(N)$.
Then for any prime~$p$ and for any continuous function $g$ on 
$\C^\times \times \C^\times$ that is invariant by the action of the Weyl group we have 
\begin{equation*}
    \lim_{\substack{p \nmid N \to \infty \\ N \> \square-\text{free}}}
\frac{\sum_{\varphi \in \B(N)} h(\nu_\pi)|W(\varphi)(\tt)|^2 g(\alpha(\pi_p),\beta(\pi_p))\,  d\pi}
{\sum_{\varphi \in \B(N)} h(\nu_\varphi)|W(\varphi)(\tt)|^2 \, d\pi} = \int g \, d\mu_{\text{ST}},
\end{equation*}
where $W(\varphi)$ denotes the Whittaker coefficients of $\varphi$, $\pi$ is the automorphic representation generated by $\varphi$, and $\nu_\pi$ is the spectral parameter of $\pi_\infty$.
\end{theorem}
We should discuss the weight appearing in Theorems~\ref{Thm1} and~\ref{Thm2}.
\begin{remark}\label{abouttheweights}
   (1) The weight $w_{K(N)}$ always involves the Whittakher coefficients of elements of $\pi$,
and in particular vanishes by definition on the non-generic spectrum.\\
\phantom{re}(2) As proved in~\cite{Mythesis}*{Lemma~2.4.6}, the residual spectrum of $\GSp_4(\A_\Q)$
is non-generic. \\
\phantom{re}(3) Given a Hecke-Maaß forms $\varphi$, the Whittaker coefficient $W(\varphi)(\tt)$ is proportional to the Whittaker function
as computed by Ishii~\cite{Ishii}.
According to the Lapid-Mao conjecture (known in some cases by~\cite{ChenIchino}), when $\varphi$ is $L^2$-normalised, the (square of the) constant of proportionality should essentially be given by the inverse of the $L$-value $L(1,\pi,\Ad)$. On the other hand, not much is known on the Whittaker function
itself. The fact that the diagonal matrix $\tt$ is arbitrary may leave the possibility to produce a more amenable weight (for instance by integration in $\tt$), but we have not pursued this here. In our joint work with Lesesvre and Man~\cite{CLM} we return to a detailed analysis of the Whittaker function.
\end{remark}
Theorem~\ref{Thm1} is a fairly direct consequence of the Kuznetsov formula for $\GSp_4(\A_\Q)$.
To deduce Theorem~\ref{Thm2} from Theorem~\ref{Thm1}
we remove the contribution from the continuous spectrum
by showing that it is negligible compared to the 
cuspidal spectrum  -- which is enough in view of (2) in Remark~\ref{abouttheweights}.
Aside from the case of $\GL_2$~\cite{KL}, few previous works
have required bounding the continuous spectrum 
in the Kuznetsov formula in the level aspect.
In~\cite{AB} the authors obtain an upper bound for the whole spectrum, with spectral weights coming from the Kuznetsov formula.
They then proceed to remove the weights for the cuspidal spectrum, and to drop the continuous spectrum by positivity,
before adding back the unweighted continuous spectrum.
Thus in their case, only the unweighted continuous spectrum
has to be bounded above, and in particular some of the structural features coming from the Whittakher coefficients of Eisenstein series are not seen.
Our recent joint work with Valentin Blomer~\cite{BC}
is concerned with the subgroup $\Gamma_0(p)$ of $\GL_3$, for 
which a newform theory is available, as well as
an explicit formula for the values of the Whittaker coefficients of the
newform on the diagonal.
Thus we are able to run an orthonormalisation procedure and to work with the Whittaker coefficients of the Eisenstein newforms. 
In the case of $\GSp_4$, a satisfactory newform theory is
only known for the paramodular subgroup, but in this work we are interested in a more general setting. 
Thus we develop a new approach to bound the continuous spectrum in the Kuznetsov formula, that to our knowledge is novel.
In particular we uncover an interesting relation between the size of spectrum induced from a parabolic subgroup $P=NM$ with level $H \subset G(\hat{\Z})$ 
and  the combinatorics of the double cosets 
$P(\A_\fin) \back G(\A_\fin) / H$. 
More precisely let $X_P \subset G(\hat{\Z})$
be a set of representatives of the double quotient $P(\A_\fin) \back G(\A_\fin) / H$.
Then the size of the spectrum is related to the interplay between the index $[N(\hat{\Z}):N_\xx]$ of the stabilizer $N_\xx$ in $N(\hat{\Z})$ of
each coset $\xx H$ ($\xx \in X_P$) and the size of certain integrals supported on $N(\A_\fin) \cap \JJ P(\A_\fin) \xx H$, where $\JJ$
is the long Weyl element. For more details see Section~\ref{sec4}.
In turn these global questions can be reduced to purely local ones. 
In particular when the level is squarefree, the
double coset decomposition can be reduced to the Bruhat decomposition over $\F_p$.
In this case (assuming also $B(\hat{\Z}) \subset H$, where $B$ is the Borel subgroup) there are only finitely many local integrals, and we calculate each one of them explicitly -- which essentially amounts to calculating the Iwasawa decomposition of elements of
$\JJ N(\Q_p)$.
It would be interesting to remove the squarefreedness assumption.
This would require tackling the following two problems:
\begin{enumerate}
    \item give a satisfactory description of 
    the double quotient  $P(\A_\fin) \back G(\A_\fin) / H$, and of the corresponding indices $[N(\hat{\Z}):N_\xx]$,
    \item find a ``uniform" way of bounding the local integrals.
\end{enumerate}
We also hope these ideas will have applications to other situations.
Finally, we need to control the unramified and Archimedean places. To this end, we derive upper bounds for the spherical Whittaker function on $\GSp_4(\R)$ and lower bounds for certain $L$-functions on the line $\Re(s)=1$, that appear to be new.
\subsubsection*{Acknowledgments} I have benefited from helpful discussions with Siu Hang Man, 
Didier Lesesvre and Valentin Blomer on various aspects of the $L$-functions appearing on the denominator of the Whittaker coefficients of Eisenstein series.
This work was supported by ERC Advanced Grant  101054336, and Germany's Excellence Strategy grant EXC-2047/1 - 390685813.

\section{Notations and preliminaries}
If $R$ is any commutative ring we let 
$$\GSp_4(R)=\left\{ \ggg \in \GL_4(R) : \exists \mu(\ggg) \in R^\times: \trans{\ggg} \JJ \ggg =\mu(\ggg) \JJ \right\},$$
where $\JJ=\left[\begin{smallmatrix}
 &   &  1 &    \\
   &   &  & 1 \\
  -1&   &  & \\
  &  -1&  &
\end{smallmatrix}\right]$. The scalar $\mu(\ggg)$ is called the multiplier system. When the letter $G$ is not used for another algebraic group, we might denote $\GSp_4(R)$ just by $G(R)$.
Throughout, we let $\A=\A_\Q$ be the ring of ad\`eles of $\Q$.
If $H$ is an algebraic subgroup of $G$ defined over $\Q$, we denote its $\A$-points also by~$H$. If $R=\hat{\Z}$ or $R=\Z_p$ for some prime $p$ and $I$ is a non-trivial ideal of $R$, we may use the notation
$$H(I)=\{ \ggg \in G(R): \exists \hh \in H(R): \ggg \equiv \hh \mod I\}.$$
 We let $K=\prod_v K_v$ be a maximal compact subgroup of $G(\A)$ by defining $K_\infty=\{\ggg \in G(\R): \trans{\ggg}^{-1}=\ggg\} \simeq (\Z/2\Z) \times U(2)$
and $K_p=G(\Z_p)$ if $p$ is a prime. For brevity we may
use the notations $\o$ for $\Z_p$ and $\p$ for $p\Z_p$.
If $G$ is any reductive algebraic group defined over $\Q$
with centre $Z$,
by an automorphic representation of $G$ we shall always understand an irreducible representation $\pi$ occurring in the spectral decomposition of $L^2(G(\Q)Z(\R)^\circ \backslash G(\A)),$ where $Z(\R)^\circ$ is the connected component of the identity in~$Z(\R).$

\subsection{Parabolic subgroups and their induced representations}
We fix a minimal parabolic subgroup 
$B=\left[\begin{smallmatrix}
 *& *  &  * & *   \\
   &  * & * & * \\
  &   & * & \\
  &  & * &*
\end{smallmatrix}\right] \cap \GSp_4$. 
We let $A$ be the maximal torus contained in $B$ and $U$ be the subgroup of unipotent matrices in $B$, so that $B=AU$. 
The Weyl group $W$ consists of the eight elements $1,$ $s_1,$ $s_2,$ $s_1s_2,$ $s_2s_1,$ $s_1s_2s_1,$ $s_1s_2s_1,$ $s_1s_2s_1s_2=\JJ$
where 
$s_1=\left[\begin{smallmatrix}
 & 1  &   &    \\
 1  &   &  &  \\
  &   &  & 1\\
  &  & 1 &
\end{smallmatrix}\right]$ and 
$s_2=\left[\begin{smallmatrix}
 1&   &   &    \\
   &   &  & 1 \\
  &   & 1 & \\
  & -1 & &
\end{smallmatrix}\right]$.
Note the different convention 
from~\cite{Mythesis} amounts to conjugate by the Weyl group element $s_1$. In particular whenever we quote results from~\cite{Mythesis}, this conjugation by $s_1$ is taken into account if applicable.
The Bruhat decomposition states that given any field $F$ we have
$$G(F)= \coprod_{s \in W} \overline{U}_s(F) s B(F),$$
where $\overline{U}_s=U \cap s \trans{U} s^{-1}$. Moreover 
letting $U_s= U \cap s {U} s^{-1}$ we have $U=U_s\overline{U}_s=\overline{U}_sU_s$.
The other parabolic subgroups of $\GSp_4$ containing $B$ are the Siegel parabolic subgroup
$$\Ps=\left[\begin{smallmatrix}
 *& *  &  * & *   \\
  * &  * & * & * \\
  &   & * & *\\
  &  & * &*
\end{smallmatrix}\right] \cap \GSp_4 =\left[\begin{smallmatrix}
 1&   &  * & *   \\
   &  1 & * & * \\
  &   & 1 & \\
  &  &  &1
\end{smallmatrix}\right] 
\left[\begin{smallmatrix}
 *& *  & \phantom{1}  &    \\
  * &  * &\phantom{1}  &  \\
  \phantom{1}&   & * & *\\
 \phantom{1} &  & * &*
\end{smallmatrix}\right] \cap \GSp_4$$
and the Klingen parabolic subgroup
$$\Pk=\left[\begin{smallmatrix}
 *& *  &  * & *   \\
   &  * & * & * \\
  &   & * & \\
  &  *& * &*
\end{smallmatrix}\right] \cap \GSp_4=\left[\begin{smallmatrix}
 1& *  &  * & *   \\
   &  1 & * &  \\
  &   & 1 & \\
  &  & * &1
\end{smallmatrix}\right]\left[\begin{smallmatrix}
 *&   &   &    \\
\phantom{1}   & * &  & * \\
\phantom{1}  &   & * & \\
\phantom{1}  & * &   &*
\end{smallmatrix}\right] \cap \GSp_4 $$

\subsubsection{Power functions}
Let $P$ be a parabolic subgroup. 
Let $P=N_PM_P$ be its Levi decomposition.
Later we shall also use the notation $A_P$ for the centre of $M_P$.
We parametrize characters of $P$ 
that are trivial on the centre of $G$
as follows.
If $P=B$ and $\pp=\left[\begin{smallmatrix}
 a_1&   &   &    \\
   & a_2  &  &  \\
  &   & a_3 a_1^{-1} & \\
  &  & & a_3 a_2^{-1}
\end{smallmatrix}\right]\uu$ with $\uu \in U$ then 
we define for~$\nu=(\nu_1,\nu_2) \in \C^2$
$$I_{B,\nu}\left(\pp\right)=|a_1|^{\nu_1} |a_2|^{\nu_2} |a_3|^{-\frac{\nu_1+\nu_2}2}$$
and we set $\rho=\rho_B=(2,1)$.
If $P$ is the Siegel parabolic subgroup 
and $\pp=\nn \mm$ with $\nn \in N_P$ and $\mm=\mat{A}{}{}{t\trans{A}^{-1}}\in M_P$ we set $\mf_1(\pp)=t$ and $\mf_2(\pp)=A$. Then for $\nu \in \C$ we define
$$I_{P,\nu}(\pp)=|\det \mf_2(\pp)|^\nu|\mf_1(\pp)|^{-\nu}$$
and we set $\rho_P=\frac32$.
If $P$ is the Klingen parabolic subgroup
and $\pp=\nn \mm$ with $\nn \in N_P$ and $\mm=\left[\begin{smallmatrix}
 t&   &   &    \\
  &  a &  & b \\
  &   & t^{-1}(ad-bc)^{-1} & \\
  & c& & d
\end{smallmatrix}\right] \in M_P$
 we set $\mf_1(\pp)=t$ and $\mf_2(\pp)=\mat{a}{b}{c}{d}$. Then for $\nu \in \C$ we define
$$I_{P,\nu}(\pp)=|\mf_1(\pp)|^\nu |\det \mf_2(\pp)|^{-\frac{\nu}2}$$
and we set $\rho_P=2$.

Finally, using the Iwasawa decomposition we extend $I_{P,\nu}$ to $G$ by setting in each case $I_{P,\nu}(\pp \kk)=I_{P,\nu}(\pp)$ if $\pp \in P$ and $\kk \in K$.
This is well defined because for $\pp \in P \cap K$ we have $I_{P,\nu}(\pp)=1$.

\begin{lemma}\label{decomposepowers}
    Let $\nu=(\nu_1,\nu_2) \in \C$. Then we have $$I_{B,\nu}=I_{\Pk,\nu_1-\nu_2}I_{\Ps,\nu_2}.$$
    In particular $I_{B,\rho}=I_{\Pk,1}I_{\Ps,1}$.
\end{lemma}
\begin{proof} 
Let $\ggg \in G$. If $\ggg \in A$ then the claim is obvious.
Next, by the Iwasawa decomposition we have
$\ggg \in U \aa K$ for some $\aa \in A$, and $I_{B,\nu}(\ggg)=I_{B,\nu}(\aa)$.
So it suffices to show $I_{P,\mu}(\ggg)=I_{P,\mu}(\aa)$ for $P=\Ps,\Pk$ and $\mu \in \C$.
But we have $\ggg \in N_P (U \cap M_P) \aa K$, 
and by construction
$I_{P,\mu}$ is left-$N_P$ and right-$K$-invariant, and trivial on $(U \cap M_P)$.
\end{proof}

\subsubsection{Eisenstein series and induced representations}~\label{ESIR} Let $P$ be a parabolic subgroup and 
let $\sigma$ be a representation occurring in the discrete spectrum of $M_P$.
Let $\nu \in \C$ if $P=\Pk,\Ps$ and $\nu \in \C^2$ if
$P=B$. We define the induced representation
$$\I_P(\sigma_\nu)=\Ind_{P(\A)}^{G(\A)}(1_{N_P} \otimes I_{P,\rho_P+\nu} \otimes \sigma).$$
The space $\H_P$ of this representation consists by definition of functions $$\phi:N_P(\A)M_P(\Q) A_P(\R)^\circ \back G(\A) \to \C$$ such that
for all $\ggg \in G(\A)$ the function 
\begin{equation}\label{phix}
\begin{split}
      \phi_\ggg &: M(\Q)A_P(\R)^\circ \back M(\A) \to \C \\
      \mm & \mapsto \phi(\mm\ggg)
\end{split}
\end{equation}
belongs to $\sigma$, and
\begin{equation}\label{InnerProduct}
\|\phi\|^2 :=\int_K \int_{M(\Q) A_P(\R)^\circ \back M(\A)} |\phi(\mm\kk)|^2 \, d\mm d\kk < \infty.
\end{equation}
The action of~$\I_P(\sigma_\nu)$ on~$\H_P$ is given by
\begin{equation}\label{Action}
    (\I_P(\sigma_\nu)(\ggg) \phi)(\xx)=I_{P,\rho_P+\nu}(\xx\ggg) I_{P,-(\rho_P+\nu)}(\xx) \phi(\xx\ggg).
\end{equation}
The representation $\I_P(\sigma_\nu)$ occurs in the 
spectrum of $G$ via the Eisenstein series defined 
by (analytic continuation of)
$$E_P(\xx,\phi,\nu)=\sum_{\gamma \in P(\Q) \back G(\Q)} 
 \phi(\gamma \xx)I_{P,\rho_P+\nu}(\gamma \xx).$$
Indeed, the Eisenstein series have the intertwining property
$$E(\xx\ggg,\phi,\nu)=E(\xx,\I_P(\sigma_\nu)(\ggg) \phi, \nu)$$
for $\phi \in \H_P$.

\subsubsection{Spectral parameters} \label{IndReps}
For each parabolic subgroup $P$, we shall be interested in representations $\pi=\Ind_{P}^{\GSp_4} (1_{N_P} \otimes \sigma)$ of $\GSp_4(\R)$ which are induced from an irreducible representation $\sigma$ of $M_P$. Among them, we restrict attention to those~$\pi$ that have trivial central character and are $K_\infty$-invariant.
For $P=B$ we have $M_P \simeq \GL_1 \times \GL_1 \times \GL_1$ and thus $\sigma$ can only be a character $I_{B,{\rho_B+\nu}}$ for some $\nu \in \C^2$.
In this case we say $\pi$ is a principal series representation with spectral parameter~$\nu_\pi=\nu$.

We now turn to $P=\Pk,\Ps$.
In this case $M_P \simeq \GL_1 \times \GL_2$.
Thus $\sigma$ must be a twist $I_{P,\rho_P+\nu} \otimes (1_{\GL_1} \otimes \tau)$ for some principal series representation $\tau$ of $\PGL_2(\R)$.
Say~$\tau$ has spectral parameter $\nu_\tau$, by which we mean $\tau$ is induced by $\mat{a}{}{}{d} \mapsto \left|\frac{a}{d}\right|^{\frac12+\nu_\tau}$. Then embedding $\GL_2$ in $M_P$ via the map $\mf_2$ and inducing in stages, one sees that $\pi$ is itself a principal series representation, and its spectral parameter is given by
$$\nu_\pi=\begin{cases}
    (\nu,\nu)+(\nu_\tau,-\nu_\tau) &\text{ if } P=\Ps,\\
    (\nu,0)+(0,2\nu_\tau) &\text{ if } P=\Pk.
\end{cases}$$
To unify these cases we introduce the notation 
$$\mz_P(\nu)=\begin{cases}
   (\nu,\nu) &\text{ if } P=\Ps,\\
    (\nu,0) &\text{ if } P=\Pk
\end{cases}$$
and
$$\ma_P(\nu)=\begin{cases}
    (\nu,-\nu) &\text{ if } P=\Ps,\\
    (0,2\nu) &\text{ if } P=\Pk,
\end{cases}$$
so that $\nu_\pi=\mz_P(\nu)+\ma_P(\nu_\tau)$.

\subsection{Whittaker coefficients}\label{DiscussWC}
Let $G$ be a reductive group defined over $\Q$ and let~$U$ be a maximal connected unipotent
subgroup of $G$.
If $\psi$ is a generic character of $U$ and $\varphi$ is an automorphic form on $G(\A)$, the $\psi$-Whittaker coefficient of $\psi$ is by definition
$$W(\varphi, \psi,U)(\ggg)=\int_{U(\Q)\back U(\A)} 
\varphi(\uu\ggg)\overline{\psi}(\uu) \, d\uu.$$
When $U$ and $\psi$ are understood from the context, we shall omit them from the notation.
\subsubsection{Factorization of the Whittaker coefficients}
Let $\pi$ be an irreducible automorphic representation.
If the map $\varphi \in \pi \mapsto W(\varphi,\psi,U)$
is non-zero then we say $\pi$ is generic, and in this case the said map
is an isomorphism between $\pi$ and its global $(\psi,U)$-Whittaker model.

On the other hand the tensor product theorem asserts that $\pi$ is isomorphic to
a restricted tensor product of representations $\bigotimes_v \pi_v$,
with respect to a family of unramified vectors $(\varphi_v^0)_v$ defined at almost all place $v$.
For each place $v$ fix an isomorphism $\W_v(\cdot,\psi_v, U)$ 
between $\pi_v$ and its local $(\psi_v,U)$-Whittaker model, such that
for each prime $p$ such that $\pi_p$ is unramified, we have $\W_p(\varphi_p^0,\psi_p, U)(1)=1$.
If $\varphi \in \pi$ is a factorizable vector,
the product $\prod_v \W_v(\varphi_v,\psi_v,U)(\ggg_v)$
is well defined for all $\ggg \in G(\A)$,
and the map $\varphi=\otimes_v \varphi_v \mapsto \prod_v \W_v(\varphi_v,\psi_v,U)(\ggg_v)$ extends to a linear isomorphism  between $\pi$ and its global $(\psi,U)$-Whittaker model.

Thus, by uniqueness of the global Whittaker model there is a constant $c$ such that
\begin{equation}\label{factorW}
    W(\varphi, \psi,U)(\ggg)=c\prod_v\W_v(\varphi_v,\psi_v,U)(\ggg_v).
\end{equation}
Once we have fixed $U$ and $\psi$, the constant $c=c_\pi$ 
above depends only on $\pi$ and on our choices of isomorphism $\W_v$ at each (ramified) place.
For our purpose, the representation $\pi$ will always have 
a unique (up to scaling) $K_\infty$-fixed vector.
Correspondingly, there exists a unique distinguished $K_\infty$-invariant vector $\W_{\pi_\infty,\psi_\infty,U}$ in the Whittaker model of $\pi_\infty$.
Now if $\varphi \in \pi$ is a $K_\infty$-invariant factorizable vector, then in the representation of $\varphi$ as a pure tensor, we can always pick $\varphi_\infty$ such that $\W_\infty(\varphi_\infty,\psi_\infty,U)=\W_{\pi_\infty,\psi_\infty,U}$. For this choice of representation we have
\begin{equation}\label{defFC}
    W(\varphi, \psi,U)(\ggg)=\W_{\pi_\infty,\psi_\infty,U}(\ggg_\infty)
    A_{\varphi}(\psi,U)(\ggg_\fin),
\end{equation}
where 
\begin{equation}\label{factorFC}
    A_{\varphi}(\psi,U)(\ggg)= c \prod_{v< \infty}\W_v(\varphi_v,\psi_v,U)(\ggg_v).
\end{equation}
If $\varphi \in \pi$ is an arbitrary $K_\infty$-fixed vector then $\varphi$ is a linear combination of 
$K_\infty$-fixed factorizable vectors. Thus its Whittaker coefficient still has a factorisation of the form~(\ref{defFC}) for some $A_{\varphi}(\psi,U)$ 
-- which, in general, does not admit a further factorization as in~(\ref{factorFC}).  By definition $A_{\varphi}(\psi,U)(\ggg_\fin)$ is the Fourier coefficient of $\varphi$.

\subsubsection{Example: $\GL_2$} In the case of $G=\GL_2$, we let $U$ be the subgroup of upper triangular unipotent matrices and we let $\psi\left(\mat{1}{x}{}{1}\right)=\theta(x)$, where $\theta$ is the standard 
additive character of $\A/\Q$.
The assumption that $\pi$ is generic and $\pi_\infty$ has a $K_\infty$-fixed vector implies $\pi$ must be a principal series representation.
After twisting if necessary, we can assume $\pi_\infty=|\cdot|^{\nu}\boxplus|\cdot|^{-\nu}$
for some $\nu \in i\R \cup (-\tfrac12,\tfrac12)$. 
In this case we make the standard choice
\begin{equation}\label{GL2Whittaker}
    \W_{\pi_\infty}\left(\mat{y}{}{}{1}\right)=|y|^{\frac12}K_\nu(2\pi |y|),
\end{equation}
where $K_\nu$ is the Bessel $K$ function of index $\nu$.
Next, assume that $\pi$ is cuspidal and (for simplicity) that it has trivial central character.
Consider a newform $\varphi \in \pi$ such that $\langle \varphi, \varphi \rangle =1$. Then $\varphi$ is a factorizable vector $\bigotimes_v \varphi_v$. We take $\W_\infty(\varphi_\infty)=\W_{\pi_{\infty}},$
and for each prime $p$ we normalise $\W_{\pi_p}:=\W_p(\varphi_p)$ such that
$\W_{\pi_p}(1)=1$.
This determines our choices of local isomorphisms $\W_v$ at every place $v$. 
Classically, $\varphi$ corresponds to a Hecke-Maa{\ss} form $u$ on 
$\mathbb{H}$ with Fourier expansion
$$u(x+iy)=y^{\frac12}\sum_{n \neq 0} a_u(n)K_{\nu}(2\pi|n|y)e(nx).$$
The correspondence between the classical Fourier coefficients and the Whittaker coefficients is given by
$$W_\varphi\left(\mat{ny}{}{}{1}\right)=a_u(n)y^{\frac12}K_{\nu}(2\pi|n|y),$$
for $y \in \R_{>0}$ and $n \in \Z \setminus\{0\}$ embedded diagonally in $\A_\Q$; from which we deduce 
$$|n|^{\frac12}A_\varphi\left(\mat{n}{}{}{1}\right)=a_u(n).$$
More generally given $\sigma \in \SL_2(\Z)$, the function $u$
has a Fourier expansion at the cusp $\mathfrak{a}=\sigma^{-1} \infty$ given by\footnote{our normalisation is such that the 
Kuznetsov formula at the cusp $\mathfrak a$ involves precisely the coefficients~$a_u(n,\mathfrak a)$.}
$$u(\sigma^{-1} z)=\left(\frac{y}{w(\mathfrak{a})}\right)^{\frac12}\sum_{n \neq 0} a_u(n,\mathfrak a)K_{\nu}\left(\frac{2\pi|n|y}{w(\mathfrak{a})}\right)e\left(\frac{nx}{w(\mathfrak{a})}\right).$$
where $w(\mathfrak{a})$ is the width of $\mathfrak{a}$.
Then we have 
\begin{equation}\label{NormalisationFourierCoefficients}
    |n|^{\frac12}A_\varphi\left(\mat{n/w(\mathfrak{a})}{}{}{1} \sigma\right)=a_u(n,\mathfrak a),
\end{equation}
For more details we refer
the reader to~\cite{CS}*{Proposition~3.2} or~\cite{Iw} for instance.
Evaluating~$(\ref{factorFC})$
at $\ggg=1$, the constant $c$ is given by $c=A_\varphi(1)$. In turn, by Rankin-Selberg theory we have
$$R(\pi) A_\varphi(1)^2=\frac{3}{\pi}\frac1{\Lambda(1,\pi,\Ad)},$$
where $$R(\pi) =\prod_{p \text{ ramified}} \frac{\zeta_p(2)}{\zeta_p(1)L_p(1,\pi,\Ad)} \int_{\Q_p^\times} |\W_{\pi_p}\left(\mat{y}{}{}{1}\right)|^2 d^\times y.$$

\subsubsection{Whittaker coefficients of Eisenstein series} If $\varphi=E_P(\cdot,\phi,\nu)$ is an Eisenstein series from a parabolic subgroup $P=NM$ of $\GSp_4$, then using the Bruhat decomposition and unfolding, one easily obtains for $\Re(\nu)$ large enough 
\begin{equation}\label{WCES}
    W(\varphi, \psi,U)(\ggg)=\int_{N(\A)}I_{P,\rho_P+\nu}(\sigma_P \nn \ggg) W(\phi_{\sigma_P\nn\ggg},\psi^P, U \cap M)(1) \overline{\psi}(\nn)\, d\nn
\end{equation}
where $\sigma_P$ is the element of $(W \cap M) \JJ$ such that 
$\sigma_P (U \cap M) \sigma_P^{-1}=U \cap M$, $\psi^P$ is the generic character of $U \cap M$ defined by
$\psi^P(\uu)=\overline{\psi}(\sigma_P^{-1}\uu\sigma_P)$, and 
$\phi_\ggg(\mm)=\phi(\mm\ggg)$.
Assume now that $\phi$ is such that  $\phi_\ggg$ is a factorizable vector in a fixed automorphic representation $\sigma$
of $M$, for all
$\ggg \in \GSp_4(\A)$.
Then using the factorisation~(\ref{factorW}) for $\phi$, $W(\varphi, \psi,U)$ factors as
a product of local integrals as follows
$$W(\varphi, \psi,U)(\ggg)=c_{\sigma}
\prod_v J_v(\phi,\nu,\psi_v,U,P)(\ggg_v)$$
where 
\begin{equation}\label{TheLocalIntegral}
    J_v(\phi,\nu,\psi_v,U,P)(\ggg_v)=\int_{N(\Q_v)}I_{P,\rho_P+\nu}(\sigma_P \nn \ggg_v) \W_v(\phi_{\sigma_P\nn\ggg_v},\psi_v^P, U \cap M)(1) \overline{\psi_v}(\nn)\, d\nn.
\end{equation}
Then again, when $U$, $P$ and $\psi$ are understood from the context
we shall omit them from the notation. When $\ggg_v=1$ we shall also omit it from the notation. In this case, by~\cite{Shahidi}*{Theorem~7.1.2} the local integral $J_v(\phi,\nu,\psi_v,U,P)$ is given at each unramified place by a certain $L$-factor.

\subsection{Miscellaneous lemmas} We record a few simple lemmas to be used in the sequel.
\subsubsection{Group-theoretic lemmas}
\begin{lemma}\label{grouptheoretic}
Let $P=NM$ be a group where $N$ is a normal subgroup and
$M$ is a subgroup such that $N \cap M=\{1\}$.
Assume $P$ acts on some set $X$, with finite orbits.
For all $x \in X$ define the following subgroups
$$P_x=\Stab_P(x),$$
$$N_x=\Stab_N(x),$$
$$M(x)=\Proj_P^M(P_x)$$
Then we have
$$[P:P_x]=[N:N_x] [M:M(x)].$$
\end{lemma}
\begin{proof}
Consider the action of $P/N$ on $N \back X$ given by
$$(pN).(Nx):=pNx=Npx.$$
We claim that $\#(P/N).(Nx)=[M:M(x)]$.
Indeed, since $P/N \simeq M$, we may view the above action as an action of $M$ on $N \back X$. We want to prove that the stabiliser of $Nx$ for this action is precisely $M(x)$.
So assume $mNx=Nmx=Nx$. Then there exists $n \in N$ such that $mnx=x$, which exactly means that $m \in M(x)$.
Next, we have 
$$\#(Px)=\sum_{Ny \in (P/N)\cdot(Nx)}\#(Ny)=[M:M(x)]\#(Nx)$$
since $\#(Ny)=\#(Nx)$ for all $y \in Px$.
\end{proof}

\begin{lemma}\label{Increasing}
    Let $K$ be a totally disconnected compact group.
    Let $H \subset K$ be an open subgroup of $K$
    and let $U \subset K$ be a closed subgroup of $K$.
    Then 
    $$\frac{\Vol_U(U \cap H)}{\Vol_K(H)} \ge 
    \frac{\Vol_U( U)}{\Vol_K(K)}.$$
\end{lemma}
\begin{proof}
We want to prove 
$$[U:U \cap H]=\frac{\Vol_U(U)}{\Vol_U(U\cap H)}\le \frac{\Vol(K)}{\Vol(H)}=[K:H].$$
We start with the case when $K$ is a finite group.
In this case we have
\begin{equation*}
    \begin{split}
        [K: U \cap H]&=[K:U][U:U \cap H]\\
        &=[K:H][H:U \cap H],
    \end{split}
\end{equation*}
    hence it suffices to show $$\frac{[H:U \cap H]}{ [K: U]} \le 1.$$
    But 
    $[K:U \cap H] \le [K:U][K:H]$
    and hence 
    $$\frac{[H:U \cap H]}{ [K: U]} \le \frac{[K:H][H:U \cap H]}{[K:U \cap H]}=1.$$
Next we consider the general case.
By~\cite{Bump}*{Proposition~4.2.1} there exists an open subgroup $I \subset H$ such that $I$ is normal in $K$. Consider the quotient map $\varphi: K \to K / I$.
Then $\varphi$ induces bijections between the quotients $K/ H$ and $\varphi(K)/\varphi(H)$, and 
$U / U \cap H$ and 
$\varphi(U) / \varphi(U \cap H)$,
respectively, and so we are done by previous case.
\end{proof}

\begin{lemma}\label{CongruenceSbgp}
    Let $p$ be a prime number and
    let $H$ be a subgroup of $\GL_4(\Z_p)$ such that for
    all $\ggg=\mat{A}{B}{C}{D}$ in $H$ we have $C_{11} \in \p$ and $C_{12}C_{21} \in \p$.
    Then 
    $$\text{eiter }
    H \subset \left[\begin{smallmatrix}
 * &  *  &  *&   * \\
 \p & *  & * & * \\
 \p & *  & * & * \\
 \p & *  & * & * \\
\end{smallmatrix}\right] \text{ or }
H \subset \left[\begin{smallmatrix}
 * &  *  &  *&   * \\
* & *  & * & * \\
 \p & \p  & * & \p \\
 * & *  & * & * \\
\end{smallmatrix}\right]
\text{ or }
H \subset \left[\begin{smallmatrix}
 * &  *  &  *&   * \\
* & *  & * & * \\
 \p & \p  & * & * \\
 \p & \p  & * & * \\
\end{smallmatrix}\right].$$
\end{lemma}
\begin{proof}
    We claim that either $\ggg_{32} \in \p$ for all  $\ggg$ in $H$ or $\ggg_{41} \in \p$ for all  $\ggg$ in $H$.
    Let us proceed by contradiction.
    Let $\ggg \in H$ such that 
    $\ggg_{41} \not \in \p$ (and so $\ggg_{32} \in \p$)
    Then $$(\ggg^2)_{31} \in \ggg_{34}\ggg_{41}+\p,$$
    which forces $\ggg_{34} \in \p$ (and hence also $\ggg_{33} \not \in \p)$.
    Similarly if $\hh \in H$ is such that 
    $\hh_{32} \not \in \p$ then we find that $\hh_{21} \in \p$ and $\hh_{11} \not \in \p$.
    But then 
    $$(\ggg \hh)_{32} \in \ggg_{33}\hh_{32}+\p \text{ and } (\ggg \hh)_{41} \in \ggg_{41}\hh_{11}+\p,$$
    and so we see that $\ggg \hh \not \in H,$ which is a contradiction.
    This establishes the claim.
    Now suppose $\ggg_{41} \in \p$ for all  $\ggg$ in $H$.
    If moreover $\ggg_{21} \in \p$ for all $\ggg \in H$ then we are done.
    Otherwise let $\ggg, \hh \in H$ and suppose $\hh_{21} \not \in \p$.
    Then we have 
    $$(\ggg \hh)_{31} \in \ggg_{32}\hh_{21}+\p \text{ and } (\ggg \hh)_{41} \in \ggg_{42}\hh_{21}+\p,$$
    from which follows $\ggg_{32},\ggg_{42} \in \p$.
    The case $\ggg_{32} \in \p$ for all  $\ggg$ in $H$ is similar.
\end{proof}

\begin{lemma}\label{ParamodularIwasawa}
    Let $\K_p=\GSp_4(\Q_p) \cap \left[\begin{smallmatrix}
 \o &  \o  &  \p^{-1} &   \o \\
\p  & \o  & \o & \o \\
 \p  &  \p  & \o  & \p\\
\p  & \o  & \o & \o 
\end{smallmatrix}\right]$ be the first paramodular subgroup and let $\mathcal{N}(\K_p)$ be its normaliser in
$\GSp_4(\Q_p)$. Then we have $$\GSp(\Q_p)=B(\Q_p)\mathcal{N}(\K_p).$$
\end{lemma}
\begin{proof}
    By the Iwasawa decomposition we have 
    $\GSp_4(\Q_p)=B(\Q_p)\GSp_4(\Z_p)$.
    By the Bruhat decomposition over $\F_p$ we have 
    $\GSp_4(\Z_p)=B(\Z_p)WU(\Z_p)I(p),$
    where $I(p)=\{\ggg \in \GSp_4(\Z_p): \ggg \equiv 1 \mod p\}$.
    Now $U(\Z_p)I(p) \subset \K_p$, and $W$ is generated by $s_1$ and $s_2$. But $s_2 \in  \K_p,$ and $\left[\begin{smallmatrix}
 1 &  &  &   \\
 & p  &  &  \\
  &    & p& \\
  &  &  & 1
\end{smallmatrix}\right]s_1 \in \mathcal{N}(\K_p)$, hence 
    $A(\Q_p)W \subset A(\Q_p)\mathcal{N}(\K_p).$
\end{proof}

\subsubsection{An integration formula} Let $dx$ be the Haar measure on $\Q_p$,
normalised so that $\Vol(\Z_p)=1$. Let $|\cdot|$ be the absolute value
on $\Q_p$ normalised so that $|p|=p^{-1}$.
Let $\theta_p$ be standard additive character of $\Q_p/\Z_p$.

% If $\Re(s)>0$ then we have
% \begin{equation}\label{zeta}
% \int_{v_p(x) \le \alpha} |x|^{-s}\frac{dx}{|x|}=p^{\alpha s}\frac{1-p^{-1}}{1-p^{-s}}.
%\end{equation}

% \begin{lemma}\label{ChofVar}
% Assume $p$ is odd.
% Let $f: \Q_p \to \C$ be a measurable function satisfying 
% $f(\epsilon x)=f(x)$ for all $x \in \Q_p$, where $\epsilon \in \Z_p^\times$ is a fixed 
% non-quadratic residue. Assume $f(x)/|x|$ is integrable.
% Then we have 
% $$\int_{\Q_p} f(x^2)\frac{dx}{|x|}=
% \int_{v_p(x) \text{ even}} f(x)\frac{dx}{|x|}.$$
% \end{lemma}
% \begin{proof}
% Let $$R=\left\{x \in \Q_p^\times : x|x| \mod p \in \{1,\cdots,\frac{p-1}2\}\right\}.$$
% Then we have $\Q_p^\times = R \sqcup (-R)$, and the map $x \mapsto x^2$
% is injective on $R$.
% Therefore, changing variables we have
% \begin{align*}
%    \int_{\Q_p} f(x^2)\frac{dx}{|x|} &= 2\int_R f(x^2)\frac{dx}{|x|}\\
%    &= 2 \int_{\Q_p} \1_{y=\square} f(y)\frac{dy}{|y|}\\
%    &= \quad \int_{\Q_p} \1_{y=\square} f(y)\frac{dy}{|y|}
%    + \int_{\Q_p} \1_{\epsilon y=\square} f(\epsilon y)\frac{dy}{|y|}\\
%    &=\quad \int_{v_p(x) \text{ even}} f(x)\frac{dx}{|x|}.
% \end{align*}
% \end{proof}

\begin{lemma}\label{GaussTransform}
    Let $(R,\mu)$ be a $\sigma$-finite measure space and let $f:R \to \C$ be an integrable function and $g_1, \cdots, g_m: R \to \Q_p$.
    Suppose there is a measurable action of $(\Z_p^\times)^m$ on $R$
    that leaves $fd\mu$ invariant, and such that 
    for all $\lambda=(\lambda_1,\cdots,\lambda_m) \in (\Z_p^\times)^m$ and for all $x \in R$ we have 
    $$g_j(\lambda \cdot r)=\lambda_j g_j(r)$$
    for $1 \le j \le m$.
    Then
    $$\int_R f(x) \theta\left(\sum_{j=1}^m g_j(x)\right) \, d\mu(x)
    =\int_{R'} f(x) \theta\left(\sum_{j=1}^m g_j(x)\right) \, d\mu(x),$$
    where 
    $$R'=\{x \in R: |g_j(x)| \le p \text{ for } 1 \le j \le m\}.$$
\end{lemma}
\begin{proof}
    Changing variables, we have for all 
    $\lambda=(\lambda_1,\cdots,\lambda_m) \in (\Z_p^\times)^m$ 
    $$\int_R f(x) \theta\left(\sum_{j=1}^m g_j(x)\right) \, d\mu(x)=
    \int_R f(x) \theta\left(\sum_{j=1}^m \lambda_j g_j(x)\right) \, d\mu(x).$$
    Integrating this expression over $(\Z_p^\times)^m$ and applying Fubini theorem gives
    \begin{align*}
    \int_R f(x) \theta\left(\sum_{j=1}^m g_j(x)\right) \, d\mu(x)&=
         \int_{(\Z_p^\times)^m} 
         \int_R f(x) \theta\left(\sum_{j=1}^m \lambda_j g_j(x)\right) \, d\mu(x)\\
         &=
   \int_R f(x) \theta\left(\sum_{j=1}^m \lambda_j g_j(x)\right) \, d\mu(x)\, d\lambda\\
    &= \int_R f(x) \prod_{j=1}^mG(g_j(x)) d\mu(x),
    \end{align*}
    where
     $$G(a):=\int_{\Z_p^\times}\theta(\lambda a)\, d^\times\lambda=
     \begin{cases}
         0 \text{ if } |a| > p\\
         -p^{-1}\zeta_p(1) \text{ if } |a|=p\\
         1 \text{ if } |a| \le 1.
     \end{cases}$$
   In particular, 
  $$ \int_R f(x) \theta\left(\sum_{j=1}^m g_j(x)\right) \, d\mu(x)
  = \int_{R'} f(x) \prod_{j=1}^mG(g_j(x)) d\mu(x).$$
  Now observe the same reasoning with $f$ replaced by $f \1_{R'}$ gives the same result on the right hand side.
\end{proof}

\subsubsection{Support of the Whittaker functions} Let $\psi_p\left(\mat{1}{x}{}{1}\right)=\theta_p(x)$, where $\theta_p$ is as in previous subsection.

\begin{lemma}\label{supportWhittaker}
    Let $\pi$ be an irreducible admissible representation of $\GL_2(\Q_p)$.
    Let $\W_p$ be an element of the $\psi_p$-Whittaker model of $\pi$ that is right-invariant by the subgroup
    $\mat{\o^\times}{\o}{\p^n}{\o^\times}$
    for some $n>0$.
    Then for all $y_1,y_2 \in \Q_p^\times$ we have
    $$\W_p\left(\mat{y_1}{}{}{y_2}\right)=0
    \text{ unless } v_p(y_2) \le v_p(y_1), \text{ and}$$
    $$\W_p\left(\mat{y_1}{}{}{y_2}\mat{}{1}{1}{}\right)=0
    \text{ unless } v_p(y_2) \le n+v_p(y_1).$$
\end{lemma}

\begin{proof}
Assume $v_p(y_2)>v_p(y_1)$. Then there exists $x \in \frac{y_1}{y_2}\o$ such that $\psi(x) \neq 1$.
By right invariance we have
$$\psi(x)\W_p\left(\mat{y_1}{}{}{y_2}\right)=
\W_p\left(\mat{1}{x}{}{1}\mat{y_1}{}{}{y_2}\right)
=\W_p\left(\mat{y_1}{}{}{y_2}\mat{1}{xy_2y_1^{-1}}{}{1}\right)=\W_p\left(\mat{y_1}{}{}{y_2}\right).$$
Thus $\W_p\left(\mat{y_1}{}{}{y_2}\right)=0$.
Next assume $v_p(y_2) > n+v_p(y_1)$.
Then there exists $x \in \frac{y_1}{y_2}\p^n$ such that $\psi(x) \neq 1$. Then we have
$$\psi(x)\W_p\left(\mat{y_1}{}{}{y_2}\mat{}{1}{1}{}\right)
=\W_p\left(\mat{y_1}{}{}{y_2}\mat{}{1}{1}{}\mat{1}{}{xy_2y_1^{-1}}{1}\right)=\W_p\left(\mat{y_1}{}{}{y_2}\right),$$
and thus $\W_p\left(\mat{y_1}{}{}{y_2}\right)=0$.
\end{proof}

\section{Equidistribution over the whole spectrum}\label{sec3}

This section is mainly taken from the author's PhD thesis~\cite{Mythesis}, with the minor modification that there we were only considering the Borel congruence subgroup. In the present work
 we consider for each positive integer $N$ a compact open subgroup $K_{\fin}(N)$ 
  subject to the following assumptions.
\begin{assumption}\label{GeneralAssumption}
$K_{\fin}(N)=\prod_p K_p(N) \subset G(\A_\fin)$ is an open compact subgroup such that
\begin{enumerate}
    \item $\mu: K_\fin(N) \to \hat{\Z}^\times$ is surjective,
    \item for all $\mat{A}{B}{C}{D} \in K(N)$ we have 
    $\det C \in N \hat{\Z}$ and $C_{11} \in N \hat{\Z}$.
\item for all prime $p$ we have $K_p(N) \cap B(\Q_p) \subset \left[\begin{smallmatrix}
 \Z_p^\times  &  \Z_p  &  \Q_p &   \Q_p \\
              & \Z_p^\times  & \Q_p & \Z_p \\
              &              & \Z_p^\times & \\
   &   &  \Z_p & \Z_p^\times
\end{smallmatrix}\right]$.
\end{enumerate}
\end{assumption}
\begin{remark}\label{conj}
    Note that Assumption~\ref{GeneralAssumption}
    is invariant by conjugation by elements of~$U(\A_\fin)$.
\end{remark}
We put $K(N)=K_\infty K_\fin(N)$.
Let $\Pi$ be the set of unitary automorphic representations of $\GSp_4(\A)$.
We recall thet $\Pi=\Pi_{\text{disc}} \sqcup \Pi_{\text{K}} \sqcup \Pi_{\text{S}} \sqcup \Pi_{\text{B}}$,
where $\Pi_{\text{disc}}$ consists of cuspidal and (non generic) residual representations $\pi$,
and for each parabolic subgroup 
$\Pi_{P}$ consists of induced representations
representations $\pi = \I_{P}(\sigma_\nu)$ 
where $\sigma$ occurs in the discrete spectrum of $M_P$, $\nu \in i\R$ for $P=\Pk,\Ps$ and $\nu \in i\R \times i\R$ for $P=B$. 
We endow $\Pi$ with a measure $d\pi$ by setting $d\pi$
to be the counting measure on $\Pi_{\disc}$ and
$d\pi=d\nu d\sigma$ if $\pi = \I_{\Pk}(\sigma_\nu)$,
where $d\nu$ is the Haar measure and $d\sigma$ is the counting measure on the discrete spectrum of $M_P$.

By the tensor product theorem, each representation $\pi \in \Pi$ factors as a restricted tensor product 
$\pi \simeq \bigotimes_v \pi_v$. We let $\nu_\pi$ be the spectral parameter of the representation $\pi_\infty$ as defined in~\S~\ref{IndReps}
We fix a matrix~$\tt \in A(\R)^\circ$ and we let $h$ be a Paley-Wiener function on $\C^2$ such that 
\begin{enumerate}
    \item $h(\nu)$ is invariant by  $(\nu_1,\nu_2) \mapsto (\nu_1,-\nu_2)$ and $(\nu_1,\nu_2) \mapsto (\nu_2, \nu_1)$,
    \item $h(\nu_\pi) \ge 0$ for all spectral parameters $\nu_\pi$ of a unitary automorphic representation of $\GSp_4(\A)$,
    \item the subset of $\Pi$ consisting of generic representations $\pi$ that have a non-zero $K(N)$-fixed vector $\varphi$ such that $h(\nu_\pi)|W(\varphi)(\tt)|^2>0$ has positive measure for~$N$ large enough.
\end{enumerate} 
The existence of such a function $h$ was proved in~\cite{Mythesis}*{Corollary~3.5.1}.
Constructing a $W$-invariant Paley-Wiener function that is non-negative on the automorphic spectrum is a standard trick: one may take the Harish-Chandra spherical transform of a function of the form 
$f * f^*$, where 
$f$ is smooth and compactly supported and $f^*(g)=\overline{f(g^ {-1})}$.
It seems highly unlikely that given a matrix $\tt$, the $K(N)$-fixed forms $\varphi$ would conspire so that, given matrix $\tt$, $W(\varphi)(\tt)$ always vanishes, however inthe proof of~\cite{Mythesis} the proof of existence of such a function $h$ satisfying (3) *{Corollary~3.5.1}
made use of the Wallach inversion theorem. 
In particular the choice of $h$ a priori depends on $\tt$.
For each $\pi \in \Pi$, let $\B^{K(N)}_\pi$ be a (possibly empty) basis of the finite-dimensional space $\pi^{K(N)}$ that is orthonormal for the inner product induced by the norm~(\ref{InnerProduct}). 
We fix a matrix $\tt \in A(\R)^\circ$ and we define a spectral weight 
\begin{equation}\label{weigth}
    w_{K(N)}(\pi)=h(\nu_\pi) \sum_{\varphi \in \B^{K(N)}_\pi} |W(\varphi)(\tt)|^2.
\end{equation}
By bilinearity this weight is independent of the choice of the orthonormal basis~$\B^{K(N)}_\pi$.
Also note that if $\uu \in U(\A_\fin)$ then $\uu K(N)\uu^{-1}$ satisfies Assumption~\ref{GeneralAssumption}, and moreover $\{\pi(\uu^{-1})\varphi: \varphi \in \B_\pi^{K(N)}\}$ is an orthonormal basis for $\pi^{\uu K(N)\uu^{-1}}$, and 
$|W(\pi(\uu^{-1})\varphi)(\tt)|^2=|W(\varphi)(\tt)|^2$.
Thus the weight $w_{K(N)}(\pi)$ remains the same if we replace $K(N)$ with $\uu K(N)\uu^{-1}$.
According to the Lapid-Mao conjecture (known in some cases by~\cite{ChenIchino}), when $\pi$ is cuspidal and $\tt=1$ the weight $w_{K(N)}(\pi)$ should be related to the $L$-value $L(1,\pi,\Ad)$. When the subgroup $K(N)$ is clear from the context we just write $w_N$ for~$w_{K(N)}$.

If $\pi \in \Pi$ is a generic representation which has
a non-zero $K(N)$-fixed vector, and if $p \nmid N$ we
let $\mathcal S_p(\pi)=(\alpha(\pi_p),\beta(\pi_p))$ be the Satake parameters of the representation~$\pi_p$.
Those live in the space $\mathcal Y=(\C^\times \times \C^\times)/W$, where the action of $W$ on $\C^\times \times \C^\times$ is generated by the two transformations $(\alpha,\beta) \mapsto (\alpha, \beta^{-1})$ and $(\alpha,\beta) \mapsto (\beta, \alpha)$. Finally we let $\mu_N$ be the push-forward of the measure $w(\pi) d\pi$ along $\mathcal S_p$.

The  $\GSp_4$ Sato-Tate measure $d\mu_{\text{ST}}$ is supported on the subset $(\Ss^1 \times \Ss^1)/W \subset \Y$ consisting of 
Satake parameters of tempered representations.
In the parameterisation 
    $(\Ss^1 \times \Ss^1)/W =\{(e^{i \theta_1},e^{i \theta_2}): 0 \le \theta_1 \le \theta_2 \le \pi\}$, the Sato-Tate measure is given by
    $$d\mu_{ST}(\theta_1,\theta_2)=\frac4{\pi^2} (\cos \theta_1 - \cos \theta_2)^2 \sin^2 \theta_1 \sin^2 \theta_2 d\theta_1 d\theta_2.$$

\begin{theorem}\label{Thm1}
Fix a prime number $p$. Then the probability measure 
$\frac1{\mu_N(\mathcal Y)}\mu_N$ converges weakly to the Sato-Tate measure $d\mu_{\text{ST}}$ as $N$ tends to infinity among integers that are coprime to $p$.
This means that for any continuous $W$-invariant function $g$ on $\C^\times \times \C^\times$ we have
\begin{equation}\label{Edistribution}
    \lim_{\substack{N \to \infty \\ p \nmid N}}
\frac{\int_{\Pi} w_N(\pi) g(\alpha(\pi_p),\beta(\pi_p))\,  d\pi}{\int_{\Pi} w_N(\pi) \, d\pi} = \int_{\mathcal Y} g \, d\mu_{\text{ST}}.
\end{equation}
\end{theorem}
\begin{remark}
    For simplicity, we have chosen not to fix a central character. The same result would hold if for each $N$ we fix a central character $\omega_N$ that is trivial on $Z \cap K(N)$,
    and we restrict attention to representations $\pi$
    having central character~$\omega_N$.
\end{remark}
\begin{proof}
    As the proof follows the lines 
    of~\cite{Mythesis}*{Theorem~3.5.1}, we only sketch how the argument goes, and we explain the required modifications.
    To prove~(\ref{Edistribution}) for all continuous test function $g$, it suffices to prove it 
    for $$g_{a,b,c}(\alpha(\pi_p),\beta(\pi_p))=\W_{\pi_p}\left(\left[\begin{smallmatrix}
 p^a&   &   &    \\
   & p^b  &  &  \\
  &   & p^{c-a} & \\
  &  & & p^{c-b}
\end{smallmatrix}\right]\right)$$
for all integers $b \le a$ and $c \in \{0,1\}$.
To this end, we let  $\tt_{abc}$ be the diagonal embedding of the matrix $\left[\begin{smallmatrix}
 p^a&   &   &    \\
   & p^b  &  &  \\
  &   & p^{c-a} & \\
  &  & & p^{c-b}
\end{smallmatrix}\right]$ in $G(\A_\Q)$
and $\tt^{abc}$ be the embedding of its inverse in $G(\R)$;
and we apply the $\GSp_4$ Kuznetsov formula (as stated in~\cite{Mythesis}*{Theorem~2.6.1}, but with the more general subgroup $K(N)$)
with the two matrices $\tt_1=\tt$ and $\tt_2=\tt \tt^{abc}\tt_{abc}$  and with a test function $f=f_\infty f_{\fin}$
such that $\tilde{f}_\infty=h$ and $ f_{\fin}=\frac{\1_{K(N)}}{\Vol(K(N))}$.
The spectral side thus gives exactly the desired expression $\int_{\Pi} w_N(\pi) g_{a,b,c}(\alpha(\pi_p),\beta(\pi_p))\,  d\pi$. In particular for $(a,b,c)=(0,0,0)$ this is $\int_{\Pi} w_N(\pi) \,  d\pi.$
On the geometric side, one has
the identity contribution as well as the contribution from the three other elements from the Weyl group
 $s_1s_2s_1,s_2s_1s_2$ and $\JJ$.
We claim that 
\begin{enumerate}
    \item the identity contribution is a non-zero multiple of $\delta_{\tt_1,\tt_2} \int_{\mathcal Y} g_{a,b,c} \, d\mu_{\text{ST}}$,
    \item the contributions from the other three Weyl group elements vanish for $N$ large enough.
\end{enumerate}
To prove the claim, we recall that the contribution from each Weyl group element $\sigma$ has the shape 
$$\frac1{\Vol(K(N))}\sum_{\delta} J_{\sigma \delta}(\tt_1,\tt_2),$$
where $J_{\sigma \delta}$ is an orbital integral whose 
finite part is over the Kloosterman set
$$X(\sigma \delta) = U(\hat{\Z}) \back Y_{\sigma \delta} / \overline{U}_\sigma(\hat{\Z})$$
where 
$$Y_{\sigma \delta} = K_\fin(N) \cap U(\A) \, \sigma \delta \, \overline{U}_\sigma(\A).$$
For $\delta=1$ by~\cite{Mythesis}*{Lemma~2.5.1} the $\delta$-sum runs over matrices of the form 
$\delta=s\tt_{abc}$ where $s \in \Q^\times$.
Now if $\ggg \in Y_\delta = K_\fin(N) \cap \delta U(\A)$ then $\mu(\ggg)=s^2p^c \in \hat{\Z}^\times \cap \Q=\{\pm 1\}$. But since $c \in \{0,1\}$, this forces $s=1$ and $c=0$.
Moreover $Y_\delta \subset K_\fin(N) \cap B(\Q_p)$,
which is a compact group, and hence all the diagonal entries of $\ggg$ must lie in $\Z_p^\times$. Thus we find that $a=b=0$ as well.
This shows that the $\delta$-sum is empty unless $\tt_{abc}=1$.
To establish part~(1) of the claim, it remains to show that $J_1(\tt,\tt) \neq 0$. The fact that the infinite part of this orbital integral doesn't vanish is~\cite{Mythesis}*{Corollary~3.5.1}. On the other hand, the finite part is given by
$$\int_{U(\A_\fin) \cap K_\fin(N)} \psi(\uu) \, d \uu=\Vol_U(U(\A_\fin) \cap K_\fin(N)) \neq 0,$$
where we have used that
$\psi$ is trivial on $U(\A_\fin) \cap K_\fin(N)$
 by part~(3) of Assumption~\ref{GeneralAssumption}.
We now proceed to check part~(2) of the claim.
For $\sigma=\JJ$ the $\delta$-sum runs over diagonal matrices $\delta=\left[\begin{smallmatrix}
 d_1&   &   &    \\
   & d_2  &  &  \\
  &   & d_3d_1^{-1} & \\
  &  & & d_3d_2^{-1}
\end{smallmatrix}\right]$ with rational coefficients.
For $\ggg_{\fin}=\mat{A}{B}{C}{D} \in Y_{\sigma\delta}$
 by explicit calculation we have $C_{11}=d_1$.
 Thus by part~(2) of Assumption~\ref{GeneralAssumption}
 we must have $d_1 \in N\Z \setminus \{0\}$. We also note $d_3=\mu(\ggg_\fin) \in \Q \cap \hat{\Z}^\times=\{\pm1\}$.
 On the other hand the Archimedean part of
 $J_{\JJ \delta}$ is given by
 \begin{equation}\label{ArchLocIntJ}
 \iint_{U(\R) \times U(\R)} f_\infty(\tt^{-1}\uu_1 \JJ \delta \tt^{abc} \uu_2\tt) \psi(\uu_1  \uu_2) \, d\uu_1 d\uu_2.
 \end{equation}
 Now $f_\infty$ has compact support modulo the centre, thence so does $f_\tt:=\ggg \mapsto f_\infty(\tt^{-1} \ggg \tt)$. In particular there exists a constant~$c_{h,\tt}$ such that if $\ggg_\infty \in \Supp(f_\tt)$ then all then entries of $\frac1{\sqrt{|\mu(\ggg_\infty)|}}\ggg_\infty$ are bounded in absolute value by $c_{h,\tt}$. But 
 for $\ggg_\infty = \uu_1 \JJ \delta \tt^{abc} \uu_2 = \mat{A'}{B'}{C'}{D'}$ 
 by the same calculation as in the finite case we have $C'_{11}=p^{-a}d_1$ and $\mu(\ggg_\infty)=p^cd_3=\pm p^c$. Thus we wee that when $N>p^{a+\frac{c}2}c_h$ the integrand of~(\ref{ArchLocIntJ}) vanishes.
 The argument is the same for $\sigma=s_1s_2s_1$.
 For $\sigma=s_2s_1s_2$ the $\delta$-sum runs over diagonal matrices $\delta$ as above with $d_2=-p^{a-b}d_1$.
For $\ggg_{\fin}=\mat{A}{B}{C}{D} \in Y_{\sigma\delta}$
 by explicit calculation we have $C_{11}=0$, $C_{12}=-d_1$ and $C_{21}=-d_2$.
 Thus by Assumption~\ref{GeneralAssumption} we see that
 $d_1^2 \in p^{a-b}N\Z.$
From then on, the argument is similar and we conclude that $\sigma=s_2s_1s_2$ has no contribution to the geometric side as soon as $N > p^{a+b+c} c_h^2.$
\end{proof}

Since the continuous spectrum is parameterised by 
automorphic representations of lower rank groups,
it is more interesting to isolate the contribution from the cuspidal spectrum only.
To this end, we need to take a closer look at the continuous spectrum, which is the purpose of next section.

\section{The continuous spectrum}~\label{sec4}
In order to refine Theorem~\ref{Thm1} to the cuspidal spectrum, it is necessary to understand the structure of the continuous spectrum. In this section we do not work with our subgroup $K(N)$, but rather with a compact open subgroup $H \subset K$, that we later specialise to the Borel congruence subgroup with squarefree level.
Ultimately, in Section~\ref{Ecusp}, we return to 
our subgroup~$K(N)$ (subject to some further assumption), and we then reduce our problem to the situation of the Borel congruence subgroup.

\subsection{General framework}\label{GenF} Before presenting the general strategy we would like to give a more
concrete description of the $H$-invariant continuous spectrum, where $H \subset K$ is a compact open subgroup. 
Let $P$ be a parabolic subgroup.
Fix $X_P$ a set of representatives of $P(\A) \back G(\A) / H$.
For each $\xx \in X_P$ define the subgroups
$$P_\xx =P(\A) \cap \xx H \xx^{-1}$$
and
$$M(\xx)=\Proj_P^{M_P}(P_\xx).$$ 
Recall the description of Eisenstein series in~\S~\ref{ESIR}, and in particular the notation~(\ref{phix})

\begin{lemma}~\label{ExplicitContinuous}
    For any automorphic representation~$\sigma$ occurring in
    the discrete spectrum of $M_P$ there is an isomorphism
   \begin{equation*}
       \begin{split}
           \I_P(\sigma_\nu)^{H} & \to \bigoplus_{\xx \in X_P} \sigma^{M(\xx)}\\
           \phi &\mapsto (\phi_\xx)_{\xx \in X_P}
       \end{split}
   \end{equation*} 
\end{lemma}
\begin{proof}
    By~(\ref{Action}) and right-$K$-invariance of~$I_{P,\nu}$
    we have $\phi \in \I_P(\sigma_\nu)^{H}$ if and only if $\phi$ is right-$H$-invariant.
Such a function $\phi$ is determined by the family 
$(\phi_\xx)_{\xx \in G(\A)/ H}$, which must in 
addition satisfy
$$\phi(\mm\pp\xx)=\phi(\underbrace{(\mm\nn_1\mm^{-1})}_{\in N(\A)}\mm\mm_1\xx)\\
=\phi(\mm\mm_1\xx),$$
 for all $\pp=\nn_1\mm_1 \in P$ and for all $\mm \in M(\A)$, and hence
\begin{equation}~\label{inducingdata}
\phi_{\pp\xx}=\sigma(\mm_1)\phi_\xx.
\end{equation}
Thus $\phi$ is completely determined by $(\phi_\xx)_{\xx \in X_P}$.
Moreover, by~(\ref{inducingdata}) for each $\xx \in X_P$, the function $\phi_\xx$ is a vector in the space of $\sigma$ which is right-invariant by the subgroup 
\begin{align*}
    \Proj_P^M(\Stab_{P(\A)}(\xx H))=\Proj_P^M(P_\xx)=M(\xx).
\end{align*}
Conversely, if for each $\xx \in X_P$ we have an element $\phi_\xx$ of $\sigma^{M_\xx}$, then
we can define a $H$-invariant function on
$N_P(\A)M_P(\Q)A_P^\circ(\R) \back G(\A)$ 
by setting 
\begin{equation*}
    \phi(\nn\mm\xx\hh)=\phi_\xx(\mm)
\end{equation*} for all $\nn \in N_P(\A)$, $\mm \in M_P(\A)$ and for all $\hh \in H$.
\end{proof}

Assume the set of representatives $X_P$ from Lemma~\ref{ExplicitContinuous} is a subset of $K$ (this is possible by the Iwasawa decomposition). Then the inner product
induced by the norm~(\ref{InnerProduct}) on the 
$\I_P(\sigma_\nu)^H$ can be written as follows
\begin{equation}\label{ExplicitInnerProduct}
    \begin{split}
        \langle \phi, \psi \rangle = \int_K \langle \phi_\kk, \psi_\kk \rangle_\sigma \, d\kk = \sum_{\xx \in X_P} v_\xx  \langle \phi_\xx, \psi_\xx \rangle_M,
    \end{split}
\end{equation}
where 
$$v_\xx=\Vol((P(\A) \cap K) \xx H)=[P(\A) \cap K: P_\xx]\Vol(H)$$ and 
$$\langle\phi_\xx, \psi_\xx \rangle_M= \int_{M(\Q) \back M(\A)^1} |\phi(\mm\kk)|^2 \, d\mm.$$
Thus we can describe an orthonormal $\B^H_\sigma$ basis of $\I_P(\sigma_\nu)^H$ as follows.
For each $\xx \in X_P$ let $\B^\xx_\sigma$ be an orthonormal basis of $\sigma^{M(\xx)}$, and for each $\varphi \in \B^\xx_\sigma$, let $\phi^{(\xx,\varphi)}$ be defined by
$$\phi_{\yy}^{(\xx,\varphi)}=\delta_{\xx \yy} \varphi \text{ for all } \yy \in X_P.$$
It is then clear that $\B_\sigma^H=\left(\frac1{\sqrt{v_\xx}}\phi^{(\xx,\varphi)}\right)_{\xx,\varphi}$ is an orthonormal basis of $\I_P(\sigma_\nu)^H$.

When applying the Kuznetsov formula with $f_\fin=\frac1{\Vol(H)}f_\infty \1_{H}$, the contribution
to the spectral side 
from the continuous spectrum induced from $P$ is given by the integral over $\nu$ of the quantity
\begin{equation*}
\begin{split}
    \sum_\sigma& \sum_{\phi \in \B^H_\sigma}
 \tilde{f}_\infty (\mz_P(\nu)+\ma_P(\nu_\sigma))W(E(\cdot, \phi, \nu))(\tt_1) 
\overline{W(E(\cdot, \phi, \nu))}(\tt_2) \\
=&\frac1{\Vol(H)}\sum_{\xx \in X_P} \frac1{[P \cap K:P_\xx]} \, \times \cdots \\ 
&\cdots \times \sum_\sigma \sum_{\varphi \in \B^\xx_\sigma} \tilde{f}_\infty (\mz_P(\nu)+\ma_P(\nu_\sigma))W(E(\cdot, \phi^{(\xx,\varphi)}, \nu))(\tt_1) 
\overline{W(E(\cdot, \phi^{(\xx,\varphi)}, \nu))}(\tt_2).
\end{split}
\end{equation*}
In first approximation, treating the Whittaker coefficients of the Eisenstein series as if they were constants, one would expect from the Weyl law for~$M$ the inner double sum to have size approximately $[M \cap K: M(\xx)]$, with some decay in $\nu$ making the integral converge. Thus using Lemma~\ref{grouptheoretic} one would naively expect a bound of size
$$\frac1{\Vol(H)}\sum_{\xx \in X_P}\frac{[M \cap K: M(\xx)]}{[P \cap K:P_\xx]}=\frac1{\Vol(H)}\sum_{\xx \in X_P}\frac{1}{[N \cap K:N_\xx]},
$$
where $N_\xx=N \cap \xx H \xx^{-1}$. On the other hand,
by the same argument as in the proof of Theorem~\ref{Thm1}, we expect the full spectral side to have size 
$\asymp \frac{\Vol_U(U \cap H)}{\Vol(H)}$.
Thus, if we are to prove that the continuous contribution is negligible compared to the cuspidal contribution, we would need a bound of the form
$$\sum_{\xx \in X_P} \frac1{[N \cap K: N_\xx]} =o\left(\Vol_U(U \cap H)\right).$$
But this can't be true in general: taking $\xx=1$ we have $N_\xx=N \cap H$. But on the other hand if $U \cap H= (U \cap M \cap H)(N \cap H)$ (which is the case \textit{e.g,} for $P=B$ or if $U \cap H= U \cap K$) then
\begin{equation*}
    \begin{split}
        \Vol_U(U \cap H)&=\Vol_{U \cap M}(U \cap M \cap H) \Vol_N(N \cap H) \\
        &\le \Vol_N(N \cap H)=\frac1{[N \cap K: N \cap H]}.
    \end{split}
\end{equation*}
However at this stage we have not exploited the dependence of the Whittaker coefficients $W(E(\cdot, \phi^{(\xx,\varphi)}, \nu))$ in $\xx$.
Thus showing that the contribution from the continuous spectrum is negligible ultimately boils down to the following question
\begin{question}\label{Q1}
Show that $W(E(\cdot, \phi^{(\xx,\varphi)}, \nu))(\tt_i)$ is sufficiently small
whenever the index $[N \cap K:N_\xx]$ is not large.
\end{question}
We have intentionally phrased Question~\ref{Q1} in a rather vague manner, because we do not wish to speculate on what would be the correct statement of a more precise formulation. 
However it seems clear that it should depend on the finer structure of the double quotient $P(\A) \back G(\A) / H$. For general $H$ this seems a challenging problem, and  thus in the sequel we will make further assumptions (see Assumption~\ref{NewAssumption} and~\S~\ref{redu}).

\subsection{Our case}\label{Ourcase} Ultimately we shall work with $$H=B(N)=\{\ggg \in G(\hat{\Z}):  \ggg \equiv \bb \mod N \text{ for some } \bb \in B\}=\prod_p B_p(N)$$ and $N$ squarefree. According to our general framework above,
there are three types of objects we need to understand:
the double cosets $\xx \in (P \cap K) \back K / B(N)$,
the automorphic representations of $M$ having non-zero $M(\xx)$-fixed vectors $\varphi$, and the local integrals giving the Whittaker coefficients 
$W(E(\cdot, \phi^{(\xx,\varphi)}, \nu))$.
We give the relevant results for the first two, and we defer the calculation of the local integrals to Section~\ref{Boring}.
In the rest of this section, $N$ denotes a squarefree integer.

\subsubsection{Double cosets}
For each prime $p \mid N$ we let 
\begin{equation*}
\begin{split}
    X_B(p)&=W,\\
    X_{\Ps}(p)&=\{1,s_2,s_2s_1,s_2s_1s_2\},\\
    X_{\Pk}(p)&=\{1,s_1,s_1s_2,s_1s_2s_1\}.
    \end{split}
\end{equation*}
By the Bruhat decomposition over $\F_p$, 
$X_P(p)$ is in each case a set of representatives for the double quotient $(P(\Q_p) \cap K_p)\back K_p / B_p(N)$.
For each prime $p \nmid N$ as well as $p=\infty$ we let $X_P(p)=\{1\}$.
Thus, setting $X_P=\prod_v X_P(v)$, the 
set~$X_P$ is a set of representatives for the double quotient $(P \cap K) \back K / B(N).$
Moreover if $\xx=\prod_v \xx_v \in X_P$ then $P_\xx=\prod_v P_{\xx_v}$, where 
$P_{\xx_v}=P(\Q_v) \cap \xx_v B_p(N) \xx_v^{-1}$.
As seen is the next lemmas, our choice of set of representatives is such that for $P=\Pk,\Ps$
for $\xx \in X_P$ we have $\mf_2(P_\xx)=\Gamma_0(N)$, the subgroup of $\GL_2(\A)$
whose Archimedean component is $\rm{O}_2(\R)$ and
whose finite component is the adelic version of the classical Hecke congruence subgroup.

\begin{lemma}\label{SiegelStabilizers}
Let $N$ be squarefree and $p \mid N$ be a prime.
For $P=\Ps$, the subgroups $P_{\xx_p}=P(\Q_p) \cap \xx_p B_p(N) \xx_p^{-1}$ are as follows
\begin{itemize}
    \item for $\xx_p=1$, $P_{\xx_p}=P \cap B(p)=P \cap \left[\begin{smallmatrix}
 \o     & \o  &     &   \\
 \p     & \o  &     &  \\
        &     & \o  & \p\\
        &     & \o  & \o
\end{smallmatrix}\right]
\left[\begin{smallmatrix}
 1   &     &  \o & \o   \\
      & 1  & \o & \o \\
        &       & 1 & \\
        &       &   &1
\end{smallmatrix}\right]$, 
\item for $\xx_p=s_2$, 
 $P_{\xx_p}=P \cap \left[\begin{smallmatrix}
 \o     & \o  &     &   \\
 \p     & \o  &     &  \\
        &     & \o  & \p\\
        &     & \o  & \o
\end{smallmatrix}\right]
\left[\begin{smallmatrix}
 1   &     &  \o & \o   \\
      & 1  & \o & \p \\
        &       & 1 & \\
        &       &   &1
\end{smallmatrix}\right]$, 
\item for $\xx_p=s_2s_1$, $P_{\xx_p}=P \cap \left[\begin{smallmatrix}
 \o     & \o  &     &   \\
 \p     & \o  &     &  \\
        &     & \o  & \p\\
        &     & \o  & \o
\end{smallmatrix}\right]
\left[\begin{smallmatrix}
 1   &     &  \o & \p   \\
      & 1  & \p & \p \\
        &       & 1 & \\
        &       &   &1
\end{smallmatrix}\right]$, %and $P_s^{(1)}=P \cap \left[\begin{smallmatrix}
%  1+\p     & \p  &     &   \\
%  \p     & 1+\p  &     &  \\
%         &     & 1+\p  & \p\\
%         &     & \p  & 1+\p
% \end{smallmatrix}\right]
% \left[\begin{smallmatrix}
%  1   &     &  \o & \p   \\
%       & 1  & \p & \p \\
%         &       & 1 & \\
%         &       &   &1
% \end{smallmatrix}\right]$,
\item for $\xx_p=s_2s_1s_2$, 
 $P_{\xx_p}=P \cap \left[\begin{smallmatrix}
 \o     & \o  &     &   \\
 \p     & \o  &     &  \\
        &     & \o  & \p\\
        &     & \o  & \o
\end{smallmatrix}\right]
\left[\begin{smallmatrix}
 1   &     &  \p & \p   \\
      & 1  & \p & \p \\
        &       & 1 & \\
        &       &   &1
\end{smallmatrix}\right]$.% and $P_s^{(1)}=P \cap \Gamma(p)$.
\end{itemize}
\end{lemma}
\begin{proof}
    Direct calculation.
\end{proof}

\begin{lemma}\label{KlingenStabilizers}
Let $N$ be squarefree and $p \mid N$ be a prime.
For $P=\Pk$, the subgroups $P_{\xx_p}=P(\Q_p) \cap \xx_p B_p(N) \xx_p^{-1}$ are as follows
\begin{itemize}
    \item for $\xx_p=1$, $P_{\xx_p}=P \cap B(p)=P \cap \left[\begin{smallmatrix}
 \o     &   &     &   \\
      & \o  &     & \o \\
        &     & \o  & \\
        &  \p   &   & \o
\end{smallmatrix}\right]
\left[\begin{smallmatrix}
 1   &   \o  &  \o & \o   \\
      & 1  & \o &  \\
        &       & 1 & \\
        &       & \o  &1
\end{smallmatrix}\right]$, 
\item for $\xx_p=s_1$, 
 $P_{\xx_p}=P \cap \left[\begin{smallmatrix}
 \o     &   &     &   \\
      & \o  &     & \o \\
        &     & \o  & \\
        &  \p   &   & \o
\end{smallmatrix}\right]
\left[\begin{smallmatrix}
 1   &   \p  &  \o & \o   \\
      & 1  & \o &  \\
        &       & 1 & \\
        &       & \o  &1
\end{smallmatrix}\right]$, 
\item for $\xx_p=s_1s_2$, $P_{\xx_p}=P \cap \left[\begin{smallmatrix}
 \o     &   &     &   \\
      & \o  &     & \o \\
        &     & \o  & \\
        &  \p   &   & \o
\end{smallmatrix}\right]
\left[\begin{smallmatrix}
 1   &   \p  &  \p &   \o \\
      & 1  & \o &  \\
        &       & 1 & \\
        &       & \p  &1
\end{smallmatrix}\right]$, 
\item for $\xx_p=s_1s_2s_1$, 
 $P_{\xx_p}=P \cap \left[\begin{smallmatrix}
 \o     &   &     &   \\
      & \o  &     & \o \\
        &     & \o  & \\
        &  \p   &   & \o
\end{smallmatrix}\right]
\left[\begin{smallmatrix}
 1   &   \p  &  \p & \p   \\
      & 1  & \p &  \\
        &       & 1 & \\
        &       & \p  &1
\end{smallmatrix}\right]$.
\end{itemize}
\end{lemma}
\begin{proof}
    Direct calculation.
\end{proof}
\begin{lemma}\label{BorelStabilizers}
Let $N$ be squarefree and $p \mid N$ be a prime.
For $P=B$ and $\xx_p=\sigma \in W$ the subgroup $P_{\xx_p}=P(\Q_p) \cap \xx_p B_p(N) \xx_p^{-1}$ 
is given by $P_{\xx_p}=A(\o)U_\sigma(\p)$.
\end{lemma}
\begin{proof}
    Almost by definition.
\end{proof}
\begin{remark}
    In these cases we have $P_{\xx_p}=(M \cap P_{\xx_p})(N \cap P_{\xx_p})$. 
    This needs not hold in general.
\end{remark}

\subsubsection{Local representations}
We are now able to give a brief description of the representations~$\sigma$ of $M_P$ such that $\I_P(\sigma_\nu)$ occurs in the continuous spectrum.
If $P=B$ then by Lemmas~\ref{ExplicitContinuous} and~\ref{BorelStabilizers}, $\sigma$ is a character of
$(\A^\times/\Q^\times)^3$ such that
$\sigma_p$ is trivial on $(\Z_p^\times)^3$ for all prime $p$, and $\sigma_\infty=1$. Therefore the only possibility is $\sigma=1$.
For $P=\Pk,\Ps$, the representation $\sigma$ can be viewed as an automorphic representation $\chi \times \tau$
of $\GL_1(\A) \times \GL_2(\A)$ via the isomorphism $\mf_1 \otimes \mf_2: M_P \to \GL_1 \times \GL_2$.
By the same argument as before, we see that $\chi$ and $\omega_\tau$ must be trivial.
Next, by Lemmas~\ref{ExplicitContinuous},~\ref{SiegelStabilizers}
and~\ref{KlingenStabilizers}, $\tau$ must have a $\Gamma_0(N)$-fixed vector. Thus $\tau$ is either 
a Maa{\ss} form for $\Gamma_0(N)$ or a character,
but we only need to consider the former case since characters of $\GL_2$ are not generic. 

\section{Equidistribution over the cuspidal spectrum}\label{Ecusp}
We are considering a subgroup $K(N)$ subject to the following.
\begin{assumption}\label{NewAssumption} The subgroup $K_\fin(N)$ satisfies Assumption~\ref{GeneralAssumption}
and there exists $\uu_0 \in U(\A_\fin)$ such that $\uu_0^{-1} B(N) \uu_0 \subset K_\fin(N)$, where $B(N)$ is the Borel 
congruence subgroup 
$$B(N) =\{\ggg \in G(\hat{\Z}):  \ggg \equiv \bb \mod N \text{ for some } \bb \in B\}.$$ 
Furthermore from now on we shall always assume $N$ is squarefree.
\end{assumption}

Under Assumption~\ref{NewAssumption} we are able to refine our equidistribution result to the cuspidal spectrum.
\begin{theorem}\label{Thm2}
Assume the subgroups $K(N)$ satisfy Assumption~\ref{NewAssumption}
Fix a prime number $p$. Then 
the Satake parameters at $p$ of cuspidal representations~$\pi$ of $\GSp_4(\A)$, weighted by $W_{N}(\pi)$,
equidistribute with respect to the Sato-Tate measure $d\mu_{\text{ST}}$
as $N$ tends to infinity among integers that are squarefree and coprime to $p$.
This means that for any continuous $W$-invariant function $g$ on $\C^\times \times \C^\times$ we have
\begin{equation}\label{Edistribution2}
    \lim_{\substack{p \nmid N \to \infty \\ N \> \square-\text{free}}}
\frac{\sum_{\pi \in \Pi_{\disc}} w_N(\pi) g(\alpha(\pi_p),\beta(\pi_p))\,  d\pi}{\sum_{\pi \in \Pi_{\disc}} w_N(\pi) \, d\pi} = \int_{\mathcal Y} g \, d\mu_{\text{ST}}.
\end{equation}
\end{theorem}
\begin{remark}
    In the statement of Theorem~\ref{Thm2} we have used the adjective ``cuspidal", but the sum in~(\ref{Edistribution2}) is over $\Pi_{\disc}$.
    This is not an inconsistency, because as proved in~\cite{Mythesis}*{Lemma~2.4.6} the weight $w_{K(N)}(\pi)$
    vanishes for representations $\pi \in \Pi_{\disc}$
    that are not cuspidal.
\end{remark}

Note that Assumption~\ref{NewAssumption} is not a special case of the general framework from Section~\ref{GenF}, as we do not assume at this stage that $K(N) \subset K$. However in the next section we reduce our
problem to the case of $K_\fin(N)=B(N)$, which does fit in our framework.

The reason we are assuming $N$ is squarefree is that
for primes $p \mid N$ the calculation of the double quotient $B(\Z_p) \back K_p / B(N\Z_p)$ amounts to the Bruhat decomposition over~$\F_p$.
On the other hand when $N$ is not squarefree, it is not clear to us what this double quotient looks like. 
Our assumption also implies that we are only considering representations which have a trivial central character. One could without much more effort allow for an arbitrary central character with conductor dividing $N$, however we have not pursued this here.

We now explain our strategy to tackle Question~\ref{Q1}.
The Whittaker coefficients $W(E(\cdot, \phi^{(\xx,\varphi)}, \nu))$ 
are given (in its domain of absolute convergence)
by the integral~(\ref{WCES}).
The fact that the integrand thereof involves the Whittaker coefficients of $\GL_2$ automorphic forms (when $P=\Pk,\Ps$) makes it very tempting to swap summation and integration order and execute the $\GL_2$
Kuznetsov formula inside the integral. This approach presents certain technical difficulties, one of which is that~(\ref{WCES}) does not converge in the range of $\nu$ we are interested in. However we proceed among similar lines, namely we bound the Whittaker coefficients $W(E(\cdot, \phi^{(\xx,\varphi)}, \nu))$
and then we treat the double sum $\sum_\pi \sum_\varphi$
using the $\GL_2$ Kuznetsov formula.
As we mentioned the integral~(\ref{WCES}) (with $\phi=\phi^{(\xx,\varphi)}$)  does not converge when $\nu$ is purely imaginary, thus it is not sufficient to bound~(\ref{WCES}) in its domain of convergence. 
Instead, calculating the integral, we obtain explicit formulas
for $W(E(\cdot, \phi^{(\xx,\varphi)}, \nu))$, which are still well-defined when $\nu$ is purely imaginary. Using simple bounds for them 
allows us to provide a positive answer to Question~\ref{Q1} in our case.
Ultimately, the good dependence in~$\xx$ that is required for Question~\ref{Q1} comes from the fact that for $\phi=\phi^{(\xx,\varphi)}$ and $\ggg=1$
the integral~(\ref{WCES}) is supported on the set of those
$\nn \in N(\A)$ such that 
$\sigma_P \nn \in B(\Q_p) \xx K(N)$, and in this set 
the quantity $I_{P,\rho_P+\nu}(\sigma_P \nn)$
decreases all the more rapidly as $[N \cap K: N_\xx]$ is small.

It would be interesting to be able to implement this strategy without assuming that $N$ is squarefree.
As alluded to above, this would 
involve understanding the fine structure of the double cosets in $P(\Q_p) \backslash G(\Q_p) / K_p(N)$ in greater generality.
Also in our approach we were able to calculate by ``brute force" each of the $16$ integrals associated to each of the $16=8+4+4$ double cosets. 
In a more general setting, one would probably need to find a uniform way of treating the integrals associated to each double coset.
Another feature not seen in the squarefree case (which should not lead to any major difficulty) is that we may
induce from representations of $\GL_2$ that have non trivial central character, and from non-trivial characters of $\GL_1$, even if we are working with a trivial central character. The reason is that in our case $P_\xx$ is always  trivial on the $\GL_1$-components, but this needs not be the case in general.
In particular for the Borel subgroup, this would lead to a sum of Gau{\ss} sums.

\subsection{Some reductions}\label{redu}
Let $\Pi_{\text{cont}}=\Pi_{\text{K}} \sqcup \Pi_{\text{S}} \sqcup \Pi_{\text{B}}$.
In order to deduce Theorem~\ref{Thm2} from Theorem~\ref{Thm1}
we want to show that for any $W$-invariant continuous bounded function~$g$
on~$\C^2$ we have 
$$ \lim_{\substack{N \to \infty \\ p \nmid N}}
\frac{\int_{\Pi_{\text{cont}}} w_{K(N)}(\pi) g(\alpha(\pi_p),\beta(\pi_p))\,  d\pi}{\int_{\Pi} w_{K(N)}(\pi) \, d\pi}=0.$$
Clearly $$\left|\int_{\Pi_{\text{cont}}} w_{K(N)}(\pi) g(\alpha(\pi_p),\beta(\pi_p))\,  d\pi\right| \le \|g\|_\infty \int_{\Pi_{\text{cont}}} w_{K(N)}(\pi) \, d\pi.$$
We already know (by the argument in proof of Theorem~\ref{Thm1}) that 
$$\int_{\Pi} w_{K(N)}(\pi) \, d\pi \asymp \frac{\Vol_U(U(\A_\fin) \cap K_\fin(N))}{\Vol(K(N))}.$$
Hence we need to show that
$$ \int_{\Pi_{\text{cont}}} w_{K(N)}(\pi) \, d\pi=o\left(\frac{\Vol_U(U(\A_\fin) \cap K_\fin(N))}{\Vol(K(N))}\right).$$
On the other hand, if $B(N) \subset K(N)$ then every $B(N)$-invariant automorphic form is also $K(N)$-invariant, and thus $w_{K(N)} \le w_{B(N)}$. This majoration still holds more generally for $K(N)$ satisfying Assumption~\ref{NewAssumption}, since the weight $w_{K(N)}$ is invariant by conjugation by any element of $U(\A_\fin)$.
Combining this with Lemma~\ref{MoreExplicit} below, it suffices to show that
\begin{equation} \label{ToShow}
    \Vol(B(N))\int_{\Pi_{\text{cont}}} w_{B(N)}(\pi) \, d\pi=o\left(N^{-1+\epsilon}\right).
\end{equation} 
This means that, instead of applying the Kuznetsov formula for an arbitrary subgroup~$K(N)$ satisfying Assumption~\ref{NewAssumption}, we are reduced to apply it for the Borel congruence subgroup $B(N)$ and $\tt_1=\tt_2=\tt$ -- provided that we can prove good enough bounds.
\begin{lemma}\label{MoreExplicit}
    Under Assumption~\ref{NewAssumption}, 
    we have $$\frac{\Vol_U(U(\A_\fin) \cap K_\fin(N))}{\Vol(K(N))} \gg \frac{N^{-1+\epsilon}}{\Vol(B(N))}$$
    for all $\epsilon >0$.
\end{lemma}
\begin{proof}
Fix a prime $p$.
Firstly, since $K_p(N)$ is compact, it is contained in a maximal compact subgroup.
Up to conjugation, those are $K_p=G(\o)$ and the first paramodular subgroup $\K_p=G(\Q_p) \cap \left[\begin{smallmatrix}
 \o &  \o  &  \p^{-1} &   \o \\
\p  & \o  & \o & \o \\
 \p  &  \p  & \o  & \p\\
\p  & \o  & \o & \o 
\end{smallmatrix}\right].$
Consider first the case $K_p(N) \subset \ggg K_p \ggg^{-1}$ for some $\ggg \in G(\A)$.
By the Iwasawa decomposition we can assume without loss of generality that $\ggg =\uu \aa$ where $\uu \in U(\Q_p)$ and $\aa=\left[\begin{smallmatrix}
 1 &   &   &   \\
 & p^i  &  & \\
   &  & p^j & \\
 &   &  & p^{j-i} 
\end{smallmatrix}\right]$ for some integers $i,j$.
   Since by assumption we have $\uu_0^{-1} B(N) \uu_0 \subset K(N)$ for some $\uu_0 \in U(\A_\fin)$
   we have 
   $$(\uu_0\uu)^{-1} \left\{\left[\begin{smallmatrix}
 1 &  x  &   &    \\
  & 1& &  \\
   &   &1  & \\
 &  & x & 1
\end{smallmatrix}\right] : x \in \o \right\} \uu_0 \uu \subset \aa K_p \aa^{-1},$$
from which we deduce that we must have $-i \le 0$.
Similarly 
   $$(\uu_0\uu)^{-1} \left\{\left[\begin{smallmatrix}
 1 &    &   &    \\
 x & 1& &  \\
   &   &1  &x \\
 &  &  & 1
\end{smallmatrix}\right] : x \in \p \right\} \uu_0 \uu \subset \aa K_p \aa^{-1},$$
from which we deduce that we must have $i \le 1$.
Repeating the same argument with all the root subgroups,
we arrive at the conclusion that the only two possibilities are $\aa=1$ or $\aa=\left[\begin{smallmatrix}
 1 &   &   &   \\
 & 1  &  & \\
   &  & p & \\
 &   &  & p 
\end{smallmatrix}\right]$.
If $\aa=1$ then we have $U(\Q_p) \cap K_p(N)=\uu_0^{-1}U(\o)\uu_0$, which has volume $1$, and by
Lemmas~\ref{CongruenceSbgp} combined with the Bruhat decomposition over $\F_p$ we have $\Vol(K_p(N))=\Vol(\uu^{-1}K_p(N)\uu) \ll p^{-3} \asymp \frac{\Vol(B(p))}{p}$.
If $\aa=\left[\begin{smallmatrix}
 1 &   &   &   \\
 & 1  &  & \\
   &  & p & \\
 &   &  & p 
\end{smallmatrix}\right]$ using Lemma~\ref{Increasing} we have 
\begin{equation*}
\begin{split}
    \frac{\Vol_U(U(\Q_p) \cap K_p(N))}{\Vol(K_p(N))}&=
\frac{\Vol_U(U(\Q_p) \cap \uu^{-1}K_p(N)\uu)}{\Vol(\uu^{-1} K_p(N) \uu)} \\
&\ge \frac{\Vol_U(U(\Q_p) \cap \aa K_p\aa^{-1})}{\Vol(\aa^{-1} K_p \aa)}=p^3.
\end{split}
\end{equation*}
Next consider the case $K_p(N) \subset \ggg K_1(p) \ggg^{-1}$. By Lemma~\ref{ParamodularIwasawa} we can assume without loss of generality that $\ggg =\uu \aa$ where $\uu \in U(\Q_p)$ and $\aa=\left[\begin{smallmatrix}
 1 &   &   &   \\
 & p^i  &  & \\
   &  & p^j & \\
 &   &  & p^{j-i} 
\end{smallmatrix}\right]$ for some integers $i,j$.
By a similar ``root space by root space" argument as in previous case, we see that we must have $\aa=1$.
Then using Lemma~\ref{Increasing} we are done because
$ \Vol_U(U(\Q_p) \cap \K_p)=p$ and $\Vol(\K_p) \asymp p^{-2}$.
\end{proof} 
The rest of this work is devoted to prove (a stronger version of)~(\ref{ToShow}).

\subsection{Lower bounds for $L$-functions} Shahidi's formula for the Whittaker coefficients of Eisenstein series at unramified place give rise to the inverse of certain $L$-functions, evaluated at the edge of the critical strip.
We thus need lower bounds for these $L$-functions.
We mention that the result from~\cite{GL} applies precisely to the $L$-functions we are interested in, however we need a good dependence in the level, which is not provided there. Let $t \in \R$.
The bound 
\begin{equation}\label{Poussin}
    \zeta(1+it) \gg \frac1{\log(|t|)}
\end{equation}
was already known by de la Vall\'ee-Poussin.
We need similar results for the $L$-function attached to 
an automorphic cuspidal representation of $\GL_2(\A)$ and its symmetric square. 
Lower bounds at $s=1$ for these $L$-functions are available in the literature, however we need lower bounds on the whole line
$\Re(s)=1$.
Deriving bounds at $s=1+it$ essentially amounts to twisting by $|\cdot|^{it}$, but as it is not always clear how this twisting affects the argument, we include a proof for completeness. First we start with a technical result.
Following~\cite{RamWa} we say that an automorphic representation
on $\GL_2$ is of solvable polyhedral type if $\sym^n \pi$ is cuspidal and admits a non-trivial self-twist for some $n\in\{1,2,3\}$.

\begin{lemma}\label{polyhedral}
Let $\pi$ be 
an automorphic cuspidal representation of $\GL_2(\A_\Q)$ with a trivial central character and square-free level. Then $\pi$ is not of solvable polyhedral type.
\end{lemma}
\begin{proof}
    If $\pi$ is spherical this is~\cite{RamWa}*{Proposition~6.8}
    Assume $\pi$ is not spherical and fix a ramified prime $p$.
    The representation $\pi_p$ has conductor $p$ and trivial central character, hence the only possibility is 
    $\pi_p=\St_2$, the Steinberg representation of degree $2$.
    But it is well known that $\sym^n \St_2=\St_{n+1}$, 
    and $\St_{n+1}$ does not admit any non-trivial self-twist.
\end{proof}

\begin{lemma}\label{1line}
  Let $\pi$ be 
an automorphic cuspidal representation of $\GL_2(\A_\Q)$ with a trivial central character.
Then for all $t \in \R$ we have
$$L(1+it,\pi) \gg_{\epsilon} (\mq(|t|+3))^{-\epsilon},$$
$$L(1+it,\pi,\sym^2) \gg_{\epsilon} (\mq(|t|+3))^{-\epsilon} .$$
where $\mq=\mq(\pi)$ is the analytic conductor.
\end{lemma}
\begin{proof}
   To unify notations, we let $\Pi$ denote either $\pi$ or the symmetric square lift of $\pi$ on $\GL_3$. By Lemma~\ref{polyhedral}, $\pi$ is in particular not dihedral. Thus by~\cite{GJ}*{Theorem~9.3} the symmetric square lift of $\pi$  is also cuspidal.
   Define    
   $$L(s)=\zeta(s)L(s+it,\Pi)L(s-it,\overline{\Pi})L(s,\Pi \times\overline{\Pi}).$$
   Before proceeding any further, we note some properties of $L(s)$.
   Firstly, using the fact that $\Pi$ is always cuspidal, the function $L(s)$
   has a double pole at $s=1$, coming from the simple poles of $\zeta$ and of $L(s,\Pi \times\overline{\Pi})$.
   As in~\cite{HL}, one verifies that $L(s)=\sum_{n>0}\frac{b_n}{n^s}$ where $b_1=1$ and  $b_n$ are non-negative.
   Next we need some information about the real zeroes of $L(s)$.
   The $L$-function $L(s,\Pi)$ has a standard zero-free region (see~\cite{IK}*{Theorem~5.10}) and no Siegel zero, as shown in the Appendix of~\cite{HL} for $\pi=\Pi$ and in~\cite{Banks} for the symmetric square. Therefore
   $L(s+it,\Pi)L(s-it,\Pi)$ has no real zero $s>1-\frac{c}{\log(\mq (|t|+3))}$. For $\pi=\Pi$ the $L$-function 
   $L(s,\Pi \times \Pi)$ factors as $\zeta(s)L(s,\pi,\sym^2)$
   and thus has no Siegel zero.
   On the other hand when $\Pi$ is the symmetric square lift,
   using Lemma~\ref{polyhedral} and~\cite{RamWa}*{Theorem~B}, it follows that $L(s,\Pi \times\overline{\Pi})$ has no Siegel zero (here we are using also that $\Pi=\overline{\Pi}$ since $\pi$ has trivial central character).
   Finally, suitable estimates on the line $\Re(s)=\frac12$
   can be obtained by the same method as~\cite{HL}*{Lemma~1.2}.
   Adapting the argument of~\cite{HL}*{Proposition~1.1}, we deduce that 
   $$\Res_{s=1} L(s) \gg \log(\mq(|t|+3)).$$
   Now 
   \begin{equation*}
   \begin{split}
       \Res_{s=1} L(s) &= r_0(\Pi \times \overline{\Pi}) |L(1+it)|^2\\
       &+ \gamma |L(1+it)|^2 \Res_{s=1} L(s,\Pi \times \overline{\Pi})\\
       &+ 2L(1+it,\Pi)L'(1+it,\Pi) \Res_{s=1}  L(s,\Pi \times \overline{\Pi}),
       \end{split}
   \end{equation*}
   where $r_0(\Pi \times \overline{\Pi})$ is the constant coefficient in the Laurent expansion of $L(s,\Pi \times \overline{\Pi})$.
   By~\cite{Molteni}*{Theorem~2} it satisfies 
   $r_0(\Pi \times \overline{\Pi}) \ll_\epsilon \mq^\epsilon$.
   Moreover we have $\Res_{s=1} L(s,\Pi \times \overline{\Pi}) \ll_\epsilon \mq^\epsilon$ (for $\Pi=\pi$ this is~\cite{Iwaniec}*{Theorem~2} and for the symmetric square it follows from a similar argument as~\cite{HL}*{Lemma~3.2}).
   Thus if $L'(1+it,\Pi)  \ll (\mq(|t|+3))^{\epsilon}$ we are done.
   On the other hand if $L'(1+it,\Pi) \gg (\mq(|t|+3))^{\epsilon}$ we switch gear and use the bound 
   $$\frac{L'}{L}(1+it,\Pi) \ll \log^2(\mq(|t|+3)),$$
   which follows from \cite{IK}*{Proposition~5.7} together with the zero-free region and absence of Siegel zero.
   \end{proof}

\subsection{Bounding the continuous spectrum}
In this section we prove a stronger version 
of~(\ref{ToShow}), taking for granted the calculation of the local integrals, that we defer to Sections~\ref{ArchLocInt} and~\ref{Boring}.
If $(\alpha_p)_{p \mid N}$ is a family of complex numbers 
indexed by the prime divisors of $N$, we
define $N^\alpha=\prod_p p^{\alpha_p}$.
\subsubsection{The Borel spectrum}
Here we prove the following
\begin{proposition}\label{B1}
We have $$ \Vol(B(N))\int_{\Pi_{\text{B}}} w_{B(N)}(\pi) \, d\pi \ll_\epsilon N^{-3+\epsilon}$$
for all $\epsilon>0$.
 \end{proposition}
 By the discussion in Section~\ref{GenF} the quantity we have to bound is given by
 $$\Sigma_B= \sum_{\xx \in X_{B}} \frac1{[P \cap K:P_\xx]}
    \int_{i\R \times i\R}h(\nu) 
   \left| W(E(\cdot, \phi^{(\xx,1)},\nu ))(\tt)\right|^2 \, d\nu.$$
    By~\cite{Shahidi}*{Theorem~7.1.2} and the discussion in Section~\ref{DiscussWC} we have 
$$
W(E(\cdot, \phi^{(\xx,1)},\nu ))(\tt)=
\frac{ J(\tt,\nu,\xx)}{\zeta^{N}(1+\nu_1)
    \zeta^{N}(1+\nu_2)\zeta^{N}(1+\nu_1+\nu_2)\zeta^N(1+\nu_1-\nu_2)},$$
where 
$$J(\tt,\nu,\xx)= J_\infty(1,\nu)(\tt)
\prod_{p \mid N}J_p(\phi^{(\xx,1)},\nu).$$
Here each factor $J_v(\phi^{(\xx,1)},\nu)$
is defined by analytic continuation of the corresponding local integral.
For $\xx=\prod_p \xx_p \in X_P=\prod_p X_B(p)$ define $\alpha(\xx)=(\alpha(\xx_p))_{p \mid N}$, where
\begin{equation*}
    \begin{aligned}
        \alpha(1) =2, \quad & \alpha(s_1) =2 ,
        \quad & \alpha(s_2)=1, \quad & \alpha(s_1s_2) =1 ,\\  
        \alpha(s_2s_1) =1 , \quad& \alpha(s_1s_2s_1) =1, \quad & \alpha(s_2s_1s_2)=1, \quad & \alpha(\JJ) =0.
    \end{aligned}
\end{equation*}
 The key input is the following proposition. 
 \begin{proposition}\label{BorelBound}
    For all $\xx \in X_P$  we have
    $$\prod_{p \mid N} J_v(\phi^{(s,1)},\nu) \ll_\epsilon N^{-\alpha(\xx)+\epsilon}.$$
 \end{proposition}
 \begin{proof}
     This follows directly from  (analytic continuation of) the local calculations in Section~\ref{BorLocInt}.
 \end{proof}
 Using Lemma~\ref{BoundJacquet} to bound  $J_\infty(1,\nu)(\tt)$ as well as~(\ref{Poussin}) and the fact that $h$ is Paley-Wiener and hence has rapid decay we obtain
 $$\Sigma_B \ll_{\tt,\epsilon} \sum_{\xx \in X_B} \frac{N^{-2\alpha(\xx)+\epsilon}}{[P \cap K: P_\xx]}. $$
 Now for each $\xx \in X_B$ define $\ell(\xx)=(\ell(\xx_p))_{p \mid N}$
 and $\ell(s)$ is the length of $s$ when $s$ is an element of the Weyl group.
By Lemma~\ref{BorelStabilizers}, we have $[P \cap K:P_\xx]=N^{\ell(\xx)}$ for all $\xx \in X_P$. Observing that for each $s \in W$ we have $2\alpha(s)+\ell(s) \ge 3$
proves Proposition~\ref{B1}.

\subsubsection{The maximal parabolic spectrum}
In this section we let $P=\Ps$ or $P=\Pk$ and we prove the following.
\begin{proposition}\label{B2}
We have $$ \Vol(B(N))\int_{\Pi_P} w_{B(N)}(\pi) \, d\pi \ll_\epsilon N^{-2+\epsilon}$$
for all $\epsilon>0$.
 \end{proposition}
By the discussion in Section~\ref{GenF} the quantity we have to
bound is 
$$\Sigma_P=\sum_{\xx \in X_P} \frac1{[P \cap K:P_\xx]} \int_{i\R}
\sum_\sigma \sum_{\varphi \in \B^\xx_\sigma} h (\mz_P(\nu)+\ma_P(\nu_\sigma))\left|W(E(\cdot, \phi^{(\xx,\varphi)}, \nu))(\tt) \right|^2 \, d\nu.$$
Here by abuse of notation we have denoted by $\nu_\sigma$ the spectral parameter of the $GL_2$ component of the representation $\sigma$.
We may, and we shall, assume that each of the basis $\B_\sigma^\xx$ consists of factorizable vectors.
Then by~\cite{Shahidi}*{Theorem~7.1.2} and the discussion in Section~\ref{DiscussWC} we have 
$$
W(E(\cdot, \phi^{(\xx,\varphi)},\nu ))(\tt)=c_{\varphi}
\frac{J(\tt,\varphi,\nu,\xx)}{L^N(1+\nu,\sigma,P)},$$
where $c_{\varphi}$ is the complex number such that 
$$W(\varphi)(\tt)=c_\varphi \W_\infty(\varphi)(\tt) \prod_p \W_p(\varphi)(1)=\W_\infty(\varphi)(\tt)A_\varphi(1),$$
 the $L$-function $L(s,\sigma,P)$ is given by $\zeta(s)L(s,\sigma)$ if $P=\Ps$ and by $L(s,\sigma,\sym^2)$ if $P=\Pk$,
and
$$J(\tt,\varphi,\nu,\xx)=  J_\infty(\varphi,\nu)(\tt)
\prod_{p \mid N}J_p(\phi^{(\xx,\varphi)},\nu).$$
Here each factor $J_v(\phi^{(\xx,\varphi)},\nu)$
is defined by analytic continuation of the corresponding local integral.
Therefore, after changing summation order
we have
\begin{equation}\label{SigmaSiegel}
    \Sigma_{P}= \int_{i\R}
\sum_{\xx \in X_P} \frac{1}{{[P \cap K: P_\xx]}}
\sum_{\sigma} \frac{h (\mz_P(\nu)+\ma_P(\nu_\sigma))}{|L^N(1+\nu,\pi,P)|^2} \sum_{\varphi \in \B_\sigma^\xx}
|c_\varphi|^2|J(\tt,\varphi,\nu,\xx)|^2 \, d\nu.
\end{equation}
For $\xx=\prod_p \xx_p \in X_P=\prod_p X_P(p)$ define $\alpha(\xx)=(\alpha(\xx_p))_{p \mid N}$ where
\begin{equation*}
    \begin{aligned}
    \alpha(s_2)=\tfrac12, \quad &
    \alpha(s_2s_1)=\tfrac12, \quad & \alpha(s_2s_1s_2)=0\\
     \alpha(s_1)=1, \quad &
    \alpha(s_1s_2)=0, \quad & \alpha(s_1s_2s_1)=0\\
    \end{aligned}
    \end{equation*}
    and $$\alpha(1) = \begin{cases}
        \tfrac32 & \text{ if } P=\Ps,\\
        1 & \text{ if } P=\Pk.
    \end{cases} $$
    
\begin{proposition}\label{SiegelBound}
    For all $\xx \in X_P$ there exists $\ggg_\xx \in \GL_2(\A_{\text{fin}})$ such that
    $$|c_\varphi|^2\prod_{p \mid N}|J_v(\phi^{(\xx,\varphi)},\nu)|^2 \ll N^{-2\alpha(\xx)}|A_\varphi(\ggg_\xx)|^2 $$
\end{proposition}
\begin{proof}
This follows directly from (analytic continuation of) the local calculations from Sections~\ref{Klingen} and~\ref{Siegel}.
\end{proof}
Using that $h$ is Paley-Wiener %and $\nu, \mu_1(\nu_\pi)$ range incomplementary subspaces of $\Lie{\a}_\C^*$, 
and that $\nu_\sigma \in i\R \cup (-\frac12+\theta,\frac12-\theta)$ for some absolute $\theta >0$ we can bound 
$ h(\nu+\mu_1(\nu_\pi)) \ll (1+|\nu|)^{-10}(1+\nu_\sigma^4)^{-3}$ and $|\Gamma\left(\frac12+\nu_\sigma\right)|^2\ll \frac1{\cosh(\pi |\nu_\sigma|)}$
Thus by Lemmas~\ref{BoundJacquet},~\ref{ExtraGamma},~\ref{SiegelBound},~\ref{1line} and~(\ref{Poussin})
we can bound the $\sigma$-sum in~(\ref{SigmaSiegel}) by
$$N^{-2\alpha(s)+\epsilon}(1+|\nu|)^{-5}
\sum_\pi\frac{(1+\nu_\sigma^{4})^{-1}}{\cosh(\pi |\nu_\sigma|)}
\sum_{\varphi \in \B_\sigma^\xx} |A_\varphi(\ggg_\xx)|^2.
$$
Applying the $\GL_2$-Kuznetsov formula~\cite{Iw}*{Theorem~9.3} for $\Gamma_\xx=\mf_2(M(\xx))$ at some cusp that depends on $\ggg_\xx$, we deduce that the $\sigma$-sum is
$$ \ll_\epsilon N^{-2\alpha(\xx)+\epsilon}(1+|\nu|)^{-5} \frac{[\GL_2(\hat{\Z}):\Gamma_\xx]}{{[P \cap K: P_\xx]}}.$$
The implied constant is independent of $\ggg_\xx$, as one can see from the normalisation~(\ref{NormalisationFourierCoefficients}).
Observe that $[\GL_2(\hat{\Z}):\Gamma_\xx]=[M(\hat{\Z}):M(\xx)]$
and thus by Lemma~\ref{grouptheoretic}, previous display 
equals $$\frac{N^{-2\alpha(\xx)+\epsilon}}{[N \cap K:N_\xx]}(1+|\nu|)^{-5}.$$
Now by the explicit calculation of Lemma~\ref{SiegelStabilizers}, we have for each $\xx \in X_P$
$$\frac{N^{-2\alpha(\xx)+\epsilon}}{[N \cap K:N_\xx]} \ll N^{-2+\epsilon}.$$
Since $(1+|\nu|)^{-5}$ is integrable, this finishes the proof
of Proposition~\ref{B2}.

\section{The Archimedean local integrals}\label{ArchLocInt}

\subsection{The Jacquet integral}
Consider the Jacquet integral, defined for $\Re(\nu)$ large enough by
$$J_\nu(\ggg)=\int_{U(\R)}I_{\rho+\nu}(\JJ \uu \ggg)\psi(\uu)\, d\uu.$$
It will be convenient to introduce new coordinates
$$\mu_1=\frac{\nu_1+\nu_2}2,\quad\mu_2=\frac{\nu_1-\nu_2}2.$$
We record the following estimate~\cite{HM}*{Proposition~9}, for Bessel $K$-functions. For all $\sigma>0$ and
for all $\epsilon >0$  for $|\Re(s)| \le \sigma$ we have
\begin{equation}\label{EstimateBessel}
K_s(x) \ll_\epsilon 
\begin{cases}
    \left(\frac{1+|\Im(s)|}{x}\right)^{\sigma+\epsilon}\exp(-\frac\pi2|\Im(s)|) & \text{ if } 0 \le x \le 1+\frac\pi2 |\Im(s)|,\\
    \exp(-x)x^{-\tfrac12} &\text{ if } x> 1+\frac\pi2 |\Im(s)|.
\end{cases}
\end{equation}

\begin{lemma}\label{BoundJacquet}
Assume $|\Re(\nus)|,|\Re(\nuk)|\le \tfrac12$. Then we have
$$J_\nu(\ggg) \ll_\ggg (1+|\Im (\nus)|)^{\tfrac12+\epsilon}
(1+|\Im (\nuk)|)^{\tfrac12+\epsilon}(1+|\Im(\nus)|+|\Im(\nuk)|).$$
% The Jacquet integral has analytic continuation given by
% \begin{equation}\label{Jacquet}
%     J_\nu(y)=\frac{16\pi^2 y_1^2y_2^{\frac32}}{\Gamma\left(\frac{\nu_1+1}2\right)\Gamma\left(\frac{\nu_2+1}2\right)\Gamma\left(\frac{\nu_1+\nu_2+1}2\right)\Gamma\left(\frac{\nu_1-\nu_2+1}2\right)}\int_0^\infty \int_0^\infty
% W_{\frac{\nu_1-\nu_2}2}\left(\frac{t_1}{t_2}\right)\omega_\nu(y,t)\frac{dt_1}{t_1}\frac{dt_2}{t_2},
% \end{equation}
% where $\omega_\nu(y,t)$ satisfies 
% \begin{equation}\label{weight}
%     \int_0^\infty \int_0^\infty
% \omega_\nu(y,t)\frac{dt_1}{t_1}\frac{dt_2}{t_2} \ll_{\nu,y} \exp\left(-\frac{\pi}2|\nu_1+\nu_2|\right).
% \end{equation}
% when $\nu_1+\nu_2 \in i\R$. Here the implicit constant is at most polynomial in $\nu$.
\end{lemma}
\begin{proof}
By right-$K_\infty$ invariance and left-$\psi$-invariance, it is sufficient to prove the lemma for 
$\ggg=\left[\begin{smallmatrix}
 y_1y_2^{\frac12}&   &  &   \\
  &  y_2^{\frac12} &  &  \\
  &   &  y_1^{-1}y_2^{-\frac12} & \\
  &  &   &y_2^{-\frac12}
\end{smallmatrix}\right]$
where $y_1,y_2>0$.
By~\cite{Ishii}*{Theorem~3.2}, for such $\ggg$ we have
for all $\nu \in \C^2$ the integral representation
$$
J_\nu(\ggg)=\frac{16\pi^2 y_1^2y_2^{\frac32}}{\Gamma\left(\frac{\nuk+\nus+1}2\right)\Gamma\left(\frac{\nuk-\nus+1}2\right)\Gamma\left(\nuk+\frac{1}2\right)\Gamma\left(\nus+\frac{1}2\right)}I_\nu(\ggg),$$
where
\begin{equation}\label{Jacquet}
    I_\nu(\ggg)=\int_0^\infty \int_0^\infty
K_{\nuk}\left(2\pi\frac{t_1}{t_2}\right)
K_{\nus}(2\pi t_1t_2)\exp\left(-\pi\left(\frac{y_1^2y_2}{t_1^2}+\frac{t_1^2}{y_2}+\frac{y_2}{t_2^2}+y_2t_2^2\right)\right)\frac{dt_1}{t_1}\frac{dt_2}{t_2}.
\end{equation}
%  First change variable $u=\frac{t_1}{t_2}$
% and $z=t_1^{-2}$, getting
% $$2K_{\nuk}\left(2\pi u\right)
% K_{\nus}(2\pi \frac1{uz})\exp\left(-\pi\left({y_1^2y_2}z+\frac{1}{zy_2}+{y_2}zu^2+\frac{y_2}{zu^2}\right)\right)\frac{du}{u}\frac{dz}{z}.$$
% Next change variables $u = y_1t$, getting
% $$2K_{\nuk}\left(2\pi y_1 t\right)
% K_{\nus}(2\pi \frac1{y_1 tz})\exp\left(-\pi\left({y_1^2y_2}z(1+t^2)+\frac1{zy_2y_1^2t^2}(y_1^2t^2+{y_2^2})\right)\right)\frac{dt}{t}\frac{dz}{z}.$$
% Finally change variables $2\pi y_1^2 y_2 z(1+t^2)=x,$ getting
% $$2K_{\nuk}\left(2\pi y_1 t\right)
% K_{\nus}\left(4\pi^2 y_1y_2(t+\frac1t)\frac1x \right)\exp\left(-\frac{x}2-2\pi^2(1+\frac1{t^2})(y_1^2t^2+{y_2^2})\frac1x\right)\frac{dx}{x}\frac{dt}{t}.$$
% Then $t=y_1^{-1}u=y_1^{-1}t_1t_2^{-1}$
% and $x=2\pi y_1^2y_2(1+t^2)t_1^{-2}=2\pi y_1^2y_2(t_1^{-2}+y_1^{-2}u^2t_1^{-2})=2\pi y_1^2y_2(t_1^{-2}+y_1^{-2}t_2^{-2})$
% 
Changing variables $t=y_1^{-1}t_1t_2^{-1}$ and $x=2\pi y_1^2y_2\left(t_1^{-2}+y_1^{-2}{t_2^{-2}}\right)$, we obtain
\begin{align*}
    I_\nu(\ggg)=2&\int_0^\infty K_{\nuk}\left(2\pi y_1t\right) \\
&\int_0^\infty K_{\nus}\left(4\pi^2 y_1y_2\left(t+\frac1{t}\right)\frac1x\right)
\exp\left(-\frac{x}2-2\pi^2\left(1+\frac1{t^2}\right)\left(y_1^2t^2+y_2^2\right)\frac1x\right) \frac{dx}{x}\frac{dt}{t}.
\end{align*}
Note that 
$$4\pi^2 y_1y_2\left(t+\frac1{t}\right)=zw \text{ and } 
2\pi^2\left(1+\frac1{t^2}\right)\left(y_1^2t^2+y_2^2\right)=\tfrac12(z^2+w^2),$$
where $$z=2\pi{y_1}\sqrt{1+t^{2}}, \quad w=2\pi y_2\sqrt{1+t^{-2}}.$$
By Formula~6.653~2. in~\cite{GR}, the $x$-integral is given by $$2K_{\nus}\left(z\right)K_{\nus}\left(w\right),$$ and so
$$I_\nu(\ggg)=-4\int_0^\infty K_{\nuk}\left(2\pi y_1t\right)K_{\nus}\left(2\pi y_1\sqrt{1+t^2}\right)K_{\nus}\left(2\pi y_2 \sqrt{1+t^{-2}}\right) \frac{dt}{t}.$$
By~(\ref{EstimateBessel}) in the range
$$R_0=\{t>0 : 2\pi y_1 t \le 1+\frac\pi2 |\Im(\nuk)| \text{ and }
w > 1+\frac{\pi}2 |\Im(\nus)|\}$$
the integrand is 
$$ \ll_\ggg X^\epsilon (1+|\Im \nus|)^{\tfrac12}(1+|\Im \nuk|)^{\tfrac12}t^{-\frac32}\exp\left(-\frac\pi2(|\Im(\nus)|+|\Im(\nuk)|)\right)\exp(-w),$$
for all $\epsilon>0$, where we have set $X=(1+|\Im (\nuk)|)(1+|\Im (\nus)|)(t+t^{-1})$.
Let $$t_0=\sup R_0 \ge \min 
\left\{\frac1{2 \pi y_1}\left(1+\frac\pi2|\Im(\nuk)|\right),
2\pi y_2\left(1+\frac\pi2|\Im(\nus)|\right)^{-1}\right\}.$$
Then for $t<t_0$ we have 
$$\sqrt{1+t^{-2}}-\sqrt{1+t_0^{-2}}=\int_{t_0^{-1}}^{t^{-1}} \frac{u}{\sqrt{1+u^2}} \, du \ge \int_{t_0^{-1}}^{t^{-1}} \frac{u}{1+u} \, du =(t^{-1}-t_0^{-1})-\log\left(\frac{t^{-1}+1}{t_0^{-1}+1}\right),$$
and thus 
$$\exp(-w) \le 
\left(\frac{t^{-1}+1}{t_0^{-1}+1}\right)^{2\pi y_2}
\exp\left(-1-\frac{\pi}2|\Im(\nus)|\right)
\exp(-2\pi y_2 (t^{-1}-t_0^{-1})).$$
Changing variable $u=t^{-1}-t_0^{-1}$ and observing
$\frac{t^{-1}+1}{t_0^{-1}+1}=1+\frac{u}{t_0^{-1}+1}\le u$
and $(u+t_0^{-1})^{\tfrac12} \ll_\ggg u^{\tfrac12} + \left(1+\frac\pi2|\Im(\nus)|\right)^{\tfrac12} $
we obtain
$$\int_{R_0}  \ll_\ggg (1+|\Im (\nuk)|)^{\tfrac12+\epsilon}
(1+|\Im (\nus)|)^{1+\epsilon}\exp\left(-\frac\pi2(2|\Im(\nus)|+|\Im(\nuk)|)\right).% \int_{0}^\infty
%u^{2\pi y_2}(u^{\tfrac12}+1) \exp(-2\pi y_2 u) \, du
$$
In the range $t > 1+\frac\pi2|\Im(\nuk)|$, a similar reasoning holds except we bound the integrand by
$$\ll_\ggg X^\epsilon t^{-\frac32} (1+|\Im(\nus)|) \exp\left(-\pi|\Im(\nus)|\right)\exp(-t),$$
leading to the bound 
$$\ll_\ggg (1+|\Im (\nuk)|)^{\epsilon}
(1+|\Im (\nus)|)^{1+\epsilon} \exp\left(-\frac\pi2(2|\Im(\nus)|+\Im(\nuk)\right)$$ for the integral.
Finally, the complementary range 
$$\{t>0: 2\pi y_1 t \le 1+\frac\pi2|\Im(\nuk)| \text{ and } w \le 1+\frac{\pi}2|\Im(s)| \}$$
is contained in the (possibly empty) interval
$\left[2\pi y_2\left(1+\frac\pi2|\Im(\nus)|\right)^{-1},\frac1{2 \pi y_1}\left(1+\frac\pi2|\Im(\nuk)|\right)
\right]$. In this range, we bound the integrand by
$$\ll_\ggg X^\epsilon (1+|\Im \nuk|)^{\tfrac12}
(1+|\Im \nus|)t^{-\frac32}\exp\left(-\frac\pi2(2|\Im(\nus)|+|\Im(\nuk)|)\right),$$
obtaining the bound
$$\ll_\ggg  (1+|\Im (\nuk)|)^{\tfrac12+\epsilon}
(1+|\Im (\nus)|)^{\frac32+\epsilon} \exp\left(-\frac\pi2(2|\Im(\nus)|+|\Im(\nuk)|\right)$$
for the integral. 
Combining the bounds in the three ranges, we conclude
\begin{equation}\label{bound1}
    I_\nu(y) \ll_\ggg  (1+|\Im (\nuk)|)^{\tfrac12+\epsilon}
(1+|\Im (\nus)|)^{\frac32+\epsilon} \exp\left(-\frac\pi2(2|\Im(\nus)|+|\Im(\nuk)|\right).
\end{equation}
Next, observe that changing variables $t_2 \mapsto t_2^{-1}$
has the effect of swapping $\nuk$ and $\nus$ in the subsequent argument. Therefore we obtain the bound
\begin{equation}\label{bound2}
    I_\nu(y) \ll_\ggg  (1+|\Im (\nus)|)^{\tfrac12+\epsilon}
(1+|\Im (\nuk)|)^{\frac32+\epsilon} \exp\left(-\frac\pi2(2|\Im(\nuk)|+|\Im(\nus)|\right).
\end{equation}
Finally, using Stirling's formula and~(\ref{bound1})
we conclude $$J_\nu(\ggg)\ll_\ggg (1+|\Im (\nuk)|)^{\tfrac12+\epsilon}
(1+|\Im (\nus)|)^{\frac32+\epsilon} \exp\left(\frac\pi2\left(\frac{|\Im(\nuk+\nus)|}{2}+\frac{|\Im(\nuk-\nus)|}{2}-|\Im(\nus)|\right)\right).$$
Observe that if then $\Im(\nuk+\nus)$ and $\Im(\nuk-\nus)$
have opposite signs then $\frac{|\Im(\nuk+\nus)|}{2}+\frac{|\Im(\nuk-\nus)|}{2}-|\Im(\nus)|=0$.
On the other hand, combining Stirling's formula and~(\ref{bound2}) we have 
$$J_\nu(\ggg)\ll_\ggg (1+|\Im (\nus)|)^{\tfrac12+\epsilon}
(1+|\Im (\nuk)|)^{\frac32+\epsilon} \exp\left(\frac\pi2\left(\frac{|\Im(\nuk+\nus)|}{2}+\frac{|\Im(\nuk-\nus)|}{2}-|\Im(\nuk)|\right)\right),$$
and if $\Im(\nuk+\nus)$ and $\Im(\nuk-\nus)$
have the same sign then $\frac{|\Im(\nuk+\nus)|}{2}+\frac{|\Im(\nuk-\nus)|}{2}-|\Im(\nuk)|=0$.
Therefore in both case we obtain the desired bound.
\end{proof}

\subsection{Relation between the local integral and the Jacquet integral}\label{transitivity}
Let $P=\Pk$ or $P=\Ps$ and consider 
$\pi_\infty=\Ind_{P(\R)}^{\GSp_4(\R)}(1_{N_P} \otimes I_{P,\rho_P+\nu} \otimes \sigma_\infty)$ where $\sigma_\infty=1\otimes \tau_\infty$
and $\tau_\infty$ is a principal series representation of $\GL_2(\R)$ with spectral parameter $\nu_\tau$.
For $\ggg=\pp\kk \in \GSp_4(\R)$ with $\pp \in P(\R)$
and $\kk \in K_\infty$, abusing slightly notation define $\mf_2(\ggg)=\mf_2(\pp)$ (this is only well-defined up to right multiplication by $M \cap K_\infty$).
Let $\varphi \in \tau$ be a non-zero $\rm{O}_2(\R)$-fixed vector and define $\phi(\ggg)=\varphi(\mf_2(\ggg))$
(this is well-defined since $\mf_2(M \cap K_\infty)=\rm{O}_2(\R)$).
Both the local integral~(\ref{TheLocalIntegral}) and the Jacquet integral~$J_{\nu_\pi}$ are $K_\infty$-invariant
fixed vectors in the Whittaker model of~$\pi_\infty$, thus they must be proportional to each other. However we want to determine the constant of proportionality.
To this end we use a ``transitivity property" of the Jacquet integrals.
Recall notations from Section~\ref{DiscussWC}, and in particular our choice~(\ref{GL2Whittaker}) for the Whittaker function on $\GL_2(\R)$.
\begin{lemma}\label{ExtraGamma}
Assume $\varphi \in \tau_\infty^{K_\infty}$ is such that $\W_\infty(\varphi)=\W_{\tau_\infty}$.
    Then we have $$J_\infty(\phi,\nu_\pi)=\frac1{2\pi^{\tfrac{1}2+\nu_\tau}}\Gamma\left(\tfrac{1}2+\nu_\tau\right)J_{\nu_\pi}.$$
\end{lemma}
\begin{proof}
    Since both sides are analytic in $\nu$ and $\nu_\tau$, it suffices to prove the lemma for $\Re(\nu), \Re(\nu_\tau)$ large enough.
    Recall that by~\cite{Gold}*{Proposition~3.4.6} we have for all $\ggg \in \GL_2(\R)$
    $$\W_{\tau_\infty}\left(\ggg\right)=
\frac1{2\pi^{\tfrac{1}2+\nu_\tau}}\Gamma\left(\tfrac{1}2+\nu_\tau\right)
\int_{-\infty}^{\infty}
I_{\GL_2,\frac12+\nu_\tau}\left(\mat{}{-1}{1}{}\mat{1}{x}{}{1}\ggg\right)e(-x) \, dx,$$
where $I_{\GL_2,\frac12+\nu_\tau}$ is the left-$\mat{1}{*}{}{1}$, right-$\rm{O}_2(\R)$-invariant function that is given on the diagonal by
$$I_{\GL_2,\frac12+\nu_\tau}\left(\mat{y_1}{}{}{y_2}\right)=(y_1y_2^{-1})^{\frac12+\nu_\tau}.$$
Note that for all $\mm \in M_P(\R)$ and for all $\nu \in \C$ we have
$$I_{\GL_2,\nu}(\mf_2(\mm))=I_{B,\ma_P(\nu)}(\mm).$$
Observing that for all $\ggg \in \GSp_4(\R)$ we have
$$\W_\infty(\phi_{\ggg})(1)=
\W_{\tau_\infty}(\mf_2(\ggg))$$
(which is well defined since $\W_{\tau_\infty}$ is right-$(M \cap K_\infty)$-invariant), 
we obtain  by definition 
\begin{equation*}
    \begin{split}
        J_\infty(\phi,\nu_\pi)(\ggg)&=\frac1{2\pi^{\tfrac{1}2+\nu_\tau}}\Gamma\left(\tfrac{1}2+\nu_\tau\right) \times \cdots\\
\cdots&\times \int_{N_P(\R)}
\int_{-\infty}^{\infty}
I_{P,\rho_P+\nu}(\sigma_P \nn \ggg) 
I_{\GL_2,\frac12+\nu_\tau}\left(\mat{}{-1}{1}{}\mat{1}{x}{}{1}\mf_2(\sigma_P\nn\ggg)\right)e(-x) 
\overline{\psi_\infty}(\nn)\, dxd\nn.
    \end{split} 
\end{equation*}
Let $\uu(x)$ be the element of $M(\R) \cap U(\R)$ is the such that $\mf_2(\uu)=\mat{1}{x}{}{1}$, and observe that we have 
\begin{equation*}
    \begin{split}
        I_{P,\rho_P+\nu}(\sigma_P \nn \ggg) 
I_{\GL_2,\frac12+\nu_\tau}\left(\mat{}{-1}{1}{}\mat{1}{x}{}{1}\mf_2(\sigma_P\nn\ggg)\right)&=
I_{P,\rho_P+\nu}(\JJ \uu(x) \nn \ggg) 
I_{\GL_2,\frac12+\nu_\tau}\left(\mf_2(\JJ\uu(x)\nn\ggg)\right)\\
&=I_{B,\mz_P(\rho_P+\nu)}(\JJ \uu(x) \nn \ggg) 
I_{B,\ma_P(\frac12+\nu_\tau)}\left(\JJ\uu(x)\nn\ggg\right)\\
&=I_{B,\nu_\pi}(\JJ \uu(x) \nn \ggg).
    \end{split}
\end{equation*}
Changing variables $\uu(x)\nn=\uu \in U(\R)$ gives the result.
\end{proof}

\section{Ramified local integrals}\label{Boring}
In this section, we compute the local integrals~(\ref{TheLocalIntegral})
at ramified primes when $\phi=\phi^{(\xx,\varphi)}$ is an element of the
orthonormal basis described in Section~\ref{GenF}.
\subsection{Klingen subgroup} \label{Klingen} 
When $P$ is the Klingen parabolic subgroup 
 the element $\sigma=\sigma_P$  of $(W \cap M) \JJ$ such that 
$\sigma_P (U \cap M) \sigma_P^{-1}=U \cap M$ is given by $\sigma_P=s_1s_2s_1$.
\subsubsection{The local integral for $s=1$} 
\begin{lemma}\label{IwasawaKlingen1}
Let $\nn=\left[\begin{smallmatrix}
 1& x  & z  &  y  \\
  &  1 & y &  \\
  &   & 1 & \\
  &  &  -x &1
\end{smallmatrix}\right]  \in \Nk$. Then $\sigma \nn \in \Pk B(p)$ if and only if 
\begin{equation}\label{Klingencondition1}
    v_p(z)< \min\{0, v_p(x), v_p(y)\}. 
\end{equation}
Moreover, if $\sigma \nn =\pp  \gamma$ for some $\gamma \in B(p)$ then $\mf_2(\pp) \in \mat{ -1+\frac{xy}{z}}{-\frac{y^2}z}{-\frac{x^2}z}{1+\frac{xy}{z}} \Gamma_0(p) $
and $\mf_1(\pp) \in z^{-1} \o^\times$.
\end{lemma}
% \begin{remark}\label{Rk1}
%     Note that condition~(\ref{s2condition}) implies in particular
%     $\clubsuit \spadesuit \clubsuit$.
% \end{remark}
\begin{proof}
Assume $\sigma \nn \in \pp B(p)$ for some $\pp \in \Pk$. Equivalently, since $s_2 \sigma =\JJ$
\begin{equation*}
 \pp^{-1}s_2^{-1} \JJ \nn \JJ^{-1} \in B(p) \JJ^{-1}=\left[\begin{smallmatrix}
 *& *  &  \o^\times & *   \\
 * &  * & \p & \o^\times \\
 \o^\times &  \p & \p & \p\\
 * & \o^\times & \p&\p
\end{smallmatrix}\right] \cap G(\o).
\end{equation*}
Write $\pp^{-1}s_2^{-1}=\left[\begin{smallmatrix}
 (ad-bc)t^{-1}&   &   &    \\
  &  a &  & b \\
  &   & t & \\
  & c & & d
\end{smallmatrix}\right] \left[\begin{smallmatrix}
 1&  x_1 & z_1  & y_1   \\
  &  1 & y_1 &  \\
  &   & 1 & \\
  &  &  -x_1  &1
\end{smallmatrix}\right] $.
%We have $\mu(pp^{-1}s_2^{-1} \JJ \nn \JJ^{-1})=ad-bc$, hence we deduce
%$$ad-bc \in \Z_p^\times.$$ Furthermore 
The third line of $\pp^{-1}s_2^{-1} \JJ \nn \JJ^{-1}$ is
$$\left[\begin{smallmatrix}
    -tz & -ty & t & tx
\end{smallmatrix}\right],$$
from which we deduce
$$v_p(z)=-v_p(t)<0 \quad \text{ and } \quad v_p(y), v_p(x)>v_p(z).$$ 
Conversely, assuming~(\ref{Klingencondition1}) and choosing
$z_1=t=\frac1z,$ $x_1=\frac{y}z, y_1=\frac{-x}z$ and $\mat{a}{b}{c}{d}=\frac1z\mat{-y^2}{z+xy}{z-xy}{x^2}$,
we have
$$\pp^{-1}s_2^{-1} \JJ \nn \JJ^{-1}= \left[\begin{smallmatrix}
    &          & -1       &    \\
    &          & -\frac{y}z & 1 \\
  -1&-\frac{y}{z}  & \frac{1}z   & \frac{x}z \\
    &  1       &  -\frac{x}z&
\end{smallmatrix}\right] \in B(p) \JJ^{-1},$$
and furthermore
 $\pp \in \left[\begin{smallmatrix}
 -\frac1z&   &   &    \\
  &  -1+\frac{xy}{z} &  & -\frac{y^2}z \\
  &   & z & \\
  &  -\frac{x^2}z& & 1+\frac{xy}{z}
\end{smallmatrix}\right] \Nk$ hence $\mf_1(\pp)=\frac1z$
and $\mf_2(\pp)=\mat{ -1+\frac{xy}{z}}{-\frac{y^2}z}{-\frac{x^2}z}{1+\frac{xy}{z}}$.
Together with the calculation of $P_s$ in Lemma~\ref{KlingenStabilizers}, this finishes the proof.
\end{proof}

By Lemma~\ref{IwasawaKlingen1}, the local integral we need to calculate is over matrices $\nn=\left[\begin{smallmatrix}
 1& x  & z  &  y  \\
  &  1 & y &  \\
  &   & 1 & \\
  &  &  -x &1
\end{smallmatrix}\right] $ where $(x,y,z)$ belong to
$$R=\{(x,y,z) \in \Q_p^3 :\text{ (\ref{Klingencondition1}) holds}\}.$$
In the course of the proof of Lemma~\ref{LocalIntegral1}
below, we shall split $R$ into the following sub-ranges 
 \begin{equation*}
        \begin{split}
    R_1&=\{v_p(z+xy)<v_p(x^2)  \text{ and~(\ref{Klingencondition1}) holds}\},\\
    R_2&=\{v_p(z+xy) \ge v_p(x^2)  \text{ and~(\ref{Klingencondition1}) holds}\},\\
    R'_{1}&=R_1 \cap \{v_p(x) \ge -1 \text{ and } v_p(y^2) \ge v_p(z+xy)-1\},\\
   R'_{11}&=R'_1 \cap \{v_p(z)<-1\},\\
   R'_{12}&=R'_1 \cap \{v_p(z) = -1\},\\
     R'_{2}&=R_2 \cap \{v_p(x) \ge -1 \text{ and } v_p(z-xy) \ge v_p(x^2)-1, v_p(z) \ge v_p(x^2)\}.
        \end{split}
\end{equation*}

\begin{lemma}\label{dissect1}
We have $R'_{11}= (C_1 \setminus S_1) \sqcup C_2$
where 
$$C_1=\{(x,y,z) \in \Q_p^3 :v_p(z)<\min\{-1,v_p(x),2v_p(x),2v_p(y)+2\} \text{ and } v_p(x) \ge - 1\},$$
$$C_2=\{(x,y,z) \in R'_{11}: v_p(z+xy)>v_p(z)\},$$
$$S_1=\{(x,y,z) \in C_1 : xy \in z(-1 + \p)\}.$$
\end{lemma}
\begin{proof}
First we show that
\begin{equation}\label{i1}
    R'_{11} \subset C_1.
\end{equation}
Note that~(\ref{i1}) trivially implies
$C_2 = S_1 \cap R'_{11},$
and hence also
$R'_{11} \subset (C_1 \setminus S_1) \sqcup C_2.$
To prove~(\ref{i1}), consider $(x,y,z) \in R'_{11}$.
The inequalities
$$v_p(z)<-1, \quad v_p(z)<v_p(x^2)$$
and 
\begin{equation}\label{vx}
-1 \le v_p(x)
\end{equation}
hold by definition.
Moreover by~(\ref{Klingencondition1}) we have 
\begin{equation}\label{vy}
    v_p(z)+1 \le v_p(y)
\end{equation}
Adding~(\ref{vx}) and~(\ref{vy}) together gives $v_p(z) \le v_p(xy)$, hence $v_p(z) \le v_p(z+xy).$ Since $v(z+xy) < v(x^2),  v_p(y^2)+2$,  we obtain that $(x,y,z) \in C_1$ as claimed.
It remains to show $$(C_1 \setminus S_1) \sqcup C_2 \subset  R'_{11}.$$
Since $C_2 \subset R'_{11}$ by definition, it suffices to show $C_1 \setminus S_1 \subset  R'_{11}.$ So consider $(x,y,z) \in C_1 \setminus S_1$.
By definition we have $v_p(z)<\min\{-1,v_p(x)\}$ and
$v_p(y) \ge  \frac{v_p(z)-1}2 > v_p(z)$ since $v_p(z)<-1$, and hence~(\ref{Klingencondition1}) holds. Furthermore since $(x,y,z) \not \in S_1$ and since $v_p(xy) \ge v_p(z)$, we have 
$v_p(z+xy)=v_p(z)<\min\{v_p(y^2)+2,v_p(x^2)\}$, which finishes the proof.
\end{proof}

\begin{lemma}\label{dissect2}
We have $R'_2= p^{-1}\Z_p^\times \times p^{-1} \Z_p \times p^{-2}\Z_p^\times.$
\end{lemma}
\begin{proof}
    Assume $(x,y,z) \in R'_2$.
    The conditions $v_p(x^2) \le v_p(z)<v_p(x)$ and $v_p(x) \ge -1$ force $v_p(x)=-1$ and $v_p(z)=-2$, and so by~(\ref{Klingencondition1}) we also have $v_p(y) \ge -1$.
    Conversely if $(x,y,z) \in p^{-1}\Z_p^\times \times p^{-1} \Z_p \times p^{-2}\Z_p^\times$ then condition~(\ref{Klingencondition1}) is satisfied, and
    $v_p(z+xy), v_p(z-xy) \ge -2=v_p(x^2)$ and so $(x,y,z) \in R'_2$. 
\end{proof}

\begin{lemma}\label{KLocalIntegral1}
    Let $\varphi$ be a factorizable Maa{\ss} form for $\Gamma_0(p)$ on $\GL_2$.
    Let $\phi \in \H_{\Pk}$ defined by
    $$\phi_{\kk}(\mm)=
    \begin{cases}
       \varphi(\mf_2(m)) \text{ if } \kk=1,\\
        0 \text { if } \kk \in X_{\Pk} \setminus \{1\}.
    \end{cases}$$
    For $\nu \in \C$ with $\Re(\nu)$
     large enough, we have
    \begin{equation*}
        \begin{split}
            J_p(\phi,\nu)=& \quad \W_p(\varphi)(1)\bigl((1-p^{-1})(p^{-1-\nu}+p^{-1-2\nu})-p^{-3-3\nu}\bigr)\\
    &-\W_p(\varphi)\left(\mat{}{1}{1}{}\right)p^{-1-2\nu}(1-p^{-1}).
        \end{split}
    \end{equation*}
\end{lemma}
\begin{proof}
    By Lemma~\ref{IwasawaKlingen1}, the integral is over 
    matrices $\nn=\left[\begin{smallmatrix}
 1& x  & z  &  y  \\
  &  1 & y &  \\
  &   & 1 & \\
  &  &  -x &1
\end{smallmatrix}\right] $ 
     satisfying~(\ref{Klingencondition1}),
    and for such $\nn$ we have 
    $$I_{\Pk,2+\nu}(\sigma \nn)=|z|^{-2-\nu},$$
    and 
    $$\W_p(\phi_{\sigma \nn})(1)=\W_p(\varphi)\left(\mat{xy-z}{-y^2}{-x^2}{xy+z}\right).$$
    Write $J_p(\phi,\nu)=\int_{R_1}+\int_{R_2}.$
    If $v_p(z+xy)<v_p(x^2)$ then we have 
    $$\mat{xy-z}{-y^2}{-x^2}{xy+z}=
    \mat{1}{-\frac{y^2}{z+xy}}{}{1}
    \mat{-\frac{z^2}{z+xy}}{}{}{z+xy}
    \underbrace{\mat{1}{}{-\frac{x^2}{z+xy}}{1}}_{\in \Gamma_0(p)}
    $$
    and thus $$\W_p(\varphi)\left(\mat{xy-z}{-y^2}{-x^2}{xy+z}\right)=\theta\left(-\frac{y^2}{z+xy}\right)\W_p(\varphi)\left(\mat{-z^2/(z+xy)}{}{}{z+xy}\right).$$
    Note that $$|z|^{-2-\nu}\W_p(\varphi)\left(\mat{-z^2/(z+xy)}{}{}{z+xy}\right)$$
    is invariant by the change of variables 
    $(x,y,z) \mapsto (\mu x,\lambda \mu y,\lambda \mu^{2} z)$
    for $\lambda, \mu \in \Z_p^\times$.
    Therefore by Lemma~\ref{GaussTransform} we have 
    $$\int_{R_1}=\int_{R'_1} |z|^{-2-\nu}\W_p(\varphi)\left(\mat{-z^2/(z+xy)}{}{}{z+xy}\right)
    \theta\left(-x-\frac{y^2}{z+xy}\right) \, dx dy dz.$$
    Write $\int_{R'_1}=\int_{R'_{11}}+\int_{R'_{12}}$.
    By Lemma~\ref{dissect1} we have
    $\int_{R'_{11}}=\int_{C_1 \setminus S_1}+\int_{C_2}.$
    For $(x,y,z) \in C_1 \setminus S_1$, as seen in the proof of
    Lemma~\ref{dissect1} we have $v_p(z+xy)=v_p(z)$.
    Thus by the above change of variables and evaluating the Gau{\ss} sums, we have
    \begin{align*}
        \int_{C_1 \setminus S_1}=& \quad
        \W_p(\varphi)(1)p^{-2}\zeta_p(1)^2\int_{C_1 \setminus S_1}
        |z|^{-2-\nu}\1_{v_p(x)=-1}\1_{v_p(y^2)=v_p(z)-1}
        \, dxdydz\\
        &-\W_p(\varphi)(1)p^{-1}\zeta_p(1)\int_{C_1 \setminus S_1}
        |z|^{-2-\nu}\1_{v_p(x)=-1}\1_{v_p(y^2) \ge v_p(z)}
        \, dxdydz\\
        &-\W_p(\varphi)(1)p^{-1}\zeta_p(1)\int_{C_1 \setminus S_1}
        |z|^{-2-\nu}\1_{v_p(x)\ge 0}\1_{v_p(y^2) = v_p(z)-1}
        \, dxdydz\\
        &+\W_p(\varphi)(1)\int_{C_1 \setminus S_1}
        |z|^{-2-\nu}\1_{v_p(x)\ge 0}\1_{v_p(y^2) \ge v_p(z+xy)}
        \, dxdydz.
    \end{align*}
    Denote by $J_i$ for $1 \le i \le 4$ the four integrals in the above display, respectively.
    By definition of $C_1$ and $S_1$ we have
    \begin{align*}
        J_1=& \quad \int_{v_p(y)\le -2}\int_{v_p(z)=2v_p(y)+1}\int_{v_p(x)=-1} |z|^{-2-\nu} \, dx dz dy\\
        &-\int_{v_p(z)=-3}\int_{v_p(x)=-1}\int_{\frac{z}{x}(-1+p\Z_p)} |z|^{-2-\nu} \, dy dx dz\\
        =& \quad (1-p^{-1})^3 p^{-3\nu}(1-p^{-1-2\nu})^{-1}
        -p^{-1-3\nu}(1-p^{-1})^2,\\
\end{align*}
\begin{align*}
        J_2=& \quad \int_{v_p(y)\le -2}\int_{v_p(z) \le 2v_p(y)}\int_{v_p(x)=-1} |z|^{-2-\nu} \, dx dz dy\\
        &+\int_{v_p(y) \ge -1}\int_{v_p(z) \le -3} \int_{v_p(x)=-1} |z|^{-2-\nu} \, dx dz dy\\
        =& \quad (1-p^{-1})^3p^{-1-4\nu}(1-p^{-1-\nu})^{-1}(1-p^{-1-2\nu})^{-1}+
        (1-p^{-1})^2p^{-1-3\nu}(1-p^{-1-\nu})^{-1},\\
\end{align*}
\begin{align*}
            J_3=& \quad \int_{v_p(y)\le -2}\int_{v_p(z)= 2v_p(y)+1}\int_{v_p(x)\ge 0} |z|^{-2-\nu} \, dx dz dy\\
        =& \quad (1-p^{-1})^2p^{-1-3\nu}(1-p^{-1-2\nu})^{-1},\\
\end{align*}
\begin{align*}
        J_4=& \quad \int_{v_p(y)\le -2}\int_{v_p(z)\le  2v_p(y)}\int_{v_p(x)\ge 0} |z|^{-2-\nu} \, dx dz dy\\
        &+\int_{v_p(y)\ge -1}\int_{v_p(z)\le  -2}\int_{v_p(x)\ge 0} |z|^{-2-\nu} \, dx dz dy\\
        =& \quad (1-p^{-1})^2 p^{-2-4\nu}(1-p^{-1-\nu})^{-1}(1-p^{-1-2\nu})^{-1}+(1-p^{-1})p^{-1-2\nu}(1-p^{-1-\nu})^{-1}
    \end{align*}
    In total, this gives 
    $$\int_{C_1 \setminus S_1}=\W_p(\varphi)(1)(1-p^{-1})p^{-1-2\nu}-p^{-3-3\nu}.$$
    Next, observe that if $v_p(a)<0$ then by Lemma~\ref{supportWhittaker} $\W_p(\varphi)\left(\mat{a}{}{}{1}\right)=0$.
    In particular, $\int_{C_2}=0$.
    Finally, for $(x,y,z) \in R'_{12}$ by~(\ref{Klingencondition1}) we have $v_p(x),v_p(y) \ge 0$ and so $v_p(z+xy)=v_p(z)=-1$.
    Therefore we obtain
    \begin{align*}
    \int_{R'_{12}}= \W_p(\varphi)(1) \quad \int_{v_p(z)=-1} \int_{v_p(y) \ge 0}
    \int_{v_p(x) \ge 0}|z|^{-2-\nu} \, dxdydz=  \W_p(\varphi)(1)(1-p^{-1})p^{-1-\nu}
    \end{align*}
    
    On the other hand if $v_p(z+xy) \ge v_p(x^2)$ then we have 
        $$\mat{xy-z}{-y^2}{-x^2}{xy+z}=
    \mat{1}{\frac{z-xy}{x^2}}{}{1}
    \mat{-\frac{z^2}{x^2}}{}{}{-x^2}
    \mat{\phantom{\frac{z}{x^2}}}{1}{1}{\phantom{\frac{z}{x^2}}}
    \underbrace{\mat{1}{-\frac{z+xy}{x^2}}{}{1}}_{\in \Gamma_0(p)}.$$
    Note that $$|z|^{-2-\nu}\W_p(\varphi)\left(\mat{-z^2/x^2}{}{}{-x^2} \mat{\phantom{\frac{z}{x^2}}}{1}{1}{\phantom{\frac{z}{x^2}}}\right)$$
    is invariant by the change of variables 
    $(x,y,z) \mapsto (\mu x,\lambda \mu y,\lambda \mu^{2} z)$
    for $\lambda, \mu \in \Z_p^\times$.
    Furthermore by Lemma~\ref{supportWhittaker} we have 
    $$\W_p(\varphi)\left(\mat{-z^2/x^2}{}{}{-x^2} \mat{\phantom{\frac{z}{x^2}}}{1}{1}{\phantom{\frac{z}{x^2}}}\right)=0$$ unless $v_p(z) \ge v_p(x^2)$.
    Therefore by Lemma~\ref{GaussTransform} we have 
    $$\int_{R_2}=\int_{R'_2} |z|^{-2-\nu}\W_p(\varphi)\left(\mat{-z^2/x^2}{}{}{-x^2} \mat{\phantom{\frac{z}{x^2}}}{1}{1}{\phantom{\frac{z}{x^2}}}\right)
    \theta\left(-x-\frac{z-xy}{x^2}\right) \, dx dy dz.$$
    By Lemma~\ref{dissect2}, since for $(x,y,z) \in R'_2$ we have $v_p(z-xy) \ge v_p(x^2)$, we obtain
    \begin{align*}
        \int_{R'_2}&=-\W_p(\varphi)\left(\mat{}{1}{1}{}\right)p^{-1}\zeta_p(1)
    \int_{v_p(x)=-1}\int_{v_p(y) \ge -1}\int_{v_p(z)=-2}|z|^{-2-\nu} \, dz dy dx\\
    &=-\W_p(\varphi)\left(\mat{}{1}{1}{}\right)p^{-1-2\nu}(1-p^{-1}).
    \end{align*}
     \end{proof}

\subsubsection{The local integral for $s=s_1$} 
\begin{lemma}\label{IwasawaKlingens1}
Let $\nn=\left[\begin{smallmatrix}
 1& x  & z  &  y  \\
  &  1 & y &  \\
  &   & 1 & \\
  &  &  -x &1
\end{smallmatrix}\right]  \in \Nk$. Then $\sigma \nn \in \Pk s_1 B(p)$ if and only if 
\begin{equation}\label{Klingenconditions1}
    v_p(y)< \min\{0, v_p(x), v_p(z)+1\}. 
\end{equation}
Moreover, if $\sigma \nn =\pp s_1 \gamma$ for some $\gamma \in B(p)$ then $\mf_2(\pp) \in \mat{}{-y}{y^{-1}}{x+\frac{z}y} \Gamma_0(p) $
and $\mf_1(\pp) \in y^{-1} \Z_p^\times$.
\end{lemma}
\begin{proof}
Assume $\sigma \nn \in \pp s_1 B(p)$ for some $\pp \in \Pk$. Equivalently, since $s_2 \sigma =\JJ$
\begin{equation}\label{IKs2}
 \pp^{-1}s_2^{-1} \JJ \nn \JJ^{-1} \in s_1 B(p) \JJ^{-1}=\left[\begin{smallmatrix}
 *& *  &  \p & \o^\times   \\
 * &  * & \o^\times & * \\
 * &  \o^\times & \p & \p\\
 \o^\times & \p & \p&\p
\end{smallmatrix}\right] \cap G(\o).
\end{equation}
Write $\pp^{-1}s_2^{-1}=\left[\begin{smallmatrix}
 (ad-bc)t^{-1}&   &   &    \\
  &  a &  & b \\
  &   & t & \\
  & c & & d
\end{smallmatrix}\right] \left[\begin{smallmatrix}
 1&  x_1 & z_1  & y_1   \\
  &  1 & y_1 &  \\
  &   & 1 & \\
  &  &  -x_1  &1
\end{smallmatrix}\right] $.
%We have $\mu(pp^{-1}s_2^{-1} \JJ \nn \JJ^{-1})=ad-bc$, hence we deduce
%$$ad-bc \in \Z_p^\times.$$ Furthermore 
The third line of $\pp^{-1}s_2^{-1} \JJ \nn \JJ^{-1}$ is
$$\left[\begin{smallmatrix}
    -tz & -ty & t & tx
\end{smallmatrix}\right],$$
from which we deduce
$$v_p(y)=-v_p(t)<0 \quad \text{ and } \quad v_p(z)+1, v_p(x)> v_p(y).$$ 
Conversely, assuming~(\ref{Klingenconditions1}) and choosing
$y_1=t=\frac1y,$ $x_1=z_1=0$ and $\mat{a}{b}{c}{d}=\mat{y}{-x-\frac{z}y}{}{y^{-1}}$,
we have
$$\pp^{-1}s_2^{-1} \JJ \nn \JJ^{-1}= \left[\begin{smallmatrix}
    &          &        &  1  \\
    &          & 1 & -\frac{z}y \\
   -\frac{z}y&-1  & \frac{1}y   & \frac{x}y \\
    1&         &   & \frac{1}y
\end{smallmatrix}\right] \in s_1 B(p) \JJ^{-1},$$
and furthermore
 $\pp \in \left[\begin{smallmatrix}
 \frac1y&   &   &    \\
  &   &  & -y \\
  &   & y & \\
  &  y^{-1}& & x+\frac{z}{y}
\end{smallmatrix}\right] \Nk$ hence $\mf_1(\pp)=\frac1y$
and $\mf_2(\pp)=\mat{}{-y}{y^{-1}}{x+\frac{z}y}$.
Together with the calculation of $P_{s_1}$ in Lemma~\ref{KlingenStabilizers}, this finishes the proof.
\end{proof}

By Lemma~\ref{IwasawaKlingens1}, the local integral we need to calculate is over matrices $\nn=\left[\begin{smallmatrix}
 1& x  & z  &  y  \\
  &  1 & y &  \\
  &   & 1 & \\
  &  &  -x &1
\end{smallmatrix}\right] $ where $(x,y,z)$ belong to
$$R=\{(x,y,z) \in \Q_p^3 :\text{ (\ref{Klingenconditions1}) holds}\}.$$
In the course of the proof of Lemma~\ref{LocalIntegral1}
below, we shall split $R$ into the following sub-ranges 
 \begin{equation*}
        \begin{split}
    R_1&=\{v_p(z+xy)<0  \text{ and~(\ref{Klingenconditions1}) holds}\},\\
    R_2&=\{v_p(z+xy) \ge 0  \text{ and~(\ref{Klingenconditions1}) holds}\},\\
    R'_{1}&=R_1 \cap \{v_p(x) \ge -1 \text{ and } v_p(y^2) \ge v_p(z+xy)-1\},\\
   % R'_{11}&=R'_1 \cap \{v_p(z)<-1\},\\
   % R'_{12}&=R'_1 \cap \{v_p(z) = -1\},\\
   %   R'_{2}&=R_2 \cap \{v_p(x) \ge -1 \text{ and } v_p(z-xy) \ge v_p(x^2)-1, v_p(z) \ge v_p(x^2)\}.
        \end{split}
\end{equation*}
\begin{lemma}\label{dissect3}
We have $R'_1=C_1\sqcup (C_2 \setminus S_1)$,
where 
$$C_1=p^{-1}\Z_p^\times \times p^{-2}\Z_p^\times \times p^{-2}\Z_p,$$
$$C_2=\Z_p \times p^{-1}\Z_p^\times \times p^{-1}\Z_p,$$
$$S_1=\{(x,y,z) \in C_2: z \in -xy+\Z_p\}.$$
\end{lemma}
\begin{proof}
    Let $(x,y,z) \in R'_1$.
    The condition $v_p(y^2) \ge v_p(z+xy)-1$ is only possible either if $v_p(y^2) \ge v_p(z)-1$ or $v_p(y) \ge v_p(x)-1$.
   In the first case, the inequality
   $$\frac{v_p(z)-1}2 \le v_p(y) \le v_p(z)$$
   together with $v_p(y)<0$
   forces $v_p(y)=v_p(z)=-1$, and we have $v_p(x) \ge 0$ by definition.
   In the second case, since $v_p(y)<\min\{0,v_p(x)\}$ and $v_p(x) \ge -1$, we must have $v_p(y)=v_p(x)-1 \in \{-2,-1\}$.
   Finally, the condition $v_p(z+xy)-1 \le v_p(y^2)$ is satisfied unless $z \in -xy+p^2y^2\Z_p$, which in the above cases can only happen if $v_p(y)=-1$.
   Conversely, assume $(x,y,z) \in C_1\sqcup (C_2 \setminus S_1)$. It is clear that~(\ref{Klingenconditions1}) holds, as well as $v_p(x) \ge -1$.
   Furthermore, if $(x,y,z) \in C_1$ then $v_p(z+xy)=-3=v_p(y^2)+1$. On the other hand if
   $(x,y,z) \in C_2 \setminus S_1$ then $v_p(z+xy) = -1 =v_p(y^2)+1$ again, and so in both cases we have $(x,y,z) \in R'_1$.   
\end{proof}

\begin{lemma}\label{KLocalIntegrals1}
    Let $\varphi$ be a factorizable Maa{\ss} form for $\Gamma_0(p)$ on $\GL_2$.
    Let $\phi \in \H_{\Pk}$ defined by
    $$\phi_{\kk}(\mm)=
    \begin{cases}
       \varphi(\mf_2(m)) \text{ if } \kk=s_1,\\
        0 \text { if } \kk \in X_{\Pk} \setminus \{s_1\}.
    \end{cases}$$
    For $\nu \in \C$ with $\Re(\nu)$
     large enough, we have
    \begin{equation*}
        \begin{split}
            J_p(\phi,\nu)=p^{-1-2\nu}\W_p\left(\mat{p^2}{}{}{1}\right)-(1-p^{-1})p^{-1-\nu}\W_p(1)
        \end{split}
    \end{equation*}
\end{lemma}
\begin{proof}
By Lemma~\ref{IwasawaKlingens1} we have 
$$ J_p(\phi,\nu)=\int_R |y|^{-2-\nu} \W_p(\varphi)\left(\mat{}{-y}{y^{-1}}{x+\frac{z}y}\right) \theta(-x) \, dxdydz.$$
Write $\int_R=\int_{R_1}+\int_{R_2}$.
If $v_p(z+xy)<0$ then we have 
$$\mat{}{-y}{y^{-1}}{x+\frac{z}y}=
\mat{1}{-\frac{y^2}{z+xy}}{}{1}
\mat{y/(z+xy)}{}{}{(z+xy)/y}
\underbrace{\mat{1}{}{(z+xy)^{-1}}{1}}_{\in \Gamma_0(p)}$$
and thus
$$\W_p(\varphi)\left(\mat{}{-y}{y^{-1}}{x+\frac{z}y}\right) 
=\theta\left(-\frac{y^2}{z+xy}\right)\W_p(\varphi)\left(\mat{y/(z+xy)}{}{}{(z+xy)/y}\right).$$
Note that $|y|^{-2-\nu}\W_p(\varphi)\left(\mat{y/(z+xy)}{}{}{(z+xy)/y}\right)$ is invariant by the change of variable 
$$(x,y,z) \mapsto (\lambda x, \mu y, \lambda \mu z).$$
Therefore by Lemmas~\ref{GaussTransform} and~\ref{dissect3} we have $\int_{R_1}=\int_{R'_1}=\int_{C_1}+\int_{C_2 \setminus S_1}$.
Moreover by the same change of variables and evaluating the Gau{\ss} sums, we get
\begin{align*}
    \int_{C_1}&=p^{-2}\zeta_p(1)^2 \W_p\left(\mat{p^2}{}{}{1}\right) \int_{v_p(x)=-1} \int_{v_p(y) =-2} \int_{v_p(z)\ge-2}|y|^{-2-\nu} \, dzdydx\\
    &=p^{-1-2\nu}\W_p\left(\mat{p^2}{}{}{1}\right).
\end{align*}
Next, for $(x,y,z) \in C_2 \setminus S_1$ we have 
$v_p(y)=v_p(z+xy)=-1$ and thus by the same argument we get 
\begin{align*}
    \int_{C_2 \setminus S_1}&=-p^{-1}\zeta_p(1)
\W_p\left(\mat{1}{}{}{1}\right) \int_{v_p(x) \ge 0} \int_{v_p(y)=-1} |y|^{-2-\nu}\left(\int_{v_p(z) \ge -1} -\int_{-xy+\Z_p} dz \right)\,dydx \\
&= -(1-p^{-1})p^{-1-\nu}\W_p(1)
\end{align*}

On the other hand if $v_p(z+xy) \ge 0$ then we have 
$$\mat{}{-y}{y^{-1}}{x+\frac{z}y}=
\mat{-y}{}{}{y^{-1}}
\mat{}{1}{-1}{}
\underbrace{\mat{1}{-(z+xy)}{}{1}}_{\in \Gamma_0(p)}$$
and thus
$$\W_p(\varphi)\left(\mat{}{-y}{y^{-1}}{x+\frac{z}y}\right) 
=\W_p(\varphi)\left(\mat{}{y}{y^{-1}}{}\right)=\W_p(\varphi)\left(\mat{}{y^2}{1}{}\right).$$
But for $(x,y,z) \in R_2$ we have $v_p(y)\le -1$ and hence Lemma~\ref{supportWhittaker} implies that
$\W_p(\varphi)\left(\mat{}{y^2}{1}{}\right)=0$, thus 
$\int_{R_2}=0.$
\end{proof}

\subsubsection{The local integral for $s=s_1s_2$} 
\begin{lemma}\label{IwasawaKlingens1s2}
Let $\nn=\left[\begin{smallmatrix}
 1& x  & z  &  y  \\
  &  1 & y &  \\
  &   & 1 & \\
  &  &  -x &1
\end{smallmatrix}\right]  \in \Nk$. Then $\sigma \nn \in \Pk s_1 s_2 B(p)$ if and only if 
\begin{equation}\label{Klingenconditions1s2}
    v_p(x)\le  \min\{-1, v_p(y), v_p(z)\}. 
\end{equation}
Moreover, if $\sigma \nn =\pp s_1 s_2 \gamma$ for some $\gamma \in B(p)$ then $\mf_2(\pp) \in \mat{-x^{-1}}{-y+\frac{z}x}{}{x} \Gamma_0(p) $
and $\mf_1(\pp) \in x^{-1} \Z_p^\times$.
\end{lemma}
\begin{proof}
Assume $\sigma \nn \in \pp s_1 s_2 B(p)$ for some $\pp \in \Pk$. Equivalently, since $s_2 \sigma =\JJ$
\begin{equation}\label{IKs1s2}
 \pp^{-1}s_2^{-1} \JJ \nn \JJ^{-1} \in s_1 s_2 B(p) \JJ^{-1}=\left[\begin{smallmatrix}
 *&  \o^\times &  \p &  \p \\
 * &  * & \o^\times & * \\
 * &  * & \p &  \o^\times\\
 \o^\times & \p & \p&\p
\end{smallmatrix}\right] \cap G(\o).
\end{equation}
Write $\pp^{-1}s_2^{-1}=\left[\begin{smallmatrix}
 (ad-bc)t^{-1}&   &   &    \\
  &  a &  & b \\
  &   & t & \\
  & c & & d
\end{smallmatrix}\right] \left[\begin{smallmatrix}
 1&  x_1 & z_1  & y_1   \\
  &  1 & y_1 &  \\
  &   & 1 & \\
  &  &  -x_1  &1
\end{smallmatrix}\right] $.
%We have $\mu(pp^{-1}s_2^{-1} \JJ \nn \JJ^{-1})=ad-bc$, hence we deduce
%$$ad-bc \in \Z_p^\times.$$ Furthermore 
The third line of $\pp^{-1}s_2^{-1} \JJ \nn \JJ^{-1}$ is
$$\left[\begin{smallmatrix}
    -tz & -ty & t & tx
\end{smallmatrix}\right],$$
from which we deduce
$$v_p(x)=-v_p(t)<0 \quad \text{ and } \quad v_p(z), v_p(y) \ge v_p(x).$$ 
Conversely, assuming~(\ref{Klingenconditions1s2}) and choosing
$x_1=t=\frac1x,$ $y_1=z_1=0$ and $\mat{a}{b}{c}{d}=\mat{\frac{z}{x}-y}{x}{x^{-1}}{}$,
we have
$$\pp^{-1}s_2^{-1} \JJ \nn \JJ^{-1}= \left[\begin{smallmatrix}
    &   -1       &        &    \\
    &   \frac{z}x       & -1 &  \\
   -\frac{z}x&-\frac{y}x  & \frac{1}x   & 1 \\
    -1& \frac1x        &   & 
\end{smallmatrix}\right] \in s_1 s_2 B(p) \JJ^{-1},$$
and furthermore
 $\pp \in \left[\begin{smallmatrix}
 -\frac1x&   &   &    \\
  &  -x^{-1} &  & -y+\frac{z}{x} \\
  &   & x & \\
  &   & & x
\end{smallmatrix}\right] \Nk$ hence $\mf_1(\pp)=\frac1x$
and $\mf_2(\pp)=\mat{-x^{-1}}{-y+\frac{z}x}{}{x}$.
Together with the calculation of $P_{s_1s_2}$ in Lemma~\ref{KlingenStabilizers}, this finishes the proof.
\end{proof}

\begin{lemma}\label{KLocalIntegrals1s2}
    Let $\varphi$ be a factorizable Maa{\ss} form for $\Gamma_0(p)$ on $\GL_2$.
    Let $\phi \in \H_{\Pk}$ defined by
    $$\phi_{\kk}(\mm)=
    \begin{cases}
       \varphi(\mf_2(m)) \text{ if } \kk=s_1s_2,\\
        0 \text { if } \kk \in X_{\Pk} \setminus \{s_1s_2\}.
    \end{cases}$$
    For $\nu \in \C$ with $\Re(\nu)$
     large enough, we have
    \begin{equation*}
        \begin{split}
            J_p(\phi,\nu)=-p^{-\nu}\W_p\left(\mat{p^2}{}{}{1}\right).
        \end{split}
    \end{equation*}
\end{lemma}
\begin{proof}
Let $R=\{(x,y,z) \in \Q_p^3 :\text{ (\ref{Klingenconditions1s2}) holds}\}.$
Then by Lemma~\ref{IwasawaKlingens1s2}
\begin{align*}
  J_p(\phi,\nu)&=\int_{R}|x|^{-2-\nu}  \W_p\left(\mat{-x^{-1}}{-y+\frac{z}{x}}{}{1}\right)\theta(-x) \, dxdydz\\
  &=\int_{R}|x|^{-2-\nu}  \W_p\left(\mat{-x^{-1}}{}{}{x}\right)\theta\left(\frac{zx-y}{x^2}-x\right)\, dxdydz.
\end{align*}
By Lemma~\ref{GaussTransform} we have 
$J_p(\phi,\nu)=\int_{R'},$
where $R'=\{(x,y,z) \in R: v_p(x) = -1\}$.
Note that for $(x,y,z) \in R'$ we have $v_p(zx-y) \ge -2$ and thus $v_p(\frac{zx-y}{x^2}) \ge 0$.
Hence
\begin{align*}
    J_p(\phi,\nu)&=
   \W_p\left(\mat{p^2}{}{}{1}\right) \int_{v_p(x)=-1}\int_{v_p(y) \ge -1}\int_{v_p(z) \ge -1}
    |x|^{-2-\nu}  \theta\left(-x\right)\, dzdydx\\
    &=-p^{-\nu}\W_p\left(\mat{p^2}{}{}{1}\right).
\end{align*}
\end{proof}

\subsubsection{The local integral for $s=s_1s_2s_1$} 
\begin{lemma}\label{IwasawaKlingens1s2s1}
Let $\nn=\left[\begin{smallmatrix}
 1& x  & z  &  y  \\
  &  1 & y &  \\
  &   & 1 & \\
  &  &  -x &1
\end{smallmatrix}\right]  \in \Nk$. Then $\sigma \nn \in \Pk s_1 s_2 s_1 B(p)$ if and only if 
\begin{equation}\label{Klingenconditions1s2s1}
   x,y,z \in \Z_p. 
\end{equation}
Moreover, if $\sigma \nn =\pp s_1 s_2 \gamma$ for some $\gamma \in B(p)$ then $\mf_2(\pp) \in \Gamma_0(p) $
and $\mf_1(\pp) \in  \Z_p^\times$.
\end{lemma}
\begin{proof}
Assume $\sigma \nn \in \pp s_1 s_2 s_1 B(p)$ for some $\pp \in \Pk$. Equivalently, since $s_2 \sigma =\JJ$
\begin{equation}\label{IKs1s2s1}
 \pp^{-1}s_2^{-1} \JJ \nn \JJ^{-1} \in s_1 s_2 s_1 B(p) \JJ^{-1}=\left[\begin{smallmatrix}
 \o^\times& \p   &  \p &  \p \\
 * &  * & \p & \o^\times \\
 * &  * & \o^\times & *\\
 * & \o^\times & \p&\p
\end{smallmatrix}\right] \cap G(\o).
\end{equation}
Write $\pp^{-1}s_2^{-1}=\left[\begin{smallmatrix}
 (ad-bc)t^{-1}&   &   &    \\
  &  a &  & b \\
  &   & t & \\
  & c & & d
\end{smallmatrix}\right] \left[\begin{smallmatrix}
 1&  x_1 & z_1  & y_1   \\
  &  1 & y_1 &  \\
  &   & 1 & \\
  &  &  -x_1  &1
\end{smallmatrix}\right] $.
%We have $\mu(pp^{-1}s_2^{-1} \JJ \nn \JJ^{-1})=ad-bc$, hence we deduce
%$$ad-bc \in \Z_p^\times.$$ Furthermore 
The third line of $\pp^{-1}s_2^{-1} \JJ \nn \JJ^{-1}$ is
$$\left[\begin{smallmatrix}
    -tz & -ty & t & tx
\end{smallmatrix}\right],$$
from which we deduce
$$v_p(t)=0 \quad \text{ and } \quad v_p(x), v_p(z), v_p(y) \ge 0.$$ 
Conversely, assuming~(\ref{Klingenconditions1s2s1}) and choosing
$t=1,$ $x_1=y_1=z_1=0$ and $\mat{a}{b}{c}{d}=\mat{}{-1}{1}{}$,
we have
$$\pp^{-1}s_2^{-1} \JJ \nn \JJ^{-1}= \left[\begin{smallmatrix}
   1 &          &        &    \\
   y &          &  & -1 \\
   -z&-y  & 1   & x \\
    -x& 1        &   & 
\end{smallmatrix}\right] \in s_1 s_2 s_1 B(p) \JJ^{-1},$$
and furthermore
 $\pp \in  \Nk$ hence $\mf_1(\pp)=1$
and $\mf_2(\pp)=1$.
Together with the calculation of $P_{s_1s_2s_1}$ in Lemma~\ref{KlingenStabilizers}, this finishes the proof.
\end{proof}

\begin{lemma}\label{KLocalIntegrals1s2s1}
    Let $\varphi$ be a factorizable Maa{\ss} form for $\Gamma_0(p)$ on $\GL_2$.
    Let $\phi \in \H_{\Pk}$ defined by
    $$\phi_{\kk}(\mm)=
    \begin{cases}
       \varphi(\mf_2(m)) \text{ if } \kk=s_1s_2s_1,\\
        0 \text { if } \kk \in X_{\Pk} \setminus \{s_1s_2s_1\}.
    \end{cases}$$
    For $\nu \in \C$ with $\Re(\nu)$
     large enough, we have
    \begin{equation*}
        \begin{split}
            J_p(\phi,\nu)=\W_p\left(1\right).
        \end{split}
    \end{equation*}
\end{lemma}
\begin{proof}
This follows directly from Lemma~\ref{IwasawaKlingens1s2s1}.
\end{proof}

\begin{remark}
    As a sanity check, let $\varphi$ be a factorizable Maa{\ss} form for $\SL_2(\Z)$, normalised so that $\W_p(\varphi)\left(\mat{1}{}{}{1}\right)=1$.
    Let $\phi \in \H_{\Pk}$ defined by $\phi(pk)=\varphi(\mf_2(p))$.
    Then combining Lemmas~\ref{KLocalIntegral1},~\ref{KLocalIntegrals1}~\ref{KLocalIntegrals1s2} and~\ref{KLocalIntegrals1s2s1} we have $$ J_p(\phi,\nu)=1-(1-p^{-1-\nu})p^{-\nu}\W_p(\varphi)\left(\mat{p^2}{}{}{1}\right)-p^{-3-3\nu}=
    L_p^{-1}(1+\nu,\varphi,\rm{sym}^2)^{-1},$$ which is the correct local factor (compare with~\cite{Shahidi}*{Theorem~7.1.2}).
\end{remark}

\subsection{Siegel subgroup}\label{Siegel} 
When $P$ is the Klingen parabolic subgroup 
 the element $\sigma=\sigma_P$  of $(W \cap M) \JJ$ such that 
$\sigma_P (U \cap M) \sigma_P^{-1}=U \cap M$ is given by $\sigma_P=s_2s_1s_2$.
\subsubsection{The local integral for $s=1$} \label{Siegel1}
\begin{lemma}\label{IwasawaSiegels1}
Let $\nn=\mat{1}{Y}{}{1} \in \Ns$. Then $\sigma \nn \in \Ps B(p)$ if and only if $Y=\mat{x}{y}{y}{z}$ with
\begin{equation}\label{s1condition}
    \min\{v_p(x),v_p(y),v_p(z)\}>v_p(\det Y). 
\end{equation}
Moreover, if $\sigma \nn =\pp  \gamma$ for some $\gamma \in B(p)$ then $\mf_2(\pp) \in \mat{}{1}{1}{} Y^{-1} \Gamma_0(p).$
\end{lemma}
\begin{remark}\label{Rk1}
    Note that condition~(\ref{s1condition}) implies in particular
    $v_p(\det Y)<0$.
\end{remark}
\begin{proof}
Assume $\sigma \nn \in \pp B(p)$ for some $\pp \in \Ps$. Equivalently, since $s_1 \sigma =\JJ$
\begin{equation}\label{ISs1}
 \pp^{-1}s_1^{-1} \JJ \nn \JJ^{-1} \in B(p) \JJ^{-1}=\left[\begin{smallmatrix}
 *& *  &  * & *   \\
 * &  * & \p & * \\
 * &  \p & \p & \p\\
 * & * & \p&\p
\end{smallmatrix}\right] \cap G(\o).
\end{equation}
write $\nn=\mat{1}{Y}{}{1}$ and 
$\pp^{-1}s_1^{-1}=\mat{A}{}{}{t\trans{A}^{-1}}\mat{1}{X}{}{1}$
for some $A \in \GL_2(\Q_p)$, $t\in \Q_p^{\times}$ and $X$ a symmetric
matrix in $\Mat_2(\Q_p)$.
Then equation~(\ref{ISs1}) is equivalent to the following
\begin{align}
        A(1+XY) \in \Mat_2(\o), & \quad  AX \in \mat{*}{*}{ \p }{*} \cap \GL_2(\o), \label{TopRow}\\
          \trans{A}^{-1} Y \in \mat{*}{\p}{ * }{*} \cap \GL_2(\o), & \quad \trans{A^{-1}} \in p \Mat_2(\o), \label{BottomRow}
\end{align}
and $t \in \o^\times$.
First note that if~(\ref{BottomRow}) has a solution $(A,Y)$, then
$Y$ must be invertible, and taking $X=-Y^{-1}$ gives a solution to~(\ref{TopRow}), since with this choice we have
$$AX=-AY^{-1}=-\trans{\left(\trans{A}^{-1}Y\right)^{-1}}.$$
Now let us return to solve~(\ref{BottomRow}). That is, we are looking for $Y=\mat{x_1}{x_2}{x_2}{x_3}$ such that 
$$Y^{-1} \in p\Mat_2(\o).$$
Equivalently, we must have $v_p(x_j)>v_p(\det Y)$ for $j\in\{1,2,3\}$,
and in particular taking $j$ such that the corresponding $x_j$ has minimal $p$-adic valuation proves the first claim.
Next, we have 
\begin{align*}
    \mf_2(\pp)=&\mat{}{1}{1}{}A^{-1}\\
=&\trans{\left(\trans{A}^{-1}\mat{}{1}{1}{}\right)}\\
&\in \trans{\left(\mat{*}{\p}{ * }{*}Y^{-1} \mat{}{1}{1}{}\right)} 
=\mat{}{1}{1}{}Y^{-1}\Gamma_0(p).
\end{align*}
\end{proof}

By Lemma~\ref{IwasawaSiegels1}, the local integral we need to calculate is over 
$$R=\{Y=\mat{x}{y}{y}{z} \in \Mat_2(\Q_p):\text{ (\ref{s1condition}) holds}\}.$$
In the course of the proof of Lemma~\ref{LocalIntegral1}
below, we shall split $R$ into the following sub-ranges 
 \begin{equation*}
        \begin{split}
    R_1&=\{Y=\mat{x}{y}{y}{z} \in \Mat_2(\Q_p): v_p(y) < v_p(z) \text{ and~(\ref{s1condition}) holds}\},\\
    R_2&=\{Y=\mat{x}{y}{y}{z} \in \Mat_2(\Q_p): v_p(y) \ge v_p(z) \text{ and~(\ref{s1condition}) holds}\},\\
    R'_1&=R_1 \cap \{v_p(x) \ge v_p(y)-1 \text{ and } v_p(z) \ge -1\},\\
    R'_{11}&=R'_1 \cap \{v_p(y) \le v_p(x)\},\\
     R'_{12}&=R'_1 \cap \{v_p(y)=v_p(x)+1\},\\
     R'_2&=R_2 \cap \{v_p(z) \ge -1\},\\
     R'_{21}&=R'_2 \cap \{v_p(xz) > 2v_p(y)\},\\
     R'_{22}&=R'_2 \cap \{v_p(xz) \le 2v_p(y)\}.
        \end{split}
\end{equation*}
To lighten up the proof of Lemma~\ref{LocalIntegral1},
we make the following combinatorial lemmas.

\begin{lemma}\label{R'11}
We have
    \begin{equation*}
        R'_{11}=C_{11} \setminus S_{11},
    \end{equation*}
    where 
    $$C_{11}=\{Y=\mat{x}{y}{y}{z} \in \Mat_2(\Q_p): v_p(y) \le \min\{-1, v_p(x)\} \text{ and } v_p(z) \ge -1\},$$
    $$S_{11}=\{Y=\mat{x}{y}{y}{z} \in \Mat_2(\Q_p): v_p(y)=v_p(z)=-1 \le v_p(x) \} \subset C_{11}.$$
\end{lemma}
\begin{proof}
Let $Y=\mat{x}{y}{y}{z} \in R'_{11}$.
Since $v_p(y) \le v_p(x)$ and $v_p(y) < v_p(z)$, we have
$v_p(y^2) < v_p(xz)$ and thus 
$v_p(\det Y) = 2v_p(y) $. 
So~(\ref{s1condition}) holds if and 
only if $v_p(y)<0$.
Thus $Y \in C_{11}$.
Moreover the condition $v_p(y)<v_p(z)$ forces 
$Y \not \in S_{11}$.
Conversely, if $Y \in C_{11} \setminus S_{11}$, then it is easy
to see that $Y \in R'_{11}$.
\end{proof}

\begin{lemma}\label{R'12}
We have
    \begin{equation*}
        R'_{12}=C_{12} \setminus S_{12},
    \end{equation*}
    where 
    $$C_{12}=\{Y=\mat{x}{y}{y}{z} \in \Mat_2(\Q_p): v(y)=v(x)+1 < -1 \le v(z)\},$$
    $$S_{12}=\{Y=\mat{x}{y}{y}{z} \in \Mat_2(\Q_p):  v_p(x)+2=v_p(y)+1=v_p(z)=-1\} \subset C_{12}.$$
\end{lemma}
\begin{proof}
Let $Y=\mat{x}{y}{y}{z} \in R'_{12}$.
If $v_p(z)>v_p(y)+1$ then $v_p(xz)>v_p(y^2)$ and thus
$v_p(\det Y)=2v_p(y)$.
Then~(\ref{s1condition}) is equivalent to $2v_p(y)<v_p(x)$, that
is $v_p(y)<-1$.
On the other hand, if $v_p(z)=v_p(y)+1=v_p(x)+2$ then we have
$v_p(xz)=v_p(y^2)$ and thus $v_p(xz) \le v_p(\det Y)< v_p(x)$,
which forces $v_p(z)=-1$.
Finally, condition~(\ref{s1condition}) holds if and only if
$z \not \in \frac{y^2}{x}+\Z_p$.
Thus $Y \in C_{12} \setminus S_{12}.$
The converse is obvious.
\end{proof}
\begin{remark}\label{vdetY2y}
    The proof shows that for $Y \in R'_{12}$ we have 
    $v_p(\det Y)=2v_p(y)$.
\end{remark}

\begin{lemma}\label{R21}
We have
\begin{equation}\label{R'21}
     R'_{21}=\{Y=\mat{x}{y}{y}{z} \in \Mat_2(\Q_p): 
    v_p(z) = v_p(y) = -1 < v_p(x)\}.
     \end{equation}
\end{lemma}
\begin{proof}
Let $Y \in R'_{21}$. Then we have $2 v_p(y) = v_p(\det Y) < v_p(y)$, hence $v_p(y) <0$. But since $-1 \le v_p(z) \le v_p(y)$
we must have $v_p(y)=v_p(z)=-1$.
So the condition $2v_p(y) < v_p(xz)$ becomes $v_p(x) > -1$.
The reverse inclusion is obvious.
\end{proof}

\begin{lemma}\label{R22}
  We have
\begin{equation}\label{R'22}
     R'_{22}=C_{22} \setminus S_{22},
     \end{equation}
     where 
     $$C_{22}=\{Y=\mat{x}{y}{y}{z} \in \Mat_2(\Q_p): v_p(x) \le -1 = v_p(z) \le v_p(y)\},$$
     $$S_{22}=\{Y=\mat{x}{y}{y}{z} \in \Mat_2(\Q_p): v_p(x) = v_p(y) =-1 \text{ and } z  \in \frac{y^2}{x}+\Z_p\}.$$
\end{lemma}
\begin{proof}
Let $Y \in R'_{22}$.
Then we have $v_p(xz) \le v_p(\det Y) < v_p(z), v_p(x)$,
    which forces $v_p(x), v_p(z) <0$.
    In particular we must have $v_p(z)=-1$,
    and so $Y \in C_{22}$.
    Moreover condition~(\ref{s1condition}) is equivalent
    to $z \not \in \frac{y^2}{x}+\Z_p$, and so $Y \not \in S_{22}$.
    Conversely, note that if $Y \in C_{22}$ then the condition 
    $2v_p(y) \ge v_p(xz)$ is automatically satisfied. This establishes the reverse inclusion.
\end{proof}
\begin{remark}\label{vdetYxz}
The proof shows that for $Y \in R'_{22}$ we have 
$v_p(\det Y)=v_p(xz)$.
\end{remark}

\begin{lemma}\label{LocalIntegral1}
    Let $\varphi$ be a factorizable Maa{\ss} form for $\Gamma_0(p)$ on $\GL_2$.
    Let $\phi \in \H_{\Ps}$ defined by
    $$\phi_{\kk}(\mm)=
    \begin{cases}
       \varphi(\mf_2(m)) \text{ if } \kk=1,\\
        0 \text { if } \kk \in X_{\Ps} \setminus \{1\}.
    \end{cases}$$
    For $\nu \in \C$ with $\Re(\nu)$
     large enough, we have
     %we define $$\W(\phi)=\int_{\Ns(\Q_p)} I_{\Ps,\frac32+\nu}(\sigma \nn)W(\phi_{\sigma \nn})(1) \overline{\psi}(\nn)\, d\nn.$$
   % Then
    \begin{equation*}
        \begin{split}
            J_p(\phi,\nu)=-p^{-3-4\nu}\W_p(\varphi)\left(\mat{}{1}{1}{}\right)-(1-p^{-1})p^{-\frac32-3\nu}\W_p(\varphi)
\left(\mat{p^{-1}}{}{}{1}\mat{}{1}{1}{}\right).
        \end{split}
    \end{equation*}
\end{lemma}
\begin{proof}
    By Lemma~\ref{IwasawaSiegels1}, the integral is over 
    matrices $\nn=\mat{1}{Y}{}{1}$ 
    where $Y=\mat{x}{y}{y}{z}$ satisfies~(\ref{s1condition}),
    and for such $\nn$ we have 
    $$I_{\Ps,\frac32+\nu}(\sigma \nn)=|\det(Y)|^{-\frac32-\nu},$$
    and 
    $$\W_p(\phi_{\sigma \nn})(1)=\W_p(\varphi)\left( \mat{}{1}{1}{}Y^{-1}\right)=\W_p(\varphi)\left(\mat{-y}{x}{z}{-y}\right).$$
    Write $ J_p(\phi,\nu)=J_1+J_2,$
    where $J_i=\int_{R_i}.$
    If $v_p(y) < v_p(z)$ then we have
    $$\mat{-y}{x}{z}{-y}=\mat{1}{-x/y}{}{1}
    \mat{(xz-y^2)/y}{}{}{-y}
    \underbrace{\mat{1}{}{-z/y}{1}}_{\in \Gamma_0(p)}$$
    and thus
     $$\W_p(\varphi)\left(\mat{x}{-y}{-y}{z}\right)=\theta(-x/y)\W_p(\varphi)\left(\mat{(xz-y^2)/y^2}{}{}{1}\right).$$
     Note $$|\det Y|^{-\frac32-\nu}\W_p(\varphi)\left(\mat{(xz-y^2)/y^2}{}{}{1}\mat{}{1}{-1}{}\right)$$ is invariant by the change of variable $(x,y,z) \mapsto (\lambda \mu^2 x,  \lambda \mu y, \lambda z)$ for 
     $(\lambda,\mu)\in (\Z_p^\times)^2$.  Therefore, by Lemma~\ref{GaussTransform}
     $$J_1=\int_{R'_1}|\det(Y)|^{-\frac32-\nu}\W_p(\varphi)\left(\mat{(xz-y^2)/y^2}{}{}{1}\right) \theta(-z-x/y) \, dxdydz.$$
     Write $\int_{R'_1}=\int_{R'_{11}}+\int_{R'_{12}}$.
     Firstly, by Lemma~\ref{R'11} we have
     $\int_{R'_{11}}=\int_{C_{11}}-\int_{S_{11}}$.
     But $\int_{C_{11}}$ contains $\int_{p^{-1}\Z_p} \theta(-z) dz,$ which vanishes. 
     Moreover 
      \begin{align*}
        -\int_{S_{11}}&=
\int_{p^{-1} \Z_p^\times} \int_{p^{-1} \Z_p^\times} \int_{ p^{-1}\Z_p}
 |y|^{-3-2\nu} \W_p(\varphi)\left(\mat{1}{}{}{1}\mat{}{1}{1}{}\right) \theta(-z) dxdzdy\\
&=p^{-1-2\nu}(1-p^{-1}) \W_p(\varphi)\left(\mat{}{1}{1}{}\right).
    \end{align*}
     Next, by Lemma~\ref{R'12} and Remark~\ref{vdetY2y} we have
     $$\int_{R'_{12}}=\left(\int_{C_{12}}-\int_{S_{12}}\right) |y|^{-3-2\nu} \W_p(\varphi)\left(\mat{}{1}{1}{}\right) \theta(-z-x/y) dxdzdy.$$
     The integral $\int_{C_{12}}$ vanishes for the same reason.
     Changing variables $(x,y,z) \mapsto (\lambda \mu^2 x, \lambda \mu y, \lambda z)$ and integrating over $(\lambda, \mu) \in \Z_p^\times \times \Z_p^\times$, we have 
     \begin{align*}
         -\int_{S_{12}}&=-p^{-3-4\nu}\W_p(\varphi)\left(\mat{}{1}{1}{}\right).
     \end{align*}
    On the other hand,     
    if $v_p(y)\ge v_p(z)$ we have
    $$\mat{-y}{x}{z}{-y}=\mat{1}{-y/z}{}{1}\mat{-(xz-y^2)/z}{}{}{z}\mat{}{1}{1}{}
    \underbrace{\mat{1}{-y/z}{ }{1}}_{\in \Gamma_0(p)}$$
    and thus 
    $$\W_p(\varphi)\left(\mat{x}{-y}{-y}{z}\right)=\W_p(\varphi)\left(\mat{(xz-y^2)/z^2}{}{}{1}\mat{}{1}{1}{}\right).$$
    Therefore by Lemma~\ref{GaussTransform} we have
    $$J_2=\int_{R'_2}|\det Y|^{-\frac32-\nu}\W_p(\varphi)\left(\mat{(xz-y^2)/z^2}{}{}{1}\right)\theta(-z) \, dxdydz.$$
   Write $\int_{R'_2}=\int_{R'_{21}}+\int_{R'_{22}}$.
   By Lemma~\ref{R21} we have
    \begin{align*}
        \int_{R'_{21}}&=
\int_{p^{-1} \Z_p^\times} \int_{p^{-1} \Z_p^\times} \int_{ \Z_p}
 |y|^{-3-2\nu} \W_p(\varphi)\left(\mat{y^2/z^2}{}{}{1}\mat{}{1}{1}{}\right) \theta(-z) dxdzdy\\
&=-p^{-2-2\nu}(1-p^{-1}) \W_p(\varphi)\left(\mat{}{1}{1}{}\right).
    \end{align*}
  Finally by Lemma~\ref{R22} and Remark~(\ref{vdetYxz}) we have
  $$\int_{R'_{22}}=\left(\int_{C_{22}}-\int_{S_{22}}\right) |xz|^{-\frac32-\nu} \W_p(\varphi)\left(\mat{x/z}{}{}{1}\mat{}{1}{1}{}\right) \theta(-z) dxdzdy.$$
Firstly,
    \begin{align*}
        \int_{C_{22}}
=-p^{-\tfrac12-\nu} \int_{v_p(x) \le -1}
|x|^{-\frac32-\nu} \W_p(\varphi)
\left(\mat{px}{}{}{1}\mat{}{1}{1}{}\right)
dx.
         \end{align*}
By Lemma~\ref{supportWhittaker}
$\W_p(\varphi)
\left(\mat{px}{}{}{1}\mat{}{1}{1}{}\right)=0$
unless $v_p(px) \ge -1$.
Thus 
$$\int_{C_{22}}=-(1-p^{-1})
\left(p^{-1-2\nu}\W_p(\varphi)
\left(\mat{}{1}{1}{}\right)+p^{-\frac32-3\nu}\W_p(\varphi)
\left(\mat{p^{-1}}{}{}{1}\mat{}{1}{1}{}\right)
\right).$$
Note that the first term cancels $-\int_{S_{11}}$.
Finally, changing variables $z \mapsto \lambda z$
  and integrating over $\lambda \in \Z_p^\times$, we have
\begin{align*}
    -\int_{S_{22}}&=-G(p^{-1})p^2(1-p^{-1})^2p^{-3-2\nu}\W_p(\varphi)
\left(\mat{}{1}{1}{}\right)\\
&=(1-p^{-1})p^{-2-2\nu}\W_p(\varphi)
\left(\mat{}{1}{1}{}\right),
\end{align*}
which cancels $\int_{R'_{21}}$.
\end{proof}

\subsubsection{The local integral for $s=s_2$} \label{Siegels2}
\begin{lemma}\label{Iwasawas1s2s1} 
There exists a set $S$ of measure $0$ in $\Ns$ such that 
for all $\nn=\mat{1}{Y}{}{1} \in \Ns \setminus S$ we have $\sigma \nn \in \Ps s_2 B(p)$ if and only if $Y=\mat{x}{y}{y}{z}$ with
\begin{equation}\label{s1s2s1condition}
   v_p(x) \le \min\{-1,v_p(y)-1,v_p(\det Y)\}.
\end{equation}
Moreover, if $\sigma \nn =\pp s_2 \gamma$ for some $\gamma \in B(p)$ then $\mf_2(\pp) \in \mat{}{1}{1/x}{-y/x}\Gamma_0(p).$
\end{lemma}
\begin{proof}
    Assume $\sigma \nn \in \pp s_2 B(p)$ for some $\pp \in \Ps$. Equivalently, since $s_1 \sigma =\JJ$
\begin{equation}\label{ISs1s2s1}
s_1^{-1} \pp^{-1}s_1^{-1} \JJ \nn \JJ^{-1}  \in s_1s_2 B(p) \JJ^{-1}=\left[\begin{smallmatrix}
 * & \o^\times  &  \p & \p   \\
 * &  * & \o^\times & * \\
 * &  * & \p & \o^\times\\
 \o^\times & \p & \p&\p
\end{smallmatrix}\right] \cap G(\o).
\end{equation}
write $\nn=\mat{1}{Y}{}{1}$ and 
$s_1^{-1} \pp^{-1}s_1^{-1}=\mat{A}{}{}{t\trans{A}^{-1}}\mat{1}{X}{}{1}$
for some $A \in \GL_2(\Q_p)$, $t\in \Q_p^{\times}$ and $X$ a symmetric
matrix in $\Mat_2(\Q_p)$.
Then equation~(\ref{ISs1s2s1}) is equivalent to the following
\begin{align*}
        A(1+XY) \in \mat{\o}{\o^\times}{\o}{\o},  & \quad
          AX \in \mat{\p}{\p}{ \o^\times }{\o} ,\\
          \trans{A}^{-1} Y \in \mat{\o}{\o}{\o^\times}{\p}, & \quad 
     \trans{A^{-1}} \in \mat{\p}{\o^\times}{\p}{\p}, 
\end{align*}
and $t \in \o^\times$.
Observe that if $(A,X,Y)$ is a solution and $u,v \in \o^\times$,
$k \in \o$ then $(A',X,Y)$ is also a solution,
where $A'=\mat{u}{}{k}{v}A$. Thus we can assume without loss of generality that 
$$\tilde{Y}:=\trans{A}^{-1}Y=\mat{0}{y_1}{1}{y_2}$$
for some $y_1,y_2 \in \o$.
% Let $\tilde{X}=AX \equiv \mat{0}{0}{x_1}{0} \mod p$
% for some $x_1 \in \Z_p^\times$.
% Then 
% $$\mat{1}{}{x_1y_1}{1+x_1} \equiv 1+\tilde{X}\trans{\tilde{Y}} \in \mat{\p}{\o^\times}{\o}{\o}A^{-1} \subset \mat{\o^\times}{\p}{\o}{\p}$$
% and thus the first equation holds provided that $x_1 \equiv -1 \mod p.$
We can assume without loss of generality that $Y$ is invertible and $yz \neq 0$.
Then we have 
$$\trans{A}^{-1}=\tilde{Y}Y^{-1}=\frac1{\det Y}\mat{-yy_1}{xy_1}{z-yy_2}{-y+xy_2}.$$
For the first line to be in $\p \times \o^\times$,
we must have $0 \le v_p(y_1)=v_p(\det Y)-v_p(x)$ and $v_p(y) > v_p(x)$.
For the bottom line to be in $\p \times \p$, we must have
$$y_2 \in \o \cap \left(\frac{z}{y}+p\frac{\det Y}{y}\right) \cap \left(\frac{y}{x}+\frac{\det Y}{x}\p\right).$$
This intersection is non-empty if and only if
$$\frac{z}{y}-\frac{y}{x} \in \frac{\det Y}{y}\p,$$
which is equivalent to $v_p(x)<0$. Thus we have proved that $Y$ must satisfy~(\ref{s1s2s1condition}).
Conversely, assuming~(\ref{s1s2s1condition}) holds,
we may take $y_2=\frac{y}{x} \in \o$ and $y_1=\frac{\det Y}{x} \in \o$, which gives
$$\trans{A}^{-1}=\mat{-y/x}{1}{1/x}{0} \in \mat{\p}{\o^\times}{\p}{\p}.$$
% Finally observe that~(\ref{s1s2condition}) implies $v_p(x) > v_p(z)$
% (as otherwise we would have $v_p(\det Y)=v_p(x)+v_p(z)<v_p(z)$).
% Thus $x_2=-1-\frac{x}{z} \in \Z_p^\times$, and with this value of $x_2$
% we have 
% $$A(1+XY)=\mat{1+x/z}{y/z}{}{(\det Y)/z}\in $$
Taking $X=\mat{-1/x}{0}{0}{0}$, we have
$$AX=\mat{0}{0}{1}{0}\in \mat{\p}{\p}{\o^\times}{\o}$$
and 
$$A(1+XY)=\mat{0}{1}{0}{0} \in \mat{\o}{\o^\times}{\o}{\o}.$$
This finishes the proof.
\end{proof}

Define 
\begin{equation*}
    \begin{split}
R&=\{Y=\mat{x}{y}{y}{z}:v_p(x)\le \min\{-1,v_p(y)-1,v_p(\det Y) \},\\
R'&=\{Y=\mat{x}{y}{y}{z}:v_p(x)\le \min\{-1, v_p(y)-1,v_p(\det Y) \text{ and } v_p(z) \ge -1\},\\
R_1&=R' \cap \{v_p(y) \ge 0\},\\
R_2&=R' \cap \{v_p(y) <0\},\\
R'_2&=R_2 \cap \{v_p(x) \ge v_p(y)-1\} \cap \{v_p(x) \ge v_p(y^2)\}.
    \end{split}
\end{equation*}
\begin{lemma}\label{R'2s1s2s1}
    We have $R'_2=C_{21} \cup C_{22}$,
    where $$C_{21}=\{\mat{x}{y}{y}{z}:v_p(x)=v_p(y)-1=-2 \text{ and } z \in \Z_p\},$$
    and
    $$C_{22}=\{\mat{x}{y}{y}{z}:v_p(x)=v_p(y)-1=-3 \text{ and } z \in \frac{y^2}{x}+ \Z_p\}.$$
\end{lemma}
\begin{proof}
Let $Y \in R'_2$.
The condition $v_p(x) \le v_p(\det Y)$ is equivalent to
\begin{equation}\label{zcongruent}
    z \in \frac{y^2}{x} + \Z_p.
\end{equation}
Since $v_p(\frac{y^2}{x}) \le 0$ and $v_p(z) \ge -1$,
we must have either $v_p(\frac{y^2}{x})=0$ or $v_p(\frac{y^2}{x})=-1$.
Furthermore we have $v_p(x)=v_p(y)-1$ hence
either $v_p(y)=-1$ or $v_p(y)=-2$,
that is $Y \in C_{21}$ or $Y \in C_{22}$.
The converse is obvious.
\end{proof}

\begin{lemma}~\label{LocalIntegrals1s2s1}
    Let $\varphi$ be a factorizable Maa{\ss} form for $\Gamma_0(p)$ on $\GL_2$.
    Let $\phi \in \H_{\Ps}$ defined by
    $$\phi_{\kk}(\mm)=
    \begin{cases}
       \varphi(\mf_2(m)) \text{ if } \kk=s_2,\\
        0 \text { if } \kk \in X_{\Ps} \setminus \{s_2\}.
    \end{cases}$$
    For $\nu \in \C$ with $\Re(\nu)$
    large enough, we have
    %we define $$\W(\phi)=\int_{\Ns(\Q_p)} I_{\Ps,\frac32+\nu}(\sigma \nn)W(\phi_{\sigma \nn})(1) \overline{\psi}(\nn)\, d\nn.$$
   %Then
    \begin{equation*}
    \begin{split}
            J_p(\phi,\nu)=& \quad p^{-\tfrac12-\nu}(1-p^{-1})\W_p(\varphi)\left(\mat{p^{-1}}{}{}{1}\mat{}{1}{1}{}\right)\\
            &-p^{-1-2\nu}(1-p^{-1}) \W_p(\varphi)\left(1\right)\\
            &+p^{-\frac32-3\nu} \W_p(\varphi)\left(\mat{p}{}{}{1}\right).
            \end{split}
    \end{equation*}
\end{lemma}
\begin{proof}
By Lemma~\ref{Iwasawas1s2s1} we have 
$$ J_p(\phi,\nu)=\int_R |x|^{-\frac32-\nu}\W_p(\varphi)\left(\mat{}{1}{1/x}{-y/x}\right)\theta(-z) \, dx dy dz.$$
By right $\Gamma_0(p)$-invariance, note that
$\W_p(\varphi)\left(\mat{}{1}{1/x}{-y/x}\right)$ is invariant under
the transformation $(x,y,z) \mapsto (\lambda x,\lambda y,\lambda z)$ for $\lambda \in \Z_p^\times$.
Thus by Lemma~\ref{GaussTransform} we have 
$$ J_p(\phi,\nu)=\int_{R'} |x|^{-\frac32-\nu}\W_p(\varphi)\left(\mat{}{1}{1/x}{-y/x}\right)\theta(-z) \, dx dy dz.$$
If $v_p(y) \ge 0$ we have 
 $$\mat{}{1}{1/x}{-y/x}=x^{-1}\mat{x}{}{}{1}\mat{}{1}{1}{}
 \underbrace{\mat{1}{-y}{}{1}}_{\in \Gamma_0(p)},$$
 and~(\ref{zcongruent}) is equivalent to $z \in \Z_p$.
Thus 
\begin{align*}
    \int_{R_1}&= \int_{v_p(x) <0}\int_{\Z_p}\int_{\Z_p}
    |x|^{-\frac32-\nu}\W_p(\varphi)\left(\mat{x}{}{}{1}\mat{}{1}{1}{}\right) \, dx dy dz.
\end{align*}
By Lemma~\ref{supportWhittaker},
$\W_p(\varphi)\left(\mat{x}{}{}{1}\mat{}{1}{1}{}\right) \neq 0$
only for $v_p(x) \ge -1$.
Thus $$\int_{R_1}=p^{-\tfrac12-\nu}(1-p^{-1})\W_p(\varphi)\left(\mat{p^{-1}}{}{}{1}\mat{}{1}{1}{}\right).$$
 On the other hand, if $v_p(y)<0$ then we have
  $$\mat{}{1}{1/x}{-y/x}=-\frac{y}{x}
  \mat{1}{-\frac{x}{y}}{}{1}
  \mat{-\frac{x}{y^2}}{}{}{1}
 \underbrace{\mat{1}{}{-\frac1y}{1}}_{\in \Gamma_0(p)}.$$
 and thus in this case 
 $$\W_p(\varphi)\left(\mat{}{1}{1/x}{-y/x}\right)=
 \theta\left(-x/y\right)
 \W_p(\varphi)\left(\mat{-\frac{x}{y^2}}{}{}{1}\right).$$
  By a similar argument, the above expression is only non-zero if $v_p(\frac{x}{y^2}) \ge 0$.
  Hence changing variables
  $(x, y, z \mapsto \lambda^2x, \lambda y, z)$ and using Lemmas~\ref{GaussTransform} and~\ref{R'22} we have  we have
  $\int_{R_2}=\int_{R'_2}=\int_{C_{21}}+\int_{C_{22}}$.
Finally
\begin{align*}
\int_{C_{21}}&=
\int_{p^{-2}\Z_p^\times}\int_{p^{-1}\Z_p^\times}\int_{\Z_p}
|x|^{-\frac32-\nu}\theta\left(-x/y\right)
 \W_p(\varphi)\left(\mat{-\frac{x}{y^2}}{}{}{1}\right) \, dz dy dx\\
 &=-p^{-1-2\nu}(1-p^{-1}) \W_p(\varphi)\left(\mat{1}{}{}{1}\right),
\end{align*}
and changing variables $(x,y) \mapsto (\lambda \mu^2 x, \lambda \mu y)$ and integrating over $(\lambda, \mu) \in \Z_p^\times \times \Z_p^\times$
\begin{align*}
\int_{C_{22}}&=
\int_{p^{-3}\Z_p^\times}\int_{p^{-2}\Z_p^\times}\int_{\frac{y^2}{x}+\Z_p}
|x|^{-\frac32-\nu}\theta\left(-z-x/y\right)
 \W_p(\varphi)\left(\mat{-\frac{x}{y^2}}{}{}{1}\right) \, dz dy dx\\
 &=p^{-\frac32-3\nu} \W_p(\varphi)\left(\mat{p}{}{}{1}\right).
\end{align*}
\end{proof}

\subsubsection{The local integral for $s=s_2s_1$} \label{Siegels2s1}
\begin{lemma}\label{Iwasawas1s2} 
There exists a set $S$ of measure $0$ in $\Ns$ such that 
for all $\nn=\mat{1}{Y}{}{1} \in \Ns \setminus S$ we have $\sigma \nn \in \Ps s_2 B(p)$ if and only if $Y=\mat{x}{y}{y}{z}$ with
\begin{equation}\label{s1s2condition}
   v_p(z) \le \min\{-1,v_p(y),v_p(\det Y)\}.
\end{equation}
Moreover, if $\sigma \nn =\pp s_2s_1 \gamma$ for some $\gamma \in B(p)$ then $\mf_2(\pp) \in \mat{1/z}{-y/z}{}{1}\Gamma_0(p).$
\end{lemma}
\begin{proof}
    Assume $\sigma \nn \in \pp s_2s_1 B(p)$ for some $\pp \in \Ps$. Equivalently, since $s_1 \sigma =\JJ$
\begin{equation}\label{ISs1s2}
s_1^{-1} \pp^{-1} s_1^{-1} \JJ \nn \JJ^{-1} \in s_1s_2s_1B(p) \JJ^{-1}=\left[\begin{smallmatrix}
 \o^\times& \p  &  \p & \p   \\
 * &  * & \p & \o^\times \\
 * &  * & \o^\times & *\\
 * & \o^\times & \p&\p
\end{smallmatrix}\right] \cap G(\o).
\end{equation}
write $\nn=\mat{1}{Y}{}{1}$ and 
$s_1^{-1} \pp^{-1} s_1^{-1}=\mat{A}{}{}{t\trans{A}^{-1}}\mat{1}{X}{}{1}$
for some $A \in \GL_2(\Q_p)$, $t\in \Q_p^{\times}$ and $X$ a symmetric
matrix in $\Mat_2(\Q_p)$.
Then equation~(\ref{ISs1s2}) is equivalent to the following
\begin{align}
        A(1+XY) \in \mat{\o^\times}{\p}{\o}{\o}, 
        & \quad  AX \in \mat{\p}{\p}{ \p }{\o^\times}, \label{line1}\\
          \trans{A}^{-1} Y \in \mat{\o}{\o}{\o}{\o^\times}, & \quad         
        \trans{A^{-1}} \in 
        \mat{\o^\times}{\o}{\p}{\p}, \label{line2}
\end{align}
and $t \in \o^\times$.
% Then we have 
% $$A(1+XY)=(1+\tilde{X}\trans{\tilde Y})A$$
% and thus by~(\ref{diag1})
% $$\mat{1}{}{y_1x_1+y_3x_2}{1+x_2y_2} \equiv 1+\tilde{X}\trans{\tilde Y} \in  \mat{\o^\times}{\o}{\o}{\o} A^{-1}
% \subset\mat{\o^\times}{\p}{ \o }{\p},$$
% hence we must have $x_2y_2 \equiv -1 \mod p$.
Observe that if $(A,X,Y)$ is a solution and $u,v \in \o^\times$,
$k \in \o$ then $(A',X,Y)$ is also a solution,
where $A'=\mat{u}{}{k}{v}A$. Thus we can assume without loss of generality that
$\tilde{Y}:=\trans{A}^{-1} Y=\mat{y_1}{}{y_2}{1}$
for some $y_1,y_2 \in \o$.
Since $d\nn$-almost all $Y$ are invertible and satisfy $xyz \neq 0$,
we may assume that this is the case. Then, given $Y=\mat{x}{y}{y}{z}$ and %$\tilde{X}$, 
$\tilde{Y}$ as above, we have 
\begin{align*}\label{transAinverse}
    \trans{A}^{-1}&=\tilde{Y}Y^{-1}\\
    &=\frac1{\det Y}\mat{zy_1}{-yy_1}{-y+zy_2}{x-yy_2}
\end{align*}
The first row is in $\o^\times \times \o$ if and only if
$$0 \le v_p(y_1)=v_p(\det Y)-v_p(z)$$
and 
\begin{equation}\label{zDividesy}
    v_p(y) \ge v_p(z).
\end{equation}
% Note that if $y_1 \in \Z_p^\times$ then we must have $z \in (\det Y)\Z_p^\times$, $y \in (p \det Y)\Z_p$ and $x \in (p \det Y)\Z_p$,
% whence 
% $$\det Y=xz-y^2 \in (p \det Y) \Z_p,$$
% which forces $\det Y=0$, contradicting the assumption that $Y$ is invertible. In particular, we must have
% $$v_p(z)<v_p(\det Y).$$
The bottom row is divisible by $p$ if and only if
\begin{equation}\label{Intersection}
y_2 \in \o \cap \left(\frac{y}{z}+\frac{\det Y}{z}\p\right)
\cap \left(\frac{x}{y}+\frac{\det Y}{y}\p\right).
\end{equation}
By~(\ref{zDividesy}) the intersection
$$\left(\frac{y}{z}+\frac{\det Y}{z}\p\right)
\cap \left(\frac{x}{y}+\frac{\det Y}{y}\p\right)$$
is non-empty if and only if
$$\frac{y^2-xz}{yz} \in \frac{\det Y}{y}\p,$$
which is equivalent to $$v_p(z)<0.$$
Thus we have proved that $Y$ must satisfy~(\ref{s1s2condition}).
Conversely, we need to check that, provided~(\ref{s1s2condition})
holds, we may find a symmetric matrix $X$ and $\tilde{Y}$
as above such that
equations~(\ref{line1}) and~(\ref{line2}) are satisfied.
Taking $y_2=\frac{y}{z} \in \o$ and $y_1=\frac{\det Y}z \in \o$ gives
$$\trans{A}^{-1}=\mat{1}{-y/z}{}{1/z} \in \mat{\o^\times}{\p}{\p}{\p}.$$
Furthermore choosing $X=\mat{0}{0}{0}{-z^{-1}}$ 
we have
$$AX=\mat{0}{0}{0}{-1} \in \mat{\p}{\p}{\p}{\o^\times}$$
and 
$$A(1+XY)=\mat{1}{0}{0}{0} \in \mat{\o^\times}{\p}{\o}{\o}.$$
% we have 
% $$X=A^{-1}\tilde{X}=z^{-1}\mat{1}{}{}{x_2},$$
% which in particular is symmetric.
Finally, 
$$\mf_2(\pp)=\mat{}{1}{1}{}A^{-1}\mat{}{1}{1}{}=\mat{1/z}{-y/z}{}{1},$$
and the fact that $\mf_2(P_{s_2s_1})=\Gamma_0(p)$ establishes the result.
\end{proof}
\begin{lemma}~\label{LocalIntegrals2s1}
    Let $\varphi$ be a factorizable Maa{\ss} form for $\Gamma_0(p)$ on $\GL_2$.
    Let $\phi \in \H_{\Ps}$ defined by
    $$\phi_{\kk}(\mm)=
    \begin{cases}
       \varphi(\mf_2(m)) \text{ if } \kk=s_2s_1,\\
        0 \text { if } \kk \in X_{\Ps} \setminus \{s_2s_1\}.
    \end{cases}$$
    For $\nu \in \C$ with $\Re(\nu)$
    large enough, we have
    % we define $$\W(\phi)=\int_{\Ns(\Q_p)} I_{\Ps,\frac32+\nu}(\sigma \nn)W(\phi_{\sigma \nn})(1) \overline{\psi}(\nn)\, d\nn.$$
    \begin{equation*}
    \begin{split}
              J_p(\phi,\nu)= -p^{-\tfrac12-\nu}\W_p(\varphi)\left(\mat{p}{}{}{1}\right) 
    \end{split}
    \end{equation*}
\end{lemma}
\begin{proof}
    By Lemmas~\ref{Iwasawas1s2} and~\ref{GaussTransform}, we have 
    \begin{align*}
         J_p(\phi,\nu)&=\int_{R} |z|^{-\frac32-\nu}\W_p(\varphi)\left(\mat{1/z}{}{}{1}\right) \theta(-z)\, dx dy dz\\
        &=\int_{R'} |z|^{-\frac32-\nu}\W_p(\varphi)\left(\mat{1/z}{}{}{1}\right) \theta(-z)\, dx dy dz,
    \end{align*}
        where $$R=\{Y=\mat{x}{y}{y}{z}: v_p(z) \le \min\{-1,v_p(y),v_p(\det Y)\}\},$$
        $$R'=R \cap\{v_p(z) \ge -1\}.$$
The condition $v_p(z) \le v_p(\det Y)$ is equivalent to
$$x \in \frac{y^2}{z}+\Z_p.$$
% This implies
% $$v_p(x)\begin{cases}=2v_p(y)-v_p(z) \text{ if } 2v_p(y)-v_p(z)<1\\
% \ge 1 \text{ if } 2v_p(y)-v_p(z) \ge 1.
% \end{cases}$$
% In particular, $(x,y,z)$ can only belong to $R'$ if $2v_p(y) \ge v_p(z)-1$, which is a stronger condition than $v_p(z)<v_p(y)$ when 
% $v_p(y) \le -2$.
Thus
 \begin{align*}
         J_p(\phi,\nu)=& \quad
        \int_{p^{-1} \Z_p^\times}
        \int_{\Z_p} \int_{\Z_p}
        |z|^{-\frac32-\nu}\W_p(\varphi)\left(\mat{1/z}{}{}{1}\right) \theta(-z)\, dx dy dz\\
        &+\int_{p^{-1} \Z_p^\times}
        \int_{p^{-1}\Z_p^\times}\int_{\frac{y^2}{z}+\Z_p}
        |z|^{-\frac32-\nu}\W_p(\varphi)\left(\mat{1/z}{}{}{1}\right)\theta(-z)\, dx dy dz\\
        &= -p^{-\tfrac12-\nu}\W_p(\varphi)\left(\mat{p}{}{}{1}\right).
    \end{align*}
\end{proof}

\subsubsection{The local integral for $s=s_2s_1s_2$} \label{Siegels2s1s2}
\begin{lemma}\label{IwasawaJ}
For all $\nn=\mat{1}{Y}{}{1} \in \Ns$ we have $\sigma \nn \in \Ps s_2s_1s_2 B(p)$ if and only if $Y \in \Mat_2(\Z_p)$.
Moreover, if $\sigma \nn =\pp s_2s_1s_2 \gamma$ for some $\gamma \in B(p)$ then $\mf_2(\pp)\in \Gamma_0(p).$
\end{lemma}
\begin{proof}
    Assume $\sigma \nn \in \pp s_2s_1s_2 B(p)$ for some $\pp \in \Ps$. 
    Equivalently, since $s_1\sigma=\JJ$
\begin{equation}\label{IJ}
 s_1^{-1}\pp^{-1}s_1^{-1} \JJ \nn \JJ^{-1}  \in \JJ B(p) \JJ^{-1}=\left[\begin{smallmatrix}
 * & \p  &  \p & \p   \\
 * &  * & \p & \p \\
 * &  * & * & *\\
 * & * & \p & *
\end{smallmatrix}\right] \cap G(\o).
\end{equation}
write $\nn=\mat{1}{Y}{}{1}$ and 
$s_1^{-1}\pp^{-1}s_1^{-1}=\mat{A}{}{}{t\trans{A}^{-1}}\mat{1}{X}{}{1}$
for some $A \in \GL_2(\Q_p)$, $t\in \Q_p^{\times}$ and $X$ a symmetric
matrix in $\Mat_2(\Q_p)$.
Then equation~(\ref{IJ}) is equivalent to the following
\begin{align}
        A(1+XY) \in \mat{*}{*}{\p}{*} \cap \GL_2(\o), 
        & \quad  AX \in p \Mat_2(\o), \label{TopRowJ}\\
          \trans{A}^{-1} Y \in \Mat_2(\o) , 
        & \quad \trans{A^{-1}} \in \mat{*}{\p}{ * }{*} \cap \GL_2(\o), \label{BottomRowJ}
\end{align}
and $t \in \o^\times$.
The last equation implies $A \in \trans{\Gamma_0(p)}$.
Thus we must have $X \in p\Mat_2(\o)$
and $Y \in \Mat_2(\o)$, from which $A(1+XY) \in \trans{\Gamma_0(p)}$ is automatic. Conversely if $Y \in \Mat_2(\o)$ then taking $A=1$ and $X=0$ gives a solution.
\end{proof}
\begin{lemma}~\label{LocalIntegrals2s1s2}
    Let $\varphi$ be a factorizable Maa{\ss} form for $\Gamma_0(p)$ on $\GL_2$.
    Let $\phi \in \H_{\Ps}$ defined by
    $$\phi_{\kk}(\mm)=
    \begin{cases}
       \varphi(\mf_2(m)) \text{ if } \kk=s_2s_1s_2,\\
        0 \text { if } \kk \in X_{\Ps} \setminus \{s_2s_1s_2\}.
    \end{cases}$$
    For $\nu \in \C$ with $\Re(\nu)$
    large enough, we have %define $$\W(\phi)=\int_{\Ns(\Q_p)} I_{\Ps,\frac32+\nu}(\sigma \nn)W(\phi_{\sigma \nn})(1) \overline{\psi}(\nn)\, d\nn.$$
  % Then
    \begin{equation*}
             J_p(\phi,\nu)=\W_p(\varphi)\left(1\right).
    \end{equation*}
\end{lemma}
\begin{proof}
   This follows directly from Lemma~\ref{IwasawaJ} 
\end{proof}
\begin{remark}
    As a sanity check, let $\varphi$ be a factorizable Maa{\ss} form for $\SL_2(\Z)$, normalised so that $\W_p(\varphi)\left(\mat{1}{}{}{1}\right)=1$.
    Let $\phi \in \H_{\Ps}$ defined by $\phi(pk)=\varphi(\mf_2(p))$.
    Then combining Lemmas~\ref{LocalIntegral1},~\ref{LocalIntegrals1s2s1},~\ref{LocalIntegrals2s1}, and~\ref{LocalIntegrals2s1s2} we have $$ J_p(\phi,\nu)=(1-p^{-1-2\nu})(1-p^{-\tfrac12-\nu}\W_p(\varphi)\left(\mat{p}{}{}{1}\right)+p^{-2-2\nu})=
    \zeta_p^{-1}(1+2\nu)L_p^{-1}(1+\nu,\varphi),$$ which is the correct local factor (compare with~\cite{Shahidi}*{Theorem~7.1.2}).
\end{remark}

\subsection{Borel subgroup}\label{BorLocInt}
For $P=B$ we have $\sigma_P=\JJ$.
\subsubsection{The double coset decomposition}\label{strategyBorel}
For $\uu=\left[\begin{smallmatrix}
 1& x  & a+bx  &  b+cx  \\
  &  1 & b & c \\
  &   & 1 & \\
  &  &  -x &1
\end{smallmatrix}\right] \in U$
we can determine for which element $s$ of the Weyl group
we have $\JJ \uu \in B s B(p)$ by re-using the calculations we previously did for the Klingen and Borel subgroup.
Namely, observe 
$$\uu=\left[\begin{smallmatrix}
 1& x  &  &   \\
  &  1 &  &  \\
  &   & 1 & \\
  &  &  -x &1
\end{smallmatrix}\right]\left[\begin{smallmatrix}
 1&   & a  &  b  \\
  &  1 & b & c \\
  &   & 1 & \\
  &  &   &1
\end{smallmatrix}\right]=\left[\begin{smallmatrix}
 1&   &   &    \\
  &  1 &  & c \\
  &   & 1 & \\
  &  &   &1
\end{smallmatrix}\right]\left[\begin{smallmatrix}
 1& x  & a+bx  &  b+cx  \\
  &  1 & b+cx &  \\
  &   & 1 & \\
  &  &  -x &1
\end{smallmatrix}\right],$$
so 
$$\JJ \uu = \underbrace{\left[\begin{smallmatrix}
 1&   &  &   \\
 x &  1 &  &  \\
  &   & 1 & x\\
  &  &   &1
\end{smallmatrix}\right]}_{\in \Ps} \JJ \left[\begin{smallmatrix}
 1&   & a  &  b  \\
  &  1 & b & c \\
  &   & 1 & \\
  &  &   &1
\end{smallmatrix}\right]= \underbrace{\left[\begin{smallmatrix}
 1&   &  &   \\
  &  1 &  &  \\
  &   & 1 & \\
  &  c&   &1
\end{smallmatrix}\right]}_{\in \Pk} \JJ \left[\begin{smallmatrix}
 1& x  & a+bx  &  b+cx  \\
  &  1 & b+cx &  \\
  &   & 1 & \\
  &  & -x  &1
\end{smallmatrix}\right]$$
Therefore if $\JJ\uu \in BsB(p)$ then we must have \begin{equation}\label{ReducetoSiegel}
    \JJ\left[\begin{smallmatrix}
 1&   & a  &  b  \\
  &  1 & b & c \\
  &   & 1 & \\
  &  &   &1
\end{smallmatrix}\right] \in \Ps s B(p),
\end{equation} 
and by~\S~\ref{Siegel} we know how to determine $s \in W$
such that~(\ref{ReducetoSiegel}) holds. This $s$ is only determined up to multiplication by $s_1$ on the left.
However, by the same argument, we must also have 
\begin{equation}\label{ReducetoKlingen}
    \JJ\left[\begin{smallmatrix}
 1& x  & a+bx  &  b+cx  \\
  &  1 & b+cx &  \\
  &   & 1 & \\
  &  & -x  &1
\end{smallmatrix}\right] \in \Pk s B(p),
\end{equation}
and by~\S~\ref{Klingen} we know how to determine $s \in W$
such that~(\ref{ReducetoKlingen}) holds. This $s$ is now determined up to $s_2$. Having determined $s$ both up to $s_1$ and up to $s_2$ is sufficient to completely determine $s$.
For instance, assuming that~(\ref{s1condition}) holds (with $Y$ replaced by $\mat{a}{b}{b}{c}$), we must have $s=1$ or $s=s_1$.
To determine which case we are in, it suffices to check which of~(\ref{IwasawaKlingen1}) or~(\ref{Klingenconditions1}) holds
for the matrix on the left hand side of~(\ref{ReducetoKlingen}), which amounts to compare $v_p(a+bx)$ and $v_p(b+cx)$. 
Finally note that by Lemma~\ref{decomposepowers} we have 
\begin{align*}
     I_{B,\rho+\nu}(\JJ \uu) &=  I_{\Pk,1+\nu_1-\nu_2}(\JJ\uu)I_{\Ps,1+\nu_2}(\JJ\uu) \\
     &= I_{\Pk,1+\nu_1-\nu_2}\left(\JJ\left[\begin{smallmatrix}
 1& x  & a+bx  &  b+cx  \\
  &  1 & b+cx &  \\
  &   & 1 & \\
  &  & -x  &1
\end{smallmatrix}\right]\right)
I_{\Ps,1+\nu_2}
\left(\left[\begin{smallmatrix}
 1&   & a  &  b  \\
  &  1 & b & c \\
  &   & 1 & \\
  &  &   &1
\end{smallmatrix}\right]\right).
\end{align*}
We now carry out this argument to calculate the eight local integrals.

\subsubsection{The local integral for $s=1$}
    Let $R$ be as in~\S~\ref{Siegel1}, so that, as explained above, for $\uu=\left[\begin{smallmatrix}
 1& x  & a+bx  &  b+cx  \\
  &  1 & b & c \\
  &   & 1 & \\
  &  &  -x &1
\end{smallmatrix}\right]$ we have $\JJ \uu \in B(\Q_p)B(p)$
if and only if $Y:=\mat{a}{b}{b}{c} \in R$ and $v_p(a+bx)<v_p(b+cx)$.
Let $R_1=R\cap\{v_p(c) \ge 0\}$ and $R_2=R\cap\{v_p(c)=-1\}$.
\begin{lemma}\label{111}
    Let $\uu \in U(\Q_p)$. Then we have 
    \begin{equation}\label{R1x}
        Y \in R_1, \quad v_p(x) \ge 0 \text{ and }v_p(a+bx)<v_p(b+cx)
    \end{equation}
    if and only if 
    \begin{equation}\label{2bab}
        2v_p(b)<v_p(a)<v_p(b)<0 \le v_p(x),v_p(c).
    \end{equation}
\end{lemma}
\begin{proof}
    Assume~(\ref{R1x}).
    Then we have $v_p(ac-b^2) \ge \min \{v_p(ac),2v_p(b)\}$.
    On the other hand by~(\ref{Siegel1}) we have $v_p(ac-b^2) < v_p(a) \le v_p(ac)$ since $v_p(c) \ge 0$.
    Thus $2v_p(b)<v_p(ac)$, from which follows $v_p(ac-b^2)=2v_p(b)$. Hence using~(\ref{Siegel1}) again, it follows $2v_p(b)<v_p(b), v_p(a)$. This implies in particular that $v_p(b)<0$. Therefore $v_p(b+cx)=v_p(b)$, and hence by~(\ref{R1x}) 
    $v_p(a+bx)<v_p(b) \le v_p(bx)$.
    Since $v_p(a+bx) \ge \min \{v_p(a),v_p(bx)\}$,
    it must follow that $v_p(a+bx)=v_p(a)$ and hence $v_p(a)<v_p(b)$. We have proven~(\ref{2bab}).
    Conversely,~(\ref{2bab}) trivially implies~(\ref{R1x}). 
\end{proof}

\begin{lemma}\label{112}
    Let $\uu \in U(\Q_p)$. Then we have 
    \begin{equation}\label{R1x2}
        Y \in R_1, \quad v_p(x) = -1 \text{ and }v_p(a+bx)<v_p(b+cx)
    \end{equation}
    if and only if $v_p(x)=-1<v_p(c)$ and
one of the following holds
\begin{enumerate}
\item $v_p(b)=-1 \le v_p(a)$, or
\item $v_p(b) \le -2$ and $v_p(a) \ge v_p(b)-1$ and $a \not \in -bx(1+\p)$, or
\item $2v_p(b)+1 \le v_p(a) \le v_p(b)-2$.
\end{enumerate}
\end{lemma}
\begin{proof}
The proof is the same as in Lemma~\ref{111} up until $v_p(b)<0$. If $v_p(b)<-1$ then we still have $v_p(b+cx)=v_p(b)$
but now the condition $v_p(a+bx)<v_p(b)$ is equivalent to
$a \in -bx(1+\p)$. Thus if $v_p(b)<-1$ then~(2) or~(3) must hold. On the other hand if $v_p(b)=-1$ then since $v_p(b^2)<v_p(a)$ we must have $v_p(a) \ge -1$, which is~(1).
Conversely it is clear that either of~(1),~(2) or~(3) implies~(\ref{R1x2}).
\end{proof}

\begin{lemma}\label{121}
    Let $\uu \in U(\Q_p)$. Then we have 
    \begin{equation}\label{R2x}
        Y \in R_2, \quad v_p(x) \ge 0 \text{ and }v_p(a+bx)<v_p(b+cx)
    \end{equation}
    if and only if $v_p(c)=-1$ and one of the following holds
    \begin{enumerate}
        \item $v_p(a) < 0 \le v_p(b)$ and $v_p(x)>0$,
        \item $v_p(a)<-1$, $0 \le v_p(b)$ and $v_p(x)=0$,
        \item $2v_p(b)+1<v_p(a)<v_p(b)<-1$ and $v_p(x) \ge 0$,
        \item $v_p(a)<2v_p(b)+1 \le -1$ and $v_p(x) \ge 0$,
        \item $v_p(a)=2v_p(b)+1<-1$, $ac \not \in b^2(1+\p)$
        and $v_p(x) \ge 0$,
        \item $v_p(b)=-1\le v_p(a)$, $x \in -\frac{b}{c}+\p$  and $ac \not \in b^2(1+\p)$
    \end{enumerate}
\end{lemma}
\begin{proof}
  Assume~(\ref{R2x}). Consider first the case $v_p(b) \ge 0$.
  By~(\ref{Siegel1}) we must have $v_p(b^2-ac)<\min\{v_p(a),v_p(b),v_p(c)\}$. Since $v_p(c)=-1$, this forces $v_p(b^2-ac)=v_p(a)-1$ and thus $v_p(a)<0$.
  Therefore we have $v_p(a+bx)=v_p(a)$ and hence by~(\ref{R2x})
  we must have $v_p(b+cx)>v_p(a)$. If $v_p(x)>0$ this is automatic, whereas if $v_p(x)=0$ then we have $v_p(b+cx)=-1$,
  and thus we must have $v_p(a)<-1$. Thus we have shown that if $v_p(b) \ge 0$ then either~(1) or~(2) must hold. 
  Next consider the case $v_p(b)<-1$. In this case we have 
  $v_p(b+cx)=v_p(b)$ and so by~(\ref{R2x}) we must have $v_p(a+bx)<v_p(b)$. Since $v_p(x) \ge 0$, this holds if and only if $v_p(a) < v_p(b)$ (and so in particular $v_p(a) \le -2$). Thus~(\ref{Siegel1}) becomes the condition that $v_p(b^2-ac)<v_p(a)$.
  If $2v_p(b) \neq v_p(a)-1$ this is automatically verified.
  On the other hand if $2v_p(b)=v_p(a)-1$ then we must have 
  $v_p(b^2-ac)=v_p(a)-1$, which is equivalent to $ac \not \in b^2(1+\p)$. Thus we have shown that if $v_p(b)<-1$ then either~(3),~(4) or~(5) must hold. Finally, consider the case 
  $v_p(b)=-1$. If $v_p(a)<-1$ then we are in case~(4). On the other hand if $v_p(a) \ge -1$ then by~(\ref{Siegel1})
  we must have $v_p(ac-b^2)<-1$, and the only possibility is $v_p(ac-b^2)=-2$, which implies $ac \not \in b^2(1+\p)$.
  Moreover we must have $v_p(a+bx)<v_p(b+cx)$ and since 
  $v_p(a+bx) \ge -1$ it follows that $v_p(b+cx) \ge 0$, which establishes~(6).
  Conversely, it is clear that assuming $v_p(c)=-1$ either 
  of~(1),~(2)~(3)~(4),~(5) imply~(\ref{R2x}).
  Assume that~(6) holds and $v_p(c)<-1$.
  Then it is clear that $Y \in R_2$. So it only remains to check $v_p(a+bx)<v_p(b+cx)$.
  By assumption we have $bx \in \frac{-b^2}{c}+\Z_p$.
  Since $a \not \in \frac{b^2}{c}+\Z_p$, it follows that 
  $a+bx \not \in \Z_p$, and thus $v_p(a+bx)<v_p(b+cx)$ as desired.
\end{proof}

\begin{lemma}\label{211}
    Let $\uu \in U(\Q_p)$. Then we have 
    \begin{equation}\label{R2x2}
        Y \in R_2, \quad v_p(x) =-1 \text{ and }v_p(a+bx)<v_p(b+cx)
    \end{equation}
    if and only if $v_p(x)=v_p(c)=-1,$ $bx \not \in -a(1+\p)$, $b^2 \not \in ac(1+\p)$ and if $v_p(b) \ge -1$
    then $v_p(a) < -2$.
\end{lemma}
\begin{proof}
   Let $\uu \in U(\Q_p)$ and assume~(\ref{R2x2}) holds.
   First assume $v_p(b) \ge -1$.
   Then since $v_p(cx)=-2$ we  have $v_p(b+cx)=-2$ and hence
   $v_p(a+bx)<-2$. Since $v_p(x)=-1,$ this is only possible if $v_p(a)<-2$, and thus the congruence conditions are automatically satisfied.
   Now assume $v_p(b) \le -2$.
   First note that if $v_p(a) \neq v_p(bx)$ then 
   $v_p(a+bx) \le v_p(bx)<v_p(b) \le v_p(b+cx)$.
   Similarly if $v_p(ac) \neq v_p(b^2)$ then 
   $v_p(ac-b^2) =\min\{v_p(ac),v_p(b^2)\}<\min\{v_p(a),v_p(b),v_p(c)\}$ and thus $Y \in R_2$.
   If $v_p(b^2)=v_p(ac)=v_p(a)-1$ then the condition $v_p(ac-b^2)<v_p(a)$ forces $v_p(ac-b^2)=v_p(a)-1$, which is equivalent to $b^2 \not \in ac(1+\p)$.
   Similarly the condition $v_p(a+bx)<v_p(b+cx)$ forces 
   $bx \not \in -a(1+\p)$, unless $v_p(b+cx)>v_p(b)$,
   which occurs only if $cx \in -b(1+\p)$, and hence 
   $bx \in -\frac{b^2}{c}(1+\Z_p)$. 
   But since we must have $b^2 \not \in ac(1+\p)$,
   the latter also implies $bx \not \in -a(1+\p)$, as desired.
   Conversely the condition $b^2 \not \in ac(1+\p)$ ensures 
   that $v_p(b^2-ac)=\min\{v_p(b^2),v_p(ac)\}$. On the other hand the last assumption implies that $\min\{v_p(b^2),v_p(ac)\} < \min\{v_p(a),v_p(b),v_p(c)\}$ and thus $Y \in R_2$.
   Similarly the condition $bx \not \in -a(1+\p)$ implies
   $v_p(a+bx)=\min\{v_p(a),v_p(bx)\}.$
   But $v_p(b+cx) \ge \min\{v_p(b),v_p(cx)\}$ and thus $v_p(a+bx)<v_p(b+cx)$ by our assumptions.
\end{proof}

\begin{lemma}\label{LocalIntegralBorel1}
    Let $\phi \in \H_{B}$ defined by
    $$\phi_{\kk}=
    \begin{cases}
      1 \text{ if } \kk=1,\\
        0 \text { if } \kk \in X_{B} \setminus \{1\}.
    \end{cases}$$
    For $\nu\in \Lie{a}_\C^*$ with $\Re(\nu)$
     large enough, we have
    \begin{equation*}
               \begin{split}
            J_p(\phi,\nu)=& -(1-p^{-1})p^{-2-\nu_1-\nu_2}\\
            &-(1-p^{-1})p^{-2-2\nu_1}\\
            &-(1-p^{-1})p^{-2-2\nu_1-\nu_2} \\
            &+p^{-4-3\nu_1-\nu_2}.
        \end{split}
    \end{equation*}
\end{lemma}
\begin{proof}
    By definition we have
    \begin{align*}
        J_p(\phi,\nu) &= \int_{U(\Q_p)}
        I_{B,\rho+\nu}(\JJ \uu) \psi(\uu) \1_{\JJ \uu \in B(\Q_p) B(p)} \, d\uu\\
        &=\int_R \int_{\Q_p} I_{\Pk,1+\nu_1-\nu_2}(\JJ\uu)I_{\Ps,1+\nu_2}(\JJ\uu)\theta(-x-c) \1_{v_p(a+bx)<v_p(b+cx)} \, dx dY.
    \end{align*}
    Changing variables $\uu \mapsto \tt \uu \tt^{-1}$ and then integrating over $\tt \in A(\Z_p)$, we deduce
     \begin{align*}
         J_p(\phi,\nu) =& \quad  
         \int_{R_1} \int_{\Z_p} I_{\Pk,1+\nu_1-\nu_2}(\JJ\uu)I_{\Ps,1+\nu_2}(\JJ\uu) \1_{v_p(a+bx)<v_p(b+cx)} \, dx dY\\
         &-p^{-1}\zeta_p(1)  \int_{R_2} \int_{\Z_p}I_{\Pk,1+\nu_1-\nu_2}(\JJ\uu)I_{\Ps,1+\nu_2}(\JJ\uu) \1_{v_p(a+bx)<v_p(b+cx)} \, dx dY\\
         &-p^{-1}\zeta_p(1)  \int_{R_1} \int_{p^{-1}\Z_p^\times} I_{\Pk,1+\nu_1-\nu_2}(\JJ\uu)I_{\Ps,1+\nu_2}(\JJ\uu) \1_{v_p(a+bx)<v_p(b+cx)} \, dx dY\\
         &+p^{-2}\zeta_p(1)^2 \int_{R_2} \int_{p^{-1}\Z_p^\times} I_{\Pk,1+\nu_1-\nu_2}(\JJ\uu)I_{\Ps,1+\nu_2}(\JJ\uu)\1_{v_p(a+bx)<v_p(b+cx)} \, dx dY.
     \end{align*}
     Denote by $J_i$ (for $1 \le i \le 4$) the four integrals in
     previous display, respectively.
     As in the proofs of Lemmas~\ref{KLocalIntegral1} and~\ref{LocalIntegral1} we have
     $$ I_{\Pk,1+\nu_1-\nu_2}(\JJ\uu)=|a+bx|^{-1-\nu_1+\nu_2}$$ and $$I_{\Ps,1+\nu_2}(\JJ\uu)=|ac-b^2|^{-1-\nu_2}.$$
     Therefore by Lemma~\ref{111} we have
     \begin{align*}
         J_1&=\iint_{\Z_p \times \Z_p} \int_{v_p(b)<0} \int_{2v_p(b)< v_p(a) < v_p(b)} |a|^{-1-\nu_1+\nu_2}|b|^{-2-2\nu_2} \, dadbdcdx\\
         &=\frac{1-p^{-1}}{1-p^{\nu_1-\nu_2}} \int_{v_p(b) \le -2}|b|^{-2-2\nu_2}\left(p^{\nu_1-\nu_2}|b|^{-2\nu_1+2\nu_2}-|b|^{-\nu_1+\nu_2}\right) \, db\\
         &=\frac{(1-p^{-1})^2}{1-p^{\nu_1-\nu_2}} \left(\frac{p^{-1-\nu_1-\nu_2}}{1-p^{1+\nu_1+\nu_2}}
         -p^{\nu_1-\nu_2}\frac{p^{-1-2\nu_1}}{1-p^{1+2\nu_1}}\right)\\
         &=(1-p^{-1})^2\frac{p^{-1-\nu_1-\nu_2}}{1-p^{\nu_1-\nu_2}}
         \left(\frac{1}{1-p^{1+\nu_1+\nu_2}}
         -\frac{1}{1-p^{1+2\nu_1}}\right)\\
         &=(1-p^{-1})^2\frac1{(1-p^{1+\nu_1+\nu_2})(1-p^{1+2\nu_1})}.
     \end{align*}
     By Lemma~\ref{121} we have
     \begin{align*}
         J_2 =&  \iint_{p\Z_p \times p^{-1}\Z_p^\times}
         \int_{\Z_p} \int_{v_p(a)<0} |a|^{-1-\nu_1+\nu_2} |ac|^{-1-\nu_2} \, da db dc dx\\
         &+ \iint_{\Z_p^\times \times p^{-1}\Z_p^\times}
         \int_{\Z_p} \int_{v_p(a)<-1} |a|^{-1-\nu_1+\nu_2} |ac|^{-1-\nu_2}
        \, da db dc dx\\
        &+ \iint_{\Z_p \times p^{-1}\Z_p^\times} \int_{v_p(b)<-1} \int_{2v_p(b)+1 < v_p(a)<v_p(b)} |a|^{-1-\nu_1+\nu_2} |b|^{-2-2\nu_2} \, da db dc dx\\
        &+ \iint_{\Z_p \times p^{-1}\Z_p^\times} 
        \int_{v_p(b) \le -1} \int_{\substack{v_p(a) \le 2v_p(b)+1\\ a \not \in \frac{b^2}{c}(1+p\Z_p)} }
        |a|^{-1-\nu_1+\nu_2} |ac|^{-1-\nu_2}
        \, da db dc dx\\
        &-\int_{p^{-1}\Z_p^\times} 
        \int_{v_p(b) = -1} \int_{\substack{v_p(a) =-1\\ a \not \in \frac{b^2}{c}(1+p\Z_p) }}
        \int_{\substack{x \in \Z_p \\  x \not \in -\frac{b}c+p\Z_p}}
        |a|^{-1-\nu_1+\nu_2} |ac|^{-1-\nu_2}
        \, da db dc \\
        &+\int_{p^{-1}\Z_p^\times} 
        \int_{v_p(b) = -1} \int_{\substack{v_p(a) \ge 0 } }
        \int_{ x \in -\frac{b}{c}+p\Z_p}
        |b|^{-1-\nu_1+\nu_2} |b|^{-2-2\nu_2}
        \, da db dc dx\\
     \end{align*}
     The first term equals $J_{21}:=-(1-p^{-1})^2\frac{p^{-1-\nu_2}}{1-p^{1+\nu_1}}$.
     The second term equals $J_{22}:=-(1-p^{-1})^3\frac{p^{-1-\nu_1-\nu_2}}{1-p^{1+\nu_1}}$.
     The third term equals
     \begin{align*}
         J_{23}:=\frac{p(1-p^{-1})^2}{p^{\nu_1-\nu_2}-1}&\int_{v_p(b)<-1}
         |b|^{-2-\nu_1-\nu_2}-p^{2\nu_1+2\nu_2}|b|^{-2-2\nu_1}\, db\\
         = \frac{p(1-p^{-1})^3}{p^{\nu_1-\nu_2}-1}&
         \left(\frac{p^{-2\nu_2}}{1-p^{1+2\nu_1}}-\frac{p^{-\nu_1-\nu_2}}{1-p^{1+\nu_1+\nu_2}}\right)\\
         =(1-p^{-1})^3&\frac{p^{-\nu_1-\nu_2}}{(1-p^{1+2\nu_1})(1-p^{1+\nu_1+\nu_2})}.
     \end{align*}
     The fourth term equals 
      \begin{align*}
       &\int_{p^{-1}\Z_p^\times} \int_{v_p(b)<0}
       \left(\int_{v_p(a) \le 2v_p(b)+1}-\int_{a \in \frac{b^2}{c}(1+p\Z_p)}\right)|a|^{-1-\nu_1+\nu_2}|ac|^{-1-\nu_2} \, dadbdc\\
       =& p^{-1-\nu_2}\int_{p^{-1} \Z_p^\times} \int_{v_p(b)<0}
      |b|^{-2-2\nu_1} \left(-\frac{(1-p^{-1})}{1-p^{^{1+\nu_1}}}-p^{\nu_1}\right) \, dbdc\\
       =& \underbrace{\frac{(1-p^{-1})^3p^{2+2\nu_1-\nu_2}}{(1-p^{1+2\nu_1})(1-p^{^{1+\nu_1}})}}_{:=J_{24}}+
       \underbrace{\frac{(1-p^{-1})^2p^{\nu_1-\nu_2}}{1-p^{1+2\nu_1}}}_{:=J_{25}}.
     \end{align*}
     The fifth and sixth term equal
     $$J_{26}:=-(1-p^{-1})^3(1-2p^{-1})p^{-\nu_1-\nu_2}+(1-p^{-1})^2p^{-2-\nu_1-\nu_2}.$$
     By Lemma~\ref{112} we have
     \begin{align*}
         J_3=& \quad p^{-2-2\nu_1+2\nu_2}p^{-2-2\nu_2} \iint_{p^{-1}\Z_p \times \Z_p} \iint_{p^{-1}\Z_p^\times \times p^{-1}\Z_p^\times} dx db da dc\\
         &+\iint_{p^{-1}\Z_p^\times \times \Z_p} \int_{v_p(b) \le -2}
         \int_{\substack{v_p(a) \ge v_p(b)-1\\ a \not \in -bx(1+p\Z_p)} } |bx|^{-1-\nu_1+\nu_2}|b|^{-2-2\nu_2} \, da db dc dx\\
         &+ \iint_{p^{-1}\Z_p^\times \times \Z_p} \int_{v_p(b) \le -3}
         \int_{2v_p(b)+1 \le v_p(a) \le v_p(b)-2}|a|^{-1-\nu_1+\nu_2}|b|^{-2-2\nu_2} \, dadbdcdx
     \end{align*}
     The first term equals $J_{31}:=p^{-1-2\nu_1}(1-p^{-1})^2$.
     The second term equals
     \begin{align*}
         J_{32}:=p^{-\nu_1+\nu_2}(1-p^{-1})&\int_{v_p(b) \le -2}|b|^{-3-\nu_1-\nu_2}\left(p|b|-|b|\right) \, da db\\
          =-(1-p^{-1}&)^3\frac{p^{-2\nu_1}}{1-p^{1+\nu_1+\nu_2}}.
     \end{align*}
     The third term equals 
     \begin{align*}
     J_{33}:=\frac{p(1-p^{-1})^2}{{1-p^{\nu_1-\nu_2}}}& \int_{v_p(b) \le -3} |b|^{-2-2\nu_2} (p^{\nu_1-\nu_2}|b|^{-2\nu_1+2\nu_2}-p^{-\nu_1+\nu_2}|b|^{-\nu_1+\nu_2})\, db\\
     =\frac{p(1-p^{-1})^3}{{1-p^{\nu_1-\nu_2}}}& \left(p^{-1-\nu_1}\frac{p^{-2-2\nu_1-2\nu_2}}{1-p^{1+\nu_1+\nu_2}}-p^{\nu_1-\nu_2}\frac{p^{-2-4\nu_1}}{1-p^{1+2\nu_1}}\right)\\
         =(1-p^{-1})^3& \frac{p^{-2\nu_1}}{(1-p^{1+\nu_1+\nu_2})(1-p^{1+2\nu_1})}.
     \end{align*}
     By Lemma~\ref{211} we have
     \begin{align*}
         J_4 =& \quad \iint_{p^{-1}\Z_p^\times \times p^{-1}\Z_p^\times}
         \int_{v_p(b)\le-2} 
         \int_{v_p(a)<2v_p(b)+1} |a|^{-1-\nu_1+\nu_2} |ac|^{-1-\nu_2} \, da db dc dx\\
         &+\iint_{p^{-1}\Z_p^\times \times p^{-1}\Z_p^\times}
         \int_{v_p(b)\ge-1} 
         \int_{v_p(a)<-2} |a|^{-1-\nu_1+\nu_2} |ac|^{-1-\nu_2}  \, da db dc dx\\
         &+ \iint_{p^{-1}\Z_p^\times \times p^{-1}\Z_p^\times}
         \int_{v_p(b)\le -2} 
         \int_{v_p(a)>v_p(b)-1}  |bx|^{-1-\nu_1+\nu_2} |b|^{-2-2\nu_2} \, da db dc dx\\
         &+\iint_{p^{-1}\Z_p^\times \times p^{-1}\Z_p^\times}
         \int_{v_p(b)\le-2} 
         \int_{\substack{2v_p(b) +1\le v_p(a) \le v_p(b)-1 \\
         a \not \in bx(1+p\Z_p)\\ a \not\in \frac{b^2}{c}(1+p\Z_p)}}  |a|^{-1-\nu_1+\nu_2} |b|^{-2-2\nu_2} \, da db dc dx.
     \end{align*}
     The first term equals
     \begin{align*}
        J_{41}:= p^{1-\nu_2}&\frac{(1-p^{-1})^3}{1-p^{-1-\nu_1}}  \int_{v_p(b)<-1}|b|^{-2-2\nu_1}\\
         =(1-p^{-1})^4&\frac{p^{1-\nu_1-\nu_2}}{(1-p^{1+\nu_1})(1-p^{1+2\nu_1})}.
     \end{align*}
     The second term equals $J_{42}:=-(1-p^{-1})^3\frac{p^{-2\nu_1-\nu_2}}{1-p^{1+\nu_1}}$.
     The third term equal $J_{43}:=-(1-p^{-1})^3\frac{p^{-2\nu_1}}{1-p^{1+\nu_1+\nu_2}}$.
     Finally the fourth term equals
     \begin{align*}
         & \quad \iint_{p^{-1}\Z_p^\times \times p^{-1}\Z_p^\times}
         \int_{v_p(b)\le-2} 
         \int_{2v_p(b) +1\le v_p(a) \le v_p(b)-1}  |a|^{-1-\nu_1+\nu_2} |b|^{-2-2\nu_2} \, da db dc dx\\
         &- \iint_{p^{-1}\Z_p^\times \times p^{-1}\Z_p^\times}\int_{v_p(b)\le-2}
         \left( \int_{a \in bx(1+p\Z_p)} + \int_{a \in \frac{b^2}{c}(1+p\Z_p)} \right)
          |a|^{-1-\nu_1+\nu_2} |b|^{-2-2\nu_2} \, da db dc dx\\
         &+   \iint_{p^{-1}\Z_p^\times \times p^{-1}\Z_p^\times}
         \int_{b \in xc(1+p\Z_p)} \int_{a \in bx(1+p\Z_p)}  |a|^{-1-\nu_1+\nu_2} |b|^{-3-\nu_1-\nu_2} \, da db dc dx\\
         =& \quad  \frac{p^2(1-p^{-1})^3}{1-p^{\nu_1-\nu_2}} 
         \int_{v_p(b) \le -2} |b|^{-2-2\nu_2}(p^{\nu_1-\nu_2}|b|^{-2\nu_1+2\nu_2}-|b|^{-\nu_1+\nu_2}) \, db\\
         &-p^2(1-p^{-1})^2 \int_{v_p(b) \le -2}(p^{-1-\nu_1+\nu_2}|b|^{-2-\nu_1-\nu_2}+p^{-1+\nu_1-\nu_2}|b|^{-2-2\nu_1})\, db +(1-p^{-1})^2 p^{-2-3\nu_1-\nu_2}\\
         =& \quad \underbrace{\frac{p^{2}(1-p^{-1})^4}{(1-p^{1+\nu_1+\nu_2})(1-p^{1+2\nu_1})}}_{:=J_{44}}
         +\underbrace{(1-p^{-1})^{3}\left(\frac{p^{-2\nu_1}}{1-p^{1+\nu_1+\nu_2}}+\frac{p^{-\nu_1-\nu_2}}{1-p^{1+2\nu_1}}\right)}_{:=J_{45}}\\
         &+\underbrace{(1-p^{-1})^2 p^{-2-3\nu_1-\nu_2}}_{:=J_{46}}.
     \end{align*}
     Observe that $J_{43}$ simplifies with the first term of $J_{45}$, and so
     \begin{equation}\label{5+15}\begin{split}
         -p^{-1}\zeta_p(1)J_{25}+p^{-2}\zeta_p(1)^2(J_{43}+J_{45})&=(1-p^{-1})\frac{p^{-2-\nu_1-\nu_2}-p^{-1+\nu_1-\nu_2}}{1-p^{1+2\nu_1}}\\
         &=
         (1-p^{-1})p^{-2-\nu_1-\nu_2}.
         \end{split}
     \end{equation}
     Similarly, 
     $$-p^{-1}\zeta_p(1)J_{33}
     +p^{-2}\zeta_p(1)^2J_{44}=(1-p^{-1})^2\frac{1-p^{-1-2\nu_1}}{(1-p^{1+\nu_1+\nu_2})(1-p^{1+2\nu_1})}=-(1-p^{-1})^2\frac{p^{-1-2\nu_1}}{1-p^{1+\nu_1+\nu_2}},$$
     which cancels with $-p^{-1}\zeta_p(1)J_{32}$.
   Next,
   \begin{equation}\label{1+4}\begin{split}
       J_1-p^{-1}\zeta_p(1)J_{23}&=(1-p^{-1})^2\frac{1-p^{-1-\nu_1-\nu_2}}{(1-p^{1+2\nu_1})(1-p^{1+\nu_1+\nu_2})}\\
       &=-(1-p^{-1})^2\frac{p^{-1-\nu_1-\nu_2}}{1-p^{1+2\nu_1}}\\
       &=-(1-p^{-1})^2\frac{p^{-1-\nu_1-\nu_2}-p^{-\nu_2}}{(1-p^{1+2\nu_1})(1-p^{1+\nu_1})}.
       \end{split}
   \end{equation}
   Therefore
   \begin{equation}\label{1+4+6+10}
       p^{-1}\zeta_p(1)J_{24}+p^{-2}\zeta_p(1)^2J_{41}=
   (1-p^{-1})^2\frac{p^{-\nu_2}}{1-p^{1+\nu_1}}.
   \end{equation}
   Finally,
 \begin{equation}\label{1+4+6+10+2+3+11}
 \begin{split}
     (\ref{1+4+6+10})&-p^{-1}\zeta_p(1)(J_{21}+J_{22})+p^{-2}\zeta_p(1)^2J_{42}=\\
     &(1-p^{-1})\frac{p^{-\nu_2}-p^{-1-\nu_2}+p^{-2-\nu_2}+p^{-2-\nu_1-\nu_2}-p^{-3-\nu_1-\nu_2}-p^{-2-2\nu_1-\nu_2}}{1-p^{1+2\nu_1}}\\
     =&(1-p^{-1})p^{-\nu_2}\frac{(1-p^{-2-2\nu_1})+p^{-2-\nu_1}(1-p^{1+\nu_1})+p^{-3-\nu_1}(p^{1+\nu_1}-1)}{1-p^{1+2\nu_1}}\\
     =&(1-p^{-1})p^{-\nu_2}(-p^{-1-\nu_1}-p^{-2-2\nu_1}+p^{-2-\nu_1}-p^{-3-\nu_1}).
     \end{split}
 \end{equation}
 Adding up $(\ref{5+15})+(\ref{1+4+6+10+2+3+11})-p^{-1}\zeta_p(1)(J_{26}+J_{31})+p^{-2}\zeta_p(1)^2J_{46}$ gives the result.
\end{proof}
\subsubsection{The local integral for $s=s_1$}
  Let $R$ be as in~\S~\ref{Siegel1} so that for $\uu=\left[\begin{smallmatrix}
 1& x  & a+bx  &  b+cx  \\
  &  1 & b & c \\
  &   & 1 & \\
  &  &  -x &1
\end{smallmatrix}\right]$ we have $\JJ \uu \in B(\Q_p)B(p)$
if and only if $Y:=\mat{a}{b}{b}{c} \in R$ and $v_p(b+cx)\le v_p(a+bx)$.
Let $R_1=R\cap\{v_p(c) \ge 0\}$ and $R_2=R\cap\{v_p(c)=-1\}$.
\begin{lemma}\label{B.s1.1}
    Let $\uu \in U(\Q_p)$. Then we have 
    \begin{equation}\label{s1R1x}
        Y \in R_1, \quad v_p(x) \ge 0 \text{ and }v_p(b+cx)\le v_p(a+bx)
    \end{equation}
    if and only if 
    \begin{equation}\label{s12bab}
        v_p(b) \le \min\{-1, v_p(a)\} \text{ and } v_p(c), v_p(x) \ge 0.
    \end{equation}
\end{lemma}
\begin{proof}
    The proof is the same as Lemma~\ref{111} up until $v_p(b+cx)=v_p(b)$. By~(\ref{s1R1x}) we must then have 
    $v_p(b) \le v_p(a+bx)$.
    Since $v_p(x) \ge 0$, this is equivalent to $v_p(a) \ge v_p(b)$.
    Conversely,~(\ref{s12bab}) trivially implies~(\ref{R1x}). 
\end{proof}

\begin{lemma}\label{B.s1.2}
    Let $\uu \in U(\Q_p)$. Then we have 
    \begin{equation}\label{s1R1x2}
        Y \in R_1, \quad v_p(x) = -1 \text{ and }v_p(b+cx)\le v_p(a+bx)
    \end{equation}
    if and only if 
    \begin{equation}\label{212s1}
     v_p(b)<v_p(x)=-1<v_p(c) \text{ and } a \in -bx(1+p\Z_p).
    \end{equation}
\end{lemma}
\begin{proof}
Then again proof is the same as in Lemma~\ref{111} up until $v_p(b)<0$. 
Next by~({\ref{s1R1x2}}) we have $v_p(b) \le v_p(b+cx) \le v_p(a+bx)$. Since $v_p(x)=-1$, this is only possible if $a \in -bx(1+\p)$. This implies in particular $v_p(a)=v_p(b)-1$, but since we have $2v_p(b)<v_p(a)$, this forces $v_p(b)<-1$ as desired. Conversely it is clear that~(\ref{212s1}) implies~(\ref{s1R1x2}).
\end{proof}

\begin{lemma}\label{B.s1.3}
    Let $\uu \in U(\Q_p)$. Then we have 
    \begin{equation}\label{s1R2x3}
        Y \in R_2, \quad v_p(x) \ge 0 \text{ and }v_p(b+cx) \le v_p(a+bx)
    \end{equation}
    if and only if $v_p(c)=-1$ and one of the following holds
    \begin{enumerate}
        \item $v_p(b) \ge 0$, $v_p(a)=-1$ and $v_p(x)=0$,
        \item $v_p(b) \le -1$, $v_p(x) \ge 0$, 
        $v_p(a) \ge v_p(b)$,
        $x \not \in -\frac{b}{c}(1+p\Z_p)$,  and 
        $a \not \in \frac{b^2}{c}(1+p\Z_p)$.
    \end{enumerate}
\end{lemma}
\begin{proof}
As in proof of Lemma~\ref{121}, if  $v_p(b) \ge 0$
  then we have $v_p(a+bx)=v_p(a)<0$.  Hence by~(\ref{s1R2x3})
  we must have $v_p(b+cx) \le v_p(a)$. This is only possible
  if $v_p(x)=0$ and $v_p(a)=-1$.
  Next consider the case $v_p(b)<-1$. In this case we have 
  $v_p(b+cx)=v_p(b)$ and so by~(\ref{R2x}) we must have $v_p(b) \le v_p(a+bx)$. Since $v_p(x) \ge 0$, this holds if and only if $v_p(b) \le v_p(a)$. Finally, consider the case $v_p(b)=-1$.
  If $v_p(b+cx)=-1$ then as before we need $v_p(a) \ge -1$.
  On the other hand if $v_p(b+cx) \ge 0$, that is $x \in -\frac{b}{c}(1+\p)$ then we have $a+bx \in \frac{ac-b^2}{c}+\Z_p$ and since $v_p(ac-b^2)<-1$ and $v_p(c)=-1$ then
  $v_p(a+bx)<0$, thus~(\ref{s1R2x3}) cannot hold.
  Conversely it is clear that either of~(1) or~(2) implies~(\ref{s1R2x3}).
\end{proof}

\begin{lemma}\label{B.s1.4}
    Let $\uu \in U(\Q_p)$. Then we have 
    \begin{equation}\label{s1R2x24}
        Y \in R_2, \quad v_p(x) =-1 \text{ and }v_p(b+cx)\le v_p(a+bx)
    \end{equation}
    if and only if $v_p(c)=v_p(x)=-1$,
    $a \not \in \frac{b^2}{c}(1+\p)$
    and one of the following holds
    \begin{enumerate}
        \item $v_p(b) = -1$ and $v_p(a) \ge -2$,
        \item $v_p(b) \ge 0$ and $v_p(a) \in \{-2,-1\}$,
        \item $v_p(b) \le -2$ and $a \in -bx(1+\p)$.
    \end{enumerate}
\end{lemma}
\begin{proof}
   Let $\uu \in U(\Q_p)$ and assume~(\ref{s1R2x24}) holds.
   First observe the condition $v_p(b^2-ac)<v_p(a)$ implies
   $a \not \in \frac{b^2}{c}(1+\p)$.
   Furthermore the condition $v_p(b^2-ac)<v_p(c)=-1$ implies 
   either $v_p(b)<0$ or $v_p(a)<0$.
   Now assume $v_p(b) \ge -1$.
   Then since $v_p(cx)=-2$ we  have $v_p(b+cx)=-2$ and hence
   $v_p(a+bx)\ge-2$. Since $v_p(x)=-1,$ this is only equivalent to $v_p(a) \ge -2$. 
   Next assume $v_p(b) \le -2$. Then $v_p(b+cx) \ge v_p(b)$
   and so for~(\ref{s1R2x24}) to hold we must have 
   $a \in -bx(1+\p)$.
   Conversely it is clear that assuming $v_p(c)=v_p(x)=-1$ and $a \not \in \frac{b^2}{c}(1+\p)$, either of~(1) or~(2) implies~(\ref{s1R2x24}).
   Finally, assuming~(3) we have $v_p(b^2-ac)=2v_p(b)<\min\{v_p(a),v_p(b),v_p(c)\}$ so it only remains to see that
   $v_p(b+cx)<v_p(a+bx)$. By our assumption we have $v_p(a+bx) \ge v_p(b)$, so it suffices to show $v_p(b+cx) \le v_p(b)$. But if that were not the case, we would have 
   $x \in -\frac{b}{c}(1+\p)$ and thus $\frac{b}{c}(1+\p)=\frac{a}{b}(1+\p)$, contradicting the assumption that $a \not \in \frac{b^2}{c}(1+\p)$.
\end{proof}

\begin{lemma}\label{LocalIntegralBorels1}
    Let $\phi \in \H_{B}$ defined by
    $$\phi_{\kk}=
    \begin{cases}
      1 \text{ if } \kk=s_1,\\
        0 \text { if } \kk \in X_{B} \setminus \{s_1\}.
    \end{cases}$$
    For $\nu\in \Lie{a}_\C^*$ with $\Re(\nu)$
     large enough, we have
    \begin{equation*}
        \begin{split}
            J_p(\phi,\nu)=&(1-p^{-1})p^{-2-\nu_1-\nu_2}\\
            &+(1-p^{-1})p^{-2-2\nu_1-\nu_2}\\
            &+(1-p^{-1})p^{-2-2\nu_1}\\
            &-p^{-3-2\nu_1-2\nu_2}.
        \end{split}
    \end{equation*}
\end{lemma}
\begin{proof}
    By definition we have
    \begin{align*}
        J_p(\phi,\nu) &= \int_{U(\Q_p)}
        I_{B,\rho+\nu}(\JJ \uu) \psi(\uu) \1_{\JJ \uu \in B(\Q_p) s_1 B(p)} \, d\uu\\
        &=\int_R \int_{\Q_p} I_{\Pk,1+\nu_1-\nu_2}(\JJ\uu)I_{\Ps,1+\nu_2}(\JJ\uu)\theta(-x-c) \1_{v_p(b+cx) \le v_p(a+bx)} \, dx dY.
    \end{align*}
    Changing variables $\uu \mapsto \tt \uu \tt^{-1}$ and then integrating over $\tt \in A(\Z_p)$, we deduce
     \begin{align*}
         J_p(\phi,\nu) =& \quad  
         \int_{R_1} \int_{\Z_p}I_{\Pk,1+\nu_1-\nu_2}(\JJ\uu)I_{\Ps,1+\nu_2}(\JJ\uu) \1_{v_p(b+cx) \le v_p(a+bx)} \, dx dY\\
         &-p^{-1}\zeta_p(1)  \int_{R_2} \int_{\Z_p} I_{\Pk,1+\nu_1-\nu_2}(\JJ\uu)I_{\Ps,1+\nu_2}(\JJ\uu)  \1_{v_p(b+cx) \le v_p(a+bx)} \, dx dY\\
         &-p^{-1}\zeta_p(1)  \int_{R_1} \int_{p^{-1}\Z_p^\times} I_{\Pk,1+\nu_1-\nu_2}(\JJ\uu)I_{\Ps,1+\nu_2}(\JJ\uu) \1_{v_p(b+cx) \le v_p(a+bx)} \, dx dY\\
         &+p^{-2}\zeta_p(1)^2 \int_{R_2} \int_{p^{-1}\Z_p^\times} I_{\Pk,1+\nu_1-\nu_2}(\JJ\uu)I_{\Ps,1+\nu_2}(\JJ\uu)  \1_{v_p(b+cx) \le v_p(a+bx)} \, dx dY.
     \end{align*}
     Denote by $J_i$ (for $1 \le i \le 4$) the four integrals in
     previous display, respectively.
     As in the proofs of Lemmas~\ref{KLocalIntegrals1} and~\ref{LocalIntegral1} we have
     $$ I_{\Pk,1+\nu_1-\nu_2}(\JJ\uu)=|b+cx|^{-1-\nu_1+\nu_2}$$ and $$I_{\Ps,1+\nu_2}(\JJ\uu)=|ac-b^2|^{-1-\nu_2}.$$
     Using Lemmas~\ref{B.s1.1} to~\ref{B.s1.4} and evaluating the integrals we obtain the result.
     \end{proof}

\subsubsection{The local integral for $s=s_2$}
We continue to follow the strategy described in \S~\ref{strategyBorel}.
Namely let $R$ be as in~\S~\ref{Siegels2} 
and let $\uu=\left[\begin{smallmatrix}
 1& x  & a+bx  &  b+cx  \\
  &  1 & b & c \\
  &   & 1 & \\
  &  &  -x &1
\end{smallmatrix}\right]$.
Suppose we know $Y:=\mat{a}{b}{b}{c} \in R$.
Then either $\JJ \uu \in B(\Q_p)s_2 B(p)$ or
$\JJ \uu \in B(\Q_p) s_1s_2 B(p)$.
To determine which of these two possibilities holds, we need to determine whether or not 
$v_p(a+bx) < v_p(x)$.
Let $R_1=R\cap\{v_p(c) \ge 0\}$ and $R_2=R\cap\{v_p(c)=-1\}$.

\begin{lemma}\label{B.s2.1}
    Let $\uu \in U(\Q_p)$. Then we have 
    \begin{equation}\label{s2R1}
        Y \in R_1, \quad v_p(x) \ge 0 \text{ and }v_p(a+bx)< v_p(x)
    \end{equation}
    if and only if $v_p(x),v_p(c) \ge 0$ and 
    $v_p(a) \le \min\{2v_p(b),-1\}$.
\end{lemma}
\begin{proof}
    Assume $Y \in R_1$.
    This implies $v_p(a) \le v_p(b^2-ac)$.
    Since $v_p(c) \ge 0$ we must have $v_p(a) \le v_p(b^2)$.
    Furthermore by definition of $R_1$ we also have $v_p(a) \le -1$ and $v_p(c) \ge 0$.
    Conversely, if $v_p(c) \ge 0$ and  $v_p(a) \le \min\{2v_p(b),-1\}$ then it is clear that $Y \in R_1$. 
    Furthermore we have $v_p(b) \ge \frac12 v_p(a)>v_p(a)$
    and thus $v_p(a+bx)=v_p(a)<0 \le v_p(x)$
    as desired.
\end{proof}

\begin{lemma}\label{B.s2.2}
    Let $\uu \in U(\Q_p)$. Then we have 
    \begin{equation}\label{s2R2}
        Y \in R_2, \quad v_p(x) \ge 0 \text{ and }v_p(a+bx)< v_p(x)
    \end{equation}
    if and only if $v_p(c) =-1$, $v_p(x) \ge 0$, $v_p(b) \le -2$ and $a \in \frac{b^2}{c}(1+\p)$.
\end{lemma}
\begin{proof}
    Assume $Y \in R_2$.
    This implies $v_p(a) \le v_p(b^2-ac)$.
    Since $v_p(c) =-1$ we must have $ac \in b^2(1+\p)$.
    Thus $v_p(a)=2v_p(b)+1$.
    But by definition of $R_2$ we must have $v_p(a)\le v_p(b)-1$, thus we obtain $v_p(b) \le -2$.
    Conversely assuming  $a \in \frac{b^2}{c}(1+\p)$,
    $v_p(b) \le -2$ and $v_p(c) =-1$ we have $Y \in R_2$.
    Moreover since $v_p(a)<v_p(b)$, if $v_p(x) \ge 0$ then
    $v_p(a+bx)=v_p(a)<0\le v_p(x)$ as desired.
\end{proof}

\begin{lemma}\label{B.s2.3}
    Let $\uu \in U(\Q_p)$. Then we have 
    \begin{equation}\label{s2R1x}
        Y \in R_1, \quad v_p(x) =-1 \text{ and }v_p(a+bx)< v_p(x)
    \end{equation}
    if and only if 
    \begin{equation}\label{s2R1x<->}
        v_p(x)=-1 ,v_p(c) \ge 0,
    v_p(a) \le \min\{2v_p(b),-2\}
     \text{ and } a \not \in -bx(1+\p).
    \end{equation}
\end{lemma}
\begin{proof}
   The same argument as in the proof of Lemma~\ref{B.s2.1} shows that we must have 
   $v_p(a) \le \min\{2v_p(b),-1\}$.
   But if $v_p(a)=-1$ then $v_p(b) \ge 0$ and so
   $v_p(a+bx) \ge -1$, contradicting~(\ref{s2R1x}).
   Thus $v_p(a) \le -2$.
   Finally if $v_p(b)=-1$ then the condition $v_p(a+bx)<v_p(x)$ implies
   $a \not \in -bx(1+\p)$ (note that this condition is automatic in the remaining cases).
   Conversely if~(\ref{s2R1x<->}) holds then if $v_p(a)=-2$ and $v_p(b)=-1$ then the condition
   $ a \not \in -bx(1+\p)$ ensures that 
   $v_p(a+bx)<v_p(x)$.
   In every other case we have 
   $v_p(b)>v_p(a)+1$
    and thus $v_p(a+bx)=v_p(a)<-1 = v_p(x)$
    as desired.
    The remaining of the argument is as in the proof of Lemma~\ref{B.s2.1}.
\end{proof}

\begin{lemma}\label{B.s2.4}
    Let $\uu \in U(\Q_p)$. Then we have 
    \begin{equation}\label{s2R2x}
        Y \in R_2, \quad v_p(x) =-1 \text{ and }v_p(a+bx)< v_p(x)
    \end{equation}
    if and only if 
        $v_p(c) =-1$, $v_p(x) =-1$, $v_p(b) \le -2$, $a \in \frac{b^2}{c}(1+\p)$
        and $a \not \in -bx(1+p^2\Z_p)$.
\end{lemma}
\begin{proof}
    The proof is exactly as in Lemma~\ref{B.s2.2}
    except that $v_p(a+bx)<v_p(x)$ holds if and only if either $v_p(a) \neq v_p(b)-1$ or 
    $v_p(a)=-3, v_p(b)=-2$ and $a \not \in -bx(1+p^2\Z_p)$.
\end{proof}

\begin{lemma}\label{LocalIntegralBorels2}
    Let $\phi \in \H_{B}$ defined by
    $$\phi_{\kk}=
    \begin{cases}
      1 \text{ if } \kk=s_2,\\
        0 \text { if } \kk \in X_{B} \setminus \{s_2\}.
    \end{cases}$$
    For $\nu\in \Lie{a}_\C^*$ with $\Re(\nu)$
     large enough, we have
    \begin{equation*}
        \begin{split}
            J_p(\phi,\nu)=& \quad (1-p^{-1})p^{-1-\nu_1}\\
            &+(1-p^{-1})p^{-2-2\nu_1}\\
            &+(1-p^{-1})p^{-2-2\nu_1-\nu_2}\\
            &-p^{-3-3\nu_1}.
        \end{split}
    \end{equation*}
\end{lemma}
\begin{proof}
    By definition we have
    \begin{align*}
        J_p(\phi,\nu) &= \int_{U(\Q_p)}
        I_{B,\rho+\nu}(\JJ \uu) \psi(\uu) \1_{\JJ \uu \in B(\Q_p) s_2 B(p)} \, d\uu\\
        &=\int_R \int_{\Q_p} I_{\Pk,1+\nu_1-\nu_2}(\JJ\uu)I_{\Ps,1+\nu_2}(\JJ\uu)\theta(-x-c) \1_{v_p(a+bx) < v_p(x)} \, dx dY.
    \end{align*}
    Changing variables $\uu \mapsto \tt \uu \tt^{-1}$ and then integrating over $\tt \in A(\Z_p)$, we deduce
     \begin{align*}
         J_p(\phi,\nu) =& \quad  
         \int_{R_1} \int_{\Z_p} I_{\Pk,1+\nu_1-\nu_2}(\JJ\uu)I_{\Ps,1+\nu_2}(\JJ\uu)  \1_{v_p(a+bx) < v_p(x)}\, dx dY\\
         &-p^{-1}\zeta_p(1)  \int_{R_2} \int_{\Z_p} I_{\Pk,1+\nu_1-\nu_2}(\JJ\uu)I_{\Ps,1+\nu_2}(\JJ\uu) \1_{v_p(a+bx) < v_p(x)} \, dx dY\\
         &-p^{-1}\zeta_p(1)  \int_{R_1} \int_{p^{-1}\Z_p^\times}I_{\Pk,1+\nu_1-\nu_2}(\JJ\uu)I_{\Ps,1+\nu_2}(\JJ\uu) \1_{v_p(a+bx) < v_p(x)} \, dx dY\\
         &+p^{-2}\zeta_p(1)^2 \int_{R_2} \int_{p^{-1}\Z_p^\times} I_{\Pk,1+\nu_1-\nu_2}(\JJ\uu)I_{\Ps,1+\nu_2}(\JJ\uu) \1_{v_p(a+bx) < v_p(x)} \, dx dY.
     \end{align*}
     Denote by $J_i$ (for $1 \le i \le 4$) the four integrals in
     previous display, respectively.
     As in the proofs of Lemmas~\ref{KLocalIntegral1}
     and~\ref{LocalIntegrals1s2s1} we have
     $$ I_{\Pk,1+\nu_1-\nu_2}(\JJ\uu)=|a+bx|^{-1-\nu_1+\nu_2}$$ and $$I_{\Ps,1+\nu_2}(\JJ\uu)=|a|^{-1-\nu_2}.$$
     Using Lemmas~\ref{B.s2.1} to~\ref{B.s2.4} and evaluating the integrals we obtain the result.
%      $$J_1=(1-p^{-1})^2
% \frac{p^{1+\nu_1}}
% {(1-p^{1+\nu_1})(1-p^{1+2\nu_1})}-\frac{1-p^{-1}}{1-p^{1+\nu_1}},$$
%      $$J_2=-(1-p^{-1})^2
%      \frac{p^{-\nu_1}}{1-p^{1+2\nu_1}},$$
%      $$J_3=(1-p^{-1})^3
% \frac{p^{2+\nu_1}}
% {(1-p^{1+\nu_1})(1-p^{1+2\nu_1})}-(1-p^{-1})^2\frac{p^{-\nu_1}}{1-p^{1+\nu_1}}-(1-p^{-1})^2p^{-1-2\nu_1},$$
% $$J_4=(1-p^{-1})^3 p^{-1-2\nu_1-\nu_2}
% +(1-p^{-1})^3 p^{-1-3\nu_1}
% -(1-p^{-1})^3
% \frac{p^{-3\nu_1}}{1-p^{1+2\nu_1}}.$$
%      Simplifying gives the result.
     \end{proof}
     
\subsubsection{The local integral for $s=s_1s_2$}
Let $R$ be as in~\S~\ref{Siegels2} and let $\uu=\left[\begin{smallmatrix}
 1& x  & a+bx  &  b+cx  \\
  &  1 & b & c \\
  &   & 1 & \\
  &  &  -x &1
\end{smallmatrix}\right]$.
Then 
$\JJ \uu \in B(\Q_p) s_1s_2 B(p)$ if and only if 
 $Y:=\mat{a}{b}{b}{c} \in R$ and
$v_p(x) \le v_p(a+bx)$.
Let $R_1=R\cap\{v_p(c) \ge 0\}$ and $R_2=R\cap\{v_p(c)=-1\}$.

\begin{lemma}\label{B.s1s2.1}
    Let $\uu \in U(\Q_p)$. Assume that 
        $Y \in R_1$ and  $v_p(x) \ge 0$.
        Then $v_p(a+bx) < v_p(x)$.
\end{lemma}
\begin{proof}
   This was proved in Lemma~\ref{B.s2.1}.
\end{proof}

\begin{lemma}\label{B.s1s2.2}
    Let $\uu \in U(\Q_p)$. Assume that
        $Y \in R_2$ and  $v_p(x) \ge 0$.
        Then $v_p(a+bx)< v_p(x)$.
\end{lemma}
\begin{proof}
   This was proved in Lemma~\ref{B.s2.2}.
\end{proof}

\begin{lemma}\label{B.s1s2.3}
    Let $\uu \in U(\Q_p)$. Then we have 
    \begin{equation}\label{s1s2R1x}
        Y \in R_1, \quad v_p(x) =-1 \text{ and }v_p(x) \le v_p(a+bx)< v_p(x)
    \end{equation}
    if and only if $v_p(x)=-1$,$v_p(c) \ge 0$
    and one of the following holds
    \begin{enumerate}
        \item $v_p(a)=-1<v_p(b)$,
        \item $v_p(b)=-1$ and $a \in -bx(1+\p)$.
    \end{enumerate}
\end{lemma}
\begin{proof}
   The proof follows the lines of Lemma~\ref{B.s2.3} except that we want $v_p(x) \le v_p(a+bx)$ instead 
   of $v_p(a+bx)<v_p(x).$
\end{proof}

\begin{lemma}\label{B.s1s2.4}
    Let $\uu \in U(\Q_p)$. Then we have 
    \begin{equation}\label{s1s2R2x}
        Y \in R_2, \quad v_p(x) =-1 \text{ and }v_p(x) \le v_p(a+bx)
    \end{equation}
    if and only if 
        $v_p(c) =-1$, $v_p(x) =-1$, $v_p(b) = -2$ and $a \in \frac{b^2}{c}(1+\p) \cap  -bx(1+p^2\Z_p)$.
\end{lemma}
\begin{proof}
    The proof is exactly as in Lemma~\ref{B.s2.2}
    except that $v_p(x) \le v_p(a+bx)$ holds if and only if 
    $v_p(b)=-2$ and $a \in -bx(1+p^2\Z_p)$.
\end{proof}

\begin{lemma}\label{LocalIntegralBorels1s2}
    Let $\phi \in \H_{B}$ defined by
    $$\phi_{\kk}=
    \begin{cases}
      1 \text{ if } \kk=s_1s_2,\\
        0 \text { if } \kk \in X_{B} \setminus \{s_1s_2\}.
    \end{cases}$$
    For $\nu\in \Lie{a}_\C^*$ with $\Re(\nu)$
     large enough, we have
    \begin{equation*}
        \begin{split}
            J_p(\phi,\nu)&=p^{-2-\nu_1-2\nu_2}
            -(1-p^{-1})p^{-1-\nu_1}
            -(1-p^{-1})p^{-1-\nu_1-\nu_2}.
        \end{split}
    \end{equation*}
\end{lemma}
\begin{proof}
    By definition we have
    \begin{align*}
        J_p(\phi,\nu) &= \int_{U(\Q_p)}
        I_{B,\rho+\nu}(\JJ \uu) \psi(\uu) \1_{\JJ \uu \in B(\Q_p) s_1s_2 B(p)} \, d\uu\\
        &=\int_R \int_{\Q_p} I_{\Pk,1+\nu_1-\nu_2}(\JJ\uu)I_{\Ps,1+\nu_2}(\JJ\uu)\theta(-x-c) \1_{v_p(x) \le v_p(a+bx)} \, dx dY.
    \end{align*}
    Changing variables $\uu \mapsto \tt \uu \tt^{-1}$ and integrating over $\tt \in A(\Z_p)$ then using Lemmas~\ref{B.s1s2.1} and~\ref{B.s1s2.2} we deduce
     \begin{align*}
         J_p(\phi,\nu) =
         &-p^{-1}\zeta_p(1)  \int_{R_1} \int_{p^{-1}\Z_p^\times} I_{\Pk,1+\nu_1-\nu_2}(\JJ\uu)I_{\Ps,1+\nu_2}(\JJ\uu) \1_{v_p(x) \le v_p(a+bx)} \, dx dY\\
         &+p^{-2}\zeta_p(1)^2 \int_{R_2} \int_{p^{-1}\Z_p^\times} I_{\Pk,1+\nu_1-\nu_2}(\JJ\uu)I_{\Ps,1+\nu_2}(\JJ\uu) \1_{v_p(x) \le v_p(a+bx)} \, dx dY.
     \end{align*}
     Denote by $J_1, J_2$ the two integrals in
     previous display, respectively.
     As in the proofs of Lemmas~\ref{KLocalIntegrals1s2}
     and~\ref{LocalIntegrals1s2s1} we have
     $$ I_{\Pk,1+\nu_1-\nu_2}(\JJ\uu)=|x|^{-1-\nu_1+\nu_2}=p^{-1-\nu_1+\nu_2}$$ and $$I_{\Ps,1+\nu_2}(\JJ\uu)=|a|^{-1-\nu_2}.$$
     Using Lemmas~\ref{B.s1s2.3} and~\ref{B.s1s2.4} and evaluating the integrals we obtain
$$J_1=(1-p^{-1})^2p^{-\nu_1}
+(1-p^{-1})^2p^{-\nu_1-\nu_2}$$
and 
$$J_2=(1-p^{-1})^2p^{-\nu_1-2\nu_2},$$
hence the result.
     \end{proof}

\subsubsection{The local integral for $s=s_2s_1$}
We continue to follow the strategy described in \S~\ref{strategyBorel}.
Namely let $R$ be as in~\S~\ref{Siegels2s1} and let $\uu=\left[\begin{smallmatrix}
 1& x  & a+bx  &  b+cx  \\
  &  1 & b & c \\
  &   & 1 & \\
  &  &  -x &1
\end{smallmatrix}\right]$.
Suppose we know $Y:=\mat{a}{b}{b}{c} \in R$.
Then either $\JJ \uu \in B(\Q_p)s_2s_1 B(p)$ or
$\JJ \uu \in B(\Q_p) s_1s_2s_1 B(p)$.
In the first case, we have $ \JJ\left[\begin{smallmatrix}
 1& x  & a+bx  &  b+cx  \\
  &  1 & b+cx &  \\
  &   & 1 & \\
  &  & -x  &1
\end{smallmatrix}\right] \in \Pk s_1 B(p),$
whereas in the second case  $\JJ\left[\begin{smallmatrix}
 1& x  & a+bx  &  b+cx  \\
  &  1 & b+cx &  \\
  &   & 1 & \\
  &  & -x  &1
\end{smallmatrix}\right] \in \Pk s_1s_2s_1 B(p).$
Thus to determine which of these two possibilities holds, we need to determine whether or not 
$v_p(b+cx)<0$. 
Let $R_2=R\cap\{v_p(c)=-1\}$.

\begin{lemma}\label{B.s2s1.1}
    Let $\uu \in U(\Q_p)$. Then we have 
    \begin{equation}\label{s2s1R2}
        Y \in R_2, \quad v_p(x) \ge 0 \text{ and }v_p(b+cx) <0
    \end{equation}
    if and only if $v_p(c)=-1$ and one of the following holds
    \begin{enumerate}
    \item $v_p(a),v_p(b) \ge 0$ and $v_p(x)=0$, 
    \item $v_p(b)=-1$, $a \in \frac{b^2}{c}(1+\p)$ and $x \in \Z_p \setminus -\frac{b}{c}(1+\p)$.
    \end{enumerate}
\end{lemma}
\begin{proof}
    Assume $Y \in R_2$.
    By definition we have $-1=v_p(c) \le v_p(b), v_p(ac-b^2)$.
    It is easy to see this holds if and only if either $v_p(a), v_p(b) \ge 0$ or $v_p(b)=-1$ and $a \in \frac{b^2}c(1+\p)$.
    If $v_p(b) \ge 0$ then we have $v_p(b+cx)<0$ if and only if $v_p(x)=0$,
    while if $v_p(b)=-1$ then $v_p(b+cx)<0$
    if and only if $x \not \in -\frac{b}c(1+\p)$.
\end{proof}

\begin{lemma}\label{B.s2s1.2}
    Let $\uu \in U(\Q_p)$. Then we have 
    \begin{equation}\label{s2s1R2x}
        Y \in R_2, \quad v_p(x) =-1 \text{ and }v_p(b+cx) <0
    \end{equation}
    if and only if $v_p(c)=v_p(x)=-1$ and one of the following holds
    \begin{enumerate}
    \item $v_p(a),v_p(b) \ge 0$,
    \item $v_p(b)=-1$ and $a \in \frac{b^2}{c}(1+\p)$.
    \end{enumerate}
\end{lemma}
\begin{proof}
    The beginning of the proof is as in Lemma~\ref{B.s2s1.1}. Now since $v_p(cx)=-2$ and $v_p(b) \ge -1$, the condition $v_p(b+cx)<0$ is automatic.
\end{proof}

\begin{lemma}\label{LocalIntegralBorels2s1}
    Let $\phi \in \H_{B}$ defined by
    $$\phi_{\kk}=
    \begin{cases}
      1 \text{ if } \kk=s_2s_1,\\
        0 \text { if } \kk \in X_{B} \setminus \{s_2s_1\}.
    \end{cases}$$
    For $\nu\in \Lie{a}_\C^*$ with $\Re(\nu)$
     large enough, we have
    \begin{equation*}
        \begin{split}
            J_p(\phi,\nu)&=p^{-2-2\nu_1+\nu_2}
            -(1-p^{-1})p^{-1-\nu_1}
        \end{split}
    \end{equation*}
\end{lemma}
\begin{proof}
    By definition we have
    \begin{align*}
        J_p(\phi,\nu) &= \int_{U(\Q_p)}
        I_{B,\rho+\nu}(\JJ \uu) \psi(\uu) \1_{\JJ \uu \in B(\Q_p) s_2s_1 B(p)} \, d\uu\\
        &=\int_R \int_{\Q_p} I_{\Pk,1+\nu_1-\nu_2}(\JJ\uu)I_{\Ps,1+\nu_2}(\JJ\uu)\theta(-x-c) \1_{v_p(b+cx) < 0} \, dx dY.
    \end{align*}
    Observe that by definition of $R$ we have $R\cap\{v_p(c) \ge 0\}=\varnothing$.
    Thus changing variables $\uu \mapsto \tt \uu \tt^{-1}$ and integrating over $\tt \in A(\Z_p)$ we deduce
     \begin{align*}
         J_p(\phi,\nu) =
         &-p^{-1}\zeta_p(1)  \int_{R_2} \int_{\Z_p} I_{\Pk,1+\nu_1-\nu_2}(\JJ\uu)I_{\Ps,1+\nu_2}(\JJ\uu) \1_{v_p(b+cx) < 0} \, dx dY\\
         &+p^{-2}\zeta_p(1)^2 \int_{R_2} \int_{p^{-1}\Z_p^\times} I_{\Pk,1+\nu_1-\nu_2}(\JJ\uu)I_{\Ps,1+\nu_2}(\JJ\uu) \1_{v_p(b+cx) < 0} \, dx dY.
     \end{align*}
     Denote by $J_1, J_2$ the two integrals in
     previous display, respectively.
     As in the proofs of Lemmas~\ref{KLocalIntegrals1}
     and~\ref{LocalIntegrals2s1} we have
     $$ I_{\Pk,1+\nu_1-\nu_2}(\JJ\uu)=|b+cx|^{-1-\nu_1+\nu_2}$$ and $$I_{\Ps,1+\nu_2}(\JJ\uu)=|c|^{-1-\nu_2}=p^{-1-\nu_2}.$$
     Using Lemmas~\ref{B.s2s1.1} and~\ref{B.s2s1.2} and evaluating the integrals we obtain
$$J_1=(1-p^{-1})^2 p^{-\nu_1}$$
and 
$$J_2=(1-p^{-1})^2p^{-2\nu_1+\nu_2},$$
hence the result.
     \end{proof}

\subsubsection{The local integral for $s=s_1s_2s_1$}
Let $R$ be as in~\S~\ref{Siegels2s1} and let $\uu=\left[\begin{smallmatrix}
 1& x  & a+bx  &  b+cx  \\
  &  1 & b & c \\
  &   & 1 & \\
  &  &  -x &1
\end{smallmatrix}\right]$.
Then 
$\JJ \uu \in B(\Q_p) s_1s_2s_1 B(p)$ if and only if 
 $Y:=\mat{a}{b}{b}{c} \in R$ and
$v_p(b+cx) \ge 0$.
Let $R_2=R\cap\{v_p(c)=-1\}$.

\begin{lemma}\label{B.s1s2s1.1}
    Let $\uu \in U(\Q_p)$. Then we have 
    \begin{equation}\label{s1s2s1R2}
        Y \in R_2, \quad v_p(x) \ge 0 \text{ and }v_p(b+cx) \ge 0
    \end{equation}
    if and only if $v_p(c)=-1$ and one of the following holds
    \begin{enumerate}
    \item $v_p(a),v_p(b) \ge 0$ and $v_p(x)>0$, 
    \item $v_p(b)=-1$, $a \in \frac{b^2}{c}(1+\p)$ and $x \in -\frac{b}{c}(1+\p)$.
    \end{enumerate}
\end{lemma}
\begin{proof}
   The proof follows the lines of Lemma~\ref{B.s2s1.1} except we 
   want $v_p(b+cx) \ge 0$ instead of $v_p(b+cx)<0$.
\end{proof}

\begin{lemma}\label{B.s1s2s1.2}
    Let $\uu \in U(\Q_p)$. Assume
        $Y \in R_2$ and  $v_p(x) = -1$.
        Then $v_p(b+cx) <0.$
\end{lemma}
\begin{proof}
    This follows from (the proof of) Lemma~\ref{B.s2s1.2}.
\end{proof}

\begin{lemma}\label{LocalIntegralBorels1s2s1}
    Let $\phi \in \H_{B}$ defined by
    $$\phi_{\kk}=
    \begin{cases}
      1 \text{ if } \kk=s_1s_2s_1,\\
        0 \text { if } \kk \in X_{B} \setminus \{s_1s_2s_1\}.
    \end{cases}$$
    For $\nu\in \Lie{a}_\C^*$ with $\Re(\nu)$
     large enough, we have
    \begin{equation*}
        \begin{split}
            J_p(\phi,\nu)&=-p^{-1-\nu_2}.
        \end{split}
    \end{equation*}
\end{lemma}
\begin{proof}
    By definition we have
    \begin{align*}
        J_p(\phi,\nu) &= \int_{U(\Q_p)}
        I_{B,\rho+\nu}(\JJ \uu) \psi(\uu) \1_{\JJ \uu \in B(\Q_p) s_1s_2s_1 B(p)} \, d\uu\\
        &=\int_R \int_{\Q_p} I_{\Pk,1+\nu_1-\nu_2}(\JJ\uu)I_{\Ps,1+\nu_2}(\JJ\uu)\theta(-x-c) \1_{v_p(b+cx) \ge 0} \, dx dY.
    \end{align*}
    Observe that by definition of $R$ we have $R\cap\{v_p(c) \ge 0\}=\varnothing$, and that by Lemma~\ref{B.s1s2s1.2}
    we have $R_2 \cap \{v_p(x)=-1\}=\varnothing$.
    Thus changing variables $\uu \mapsto \tt \uu \tt^{-1}$ and integrating over $\tt \in A(\Z_p)$ we deduce
     \begin{align*}
         J_p(\phi,\nu) =
         &-p^{-1}\zeta_p(1)  \int_{R_2} \int_{\Z_p} I_{\Pk,1+\nu_1-\nu_2}(\JJ\uu)I_{\Ps,1+\nu_2}(\JJ\uu) \1_{v_p(b+cx) < 0} \, dx dY.
     \end{align*}
     As in the proofs of Lemmas~\ref{KLocalIntegrals1s2s1}
     and~\ref{LocalIntegrals2s1} we have
     $$ I_{\Pk,1+\nu_1-\nu_2}(\JJ\uu)=1$$ and $$I_{\Ps,1+\nu_2}(\JJ\uu)=|c|^{-1-\nu_2}=p^{-1-\nu_2}.$$
     Using Lemma~\ref{B.s1s2s1.1} and evaluating the integral we obtain the result.
     \end{proof}

\subsubsection{The local integral for $s=s_2s_1s_2$}
We continue to follow the strategy described in \S~\ref{strategyBorel}.
Namely let $\uu=\left[\begin{smallmatrix}
 1& x  & a+bx  &  b+cx  \\
  &  1 & b & c \\
  &   & 1 & \\
  &  &  -x &1
\end{smallmatrix}\right]$.
Suppose we know $Y:=\mat{a}{b}{b}{c} \in \Mat_2(\Z_p)$.
Then either $\JJ \uu \in B(\Q_p)s_1s_2s_1 B(p)$ or
$\JJ \uu \in B(\Q_p) \JJ B(p)$.
Which of these two possibilities occurs is determined 
by whether $v_p(x)<0$ or not. 

\begin{lemma}\label{LocalIntegralBorels2s1s2}
    Let $\phi \in \H_{B}$ defined by
    $$\phi_{\kk}=
    \begin{cases}
      1 \text{ if } \kk=s_2s_1s_2,\\
        0 \text { if } \kk \in X_{B} \setminus \{s_2s_1s_2\}.
    \end{cases}$$
    For $\nu\in \Lie{a}_\C^*$ with $\Re(\nu)$
     large enough, we have
    \begin{equation*}
        \begin{split}
            J_p(\phi,\nu)&=-p^{-1-\nu_1+\nu_2}
        \end{split}
    \end{equation*}
\end{lemma}
\begin{proof}
    By definition we have
    \begin{align*}
        J_p(\phi,\nu) &= \int_{U(\Q_p)}
        I_{B,\rho+\nu}(\JJ \uu) \psi(\uu) \1_{\JJ \uu \in B(\Q_p) s_2s_1s_2 B(p)} \, d\uu\\
        &=\int_{\Mat_2(\Z_p)} \int_{\Q_p} I_{\Pk,1+\nu_1-\nu_2}(\JJ\uu)I_{\Ps,1+\nu_2}(\JJ\uu)\theta(-x-c) \1_{v_p(x) < 0} \, dx dY.
    \end{align*}
    Thus changing variables $\uu \mapsto \tt \uu \tt^{-1}$ and integrating over $\tt \in A(\Z_p)$ we deduce
     \begin{align*}
         J_p(\phi,\nu) =
         &-p^{-1}\zeta_p(1)  \int_{\Mat_2(\Z_p)} \int_{p^{-1}\Z_p^\times} I_{\Pk,1+\nu_1-\nu_2}(\JJ\uu)I_{\Ps,1+\nu_2}(\JJ\uu) \, dx dY.
     \end{align*}
     As in the proofs of Lemmas~\ref{KLocalIntegrals1s2}
     and~\ref{LocalIntegrals2s1s2} we have
     $$ I_{\Pk,1+\nu_1-\nu_2}(\JJ\uu)=|x|^{-1-\nu_1+\nu_2}=p^{-1-\nu_1+\nu_2}$$ and $$I_{\Ps,1+\nu_2}(\JJ\uu)=1.$$
     The result immediately follows.
     \end{proof}

\subsubsection{The local integral for $s=\JJ$}
\begin{lemma}\label{LocalIntegralBorelJ}
    Let $\phi \in \H_{B}$ defined by
    $$\phi_{\kk}=
    \begin{cases}
      1 \text{ if } \kk=\JJ,\\
        0 \text { if } \kk \in X_{B} \setminus \{\JJ\}.
    \end{cases}$$
    For $\nu\in \Lie{a}_\C^*$ with $\Re(\nu)$
     large enough, we have
    \begin{equation*}
        \begin{split}
            J_p(\phi,\nu)&=1
        \end{split}
    \end{equation*}
\end{lemma}
\begin{proof}
It is clear at this stage that for $\uu=\left[\begin{smallmatrix}
 1& x  & a+bx  &  b+cx  \\
  &  1 & b & c \\
  &   & 1 & \\
  &  &  -x &1
\end{smallmatrix}\right] \in U(\Q_p)$ we have 
$\JJ\uu \in B(\Q_p) \JJ B(p)$ if and only if $a,b,c,d,x \in \Z_p$, and that in this case $I_{B,\nu}(\JJ\uu)=1$.
\end{proof}

\begin{remark}
    As a sanity check, $\phi \in \H_{B}$ be the constant function $1$.
    Then combining Lemmas~\ref{LocalIntegralBorel1},~\ref{LocalIntegralBorels1},~\ref{LocalIntegralBorels2},~\ref{LocalIntegralBorels1s2},~\ref{LocalIntegralBorels2s1},~\ref{LocalIntegralBorels1s2s1},~\ref{LocalIntegralBorels2s1s2} and~\ref{LocalIntegralBorelJ} we have $$ J_p(\phi,\nu)=\zeta_p^{-1}(1+\nu_1)
    \zeta_p(1+\nu_2)^{-1}\zeta_p(1+\nu_1+\nu_2)^{-1}
    \zeta_p(1+\nu_1-\nu_2)^{-1},$$ which is the correct local factor (compare with~\cite{Shahidi}*{Theorem~7.1.2}).
\end{remark}

\end{document}